\documentclass[reqno]{amsart}      
\usepackage{amsthm}  
\usepackage{graphicx}
\usepackage{mathptmx}      
\usepackage{amsfonts}
\usepackage{amssymb}
\usepackage{multirow}
\usepackage{bm}
\usepackage{xcolor}
\usepackage{comment}
\usepackage{arydshln}
\usepackage{algorithm}
\usepackage{algpseudocode}

\usepackage[hidelinks]{hyperref}
\usepackage[justification=centering]{caption}
\usepackage{subcaption}
\usepackage{epstopdf}
\graphicspath{ {./images/} }
\usepackage{float}
\usepackage{csvsimple, booktabs, datatool}
\usepackage{siunitx}
\usepackage{tabularx}
\newcolumntype{P}[1]{>{\centering\arraybackslash}m{#1}}
\usepackage{adjustbox}
\usepackage{tikz}
\makeatletter
\let\cl@chapter\undefined
\makeatother
\usepackage{cleveref}
\makeatletter
\providecommand\theHALG@line{\thealgorithm.\arabic{ALG@line}}
\makeatother
\usepackage{ulem}
\usepackage{array}

\newtheorem{definition}{Definition}

\newtheorem{theorem}{Theorem}
\newtheorem{lemma}{Lemma}
\newtheorem{proposition}{Proposition}

\newtheorem{remark}{Remark}

\makeatletter
\newcommand{\Rmnum}[1]{\expandafter\@slowromancap\romannumeral #1@}
\makeatother
\usepackage{caption}
\def\<#1,#2>{\langle #1,#2 \rangle}

\DeclareMathOperator{\rank}{rank}

\DeclareMathOperator{\argmin}{argmin}

\newcommand{\cS}{\mathcal{S}}
\newcommand{\cA}{\mathcal{A}}

\newcommand{\cF}{\mathcal{F}}
\newcommand{\cH}{\mathcal{H}}

\newcommand{\R}{\mathbb{R}}
\newcommand{\N}{\mathbb{N}}
\newcommand{\bI}{\mathbf{I}}
\newcommand{\bzero}{\mathbf{0}}
\newcommand{\bone}{\mathbf{1}}

\newcommand{\algorithmicbreak}{\textbf{break}}

\algnewcommand{\IIf}[1]{\State\algorithmicif\ #1\ \algorithmicthen}
\algnewcommand{\EndIIf}{\unskip\ \algorithmicend\ \algorithmicif}


\usepackage[colorinlistoftodos,prependcaption,textsize=tiny]{todonotes}

\tikzstyle{nodes1} = [circle, rounded corners, minimum width=1cm, minimum height=1cm,text centered, draw=black, fill=red!30]
\tikzstyle{arrow} = [thick,->,>=stealth]

\allowdisplaybreaks[4]

\begin{document}

\title{Progressive Bound Strengthening via Doubly Nonnegative Cutting Planes for Nonconvex Quadratic Programs}

\author{Zheng Qu}
\address{School of Mathematical Sciences, Shenzhen University, 518061, Shenzhen, P. R. China}
\email{zhengqu@szu.edu.cn}

\author{Defeng Sun}
\address{Department of Applied Mathematics, The Hong Kong Polytechnic University, Hong Kong}
\email{defeng.sun@polyu.edu.hk}

\author{Jintao Xu}
\address{Department of Applied Mathematics, The Hong Kong Polytechnic University, Hong Kong}
\email{xujtmath@163.com}


\makeatletter
\renewcommand{\keywordsname}{Keywords}
\makeatother


\makeatletter
\renewcommand{\subjclassname}{%
  \textup{2020} Mathematics Subject Classification}
\makeatother
\date{}
\thanks{Defeng Sun was supported in part by the Research Center for Intelligent Operations Research and RGC Senior Research Fellow Scheme No. SRFS2223-5S02.  
Jintao Xu was partially supported  by the PolyU postdoc matching fund scheme of The Hong Kong Polytechnic University Grant No. 1-W35A, and Huawei's
Collaborative Grants ``Large scale linear programming solver'' and ``Solving
large scale linear programming models for production planning''. }

\maketitle
\begin{abstract}

We introduce a cutting-plane framework for nonconvex quadratic programs (QPs) that progressively tightens convex relaxations. Our approach leverages the doubly nonnegative (DNN) relaxation to compute strong lower bounds and generate separating cuts, which are iteratively added to improve the relaxation. We establish that, at any Karush–Kuhn–Tucker (KKT) point satisfying a second-order sufficient condition, a valid cut can be obtained by solving a linear semidefinite program (SDP), and we devise a finite-termination local search procedure to identify such points. Extensive computational experiments on both benchmark and synthetic instances demonstrate that our approach yields tighter bounds and consistently outperforms leading commercial and academic solvers in terms of efficiency, robustness, and scalability. Notably, on a standard desktop, our algorithm reduces the relative optimality gap to  $10^{-4}$ on 138 out of 140 instances of dimension  $n=100$ within one hour, without resorting to branch-and-bound.

\end{abstract}

\section{Introduction}

We consider the following nonconvex quadratic programming (QP) problem:
\begin{equation}\label{eq:QP}
\begin{aligned}
\min_{x \in \mathbb{R}^n} \quad & x^\top Q x + 2 d^\top x \\
\text{s.t.} \quad & A x \leq b.
\end{aligned}
\end{equation}
Hereinafter, $Q$ is a symmetric matrix with at least one negative eigenvalue, and the feasible region $\{x \in \mathbb{R}^n : A x \leq b\}$ is bounded. Note that this implies $\operatorname{rank}(A) = n$.

The QP problem~\eqref{eq:QP} arises naturally in diverse fields such as quadratic knapsack problems~\cite{Vavasis1992,LiDengLuWuDaiWang2023}, game theory~\cite{LiuzziLocatelliPiccialliRass2021}, and maximum clique problems~\cite{MotzkinStraus1965}. 
Considerable progress has been made in developing efficient algorithms to find high-quality local optima~\cite{Vavasis1992a,EdirisingheJeong2017,ZhangHanPang2024}.  Finding global solutions to nonconvex QP problems remains challenging, as these problems are NP-hard~\cite{pardalos1991quadratic}. Nevertheless, the special structure of QP problems enables the design of various global solvers, with branch-and-bound methods being among the most notable approaches~\cite{burer2008finite,chen2012globally,LiuzziLocatelliPiccialli2022}.


QP problems can be reformulated in different ways, including bilinear reformulation~\cite{LiuzziLocatelliPiccialliRass2021}, spectral reformulation~\cite{nohra2021}, copositive reformulation~\cite{Burer2012} and KKT reformulation~\cite{burer2008finite}. 
Different reformulations lead to diverse convex relaxations (linear relaxations~\cite{SheraliAdams1999,TawarmalaniSahinidis2004,LiuzziLocatelliPiccialli2022}, convex quadratic relaxations~\cite{PardalosGlickRosen1987,horst2013handbook,nohra2021} and semidefinite relaxations~\cite{Shor1987,Nowak1999,SheraliFraticelli2002}), as well as different branching strategies aligned with the structure of the reformulated problem~\cite{burer2008finite,nohra2021,LiuzziLocatelliPiccialli2022}. The effectiveness of the branching process depends largely on the quality of the relaxation bounds employed. 
 Consequently, designing advanced techniques for tightening the relaxation bounds is essential for enhancing the performance of all global algorithms \cite{VandenbusscheNemhauser2005a,Bonami2019,LiuzziLocatelliPiccialli2022,LocatelliPiccialliSudoso2025}. In this paper, we propose a new method for effectively improving the relaxation bound of~\eqref{eq:QP} by progressively adding cutting planes.

\subsection{A Generic Cutting Plane Framework}\label{subsec:cutpf}
We start by presenting a generic cutting plane framework for solving~\eqref{eq:QP}. Throughout the paper we denote by $\Phi$ the quadratic objective function of~\eqref{eq:QP}:
$$
\Phi(x)\equiv x^\top Q x +2 d^\top x
$$
and by $\cF$ the feasible region of~\eqref{eq:QP}:
$$
\cF:=\{x\in \R^n: Ax\leq b\}.
$$
Let $\mathcal{A} \subset \cF$ and $\bar x\in \mathcal{A}$ be a feasible solution. Let $\nu_R\leq \Phi(\bar x)$ be a reference value. 
  We say that $c\in \R^n$ defines a \textit{valid cut respect to $\bar x$, $\mathcal{A}$ and $\nu_{_R}$} if the following condition holds:
\begin{equation}
\begin{aligned}
    \nu_{_R} \leq \min \{\Phi(x): x\in  \mathcal{A},  c^\top(x-\bar x) \leq 1 \}    
    \end{aligned}.
\end{equation}

\begin{center}
\fbox{
\begin{minipage}{0.7\textwidth}
\begin{center}
\textbf{Generic Cutting Plane Method For Bound Tightening}
\end{center}
Parameters: gap tolerance parameters $\epsilon>0$.

Initialization: $\bar v =+\infty$, ${\mathcal{A}}=\cF$.

Repeat:
\begin{enumerate}
    \item Find a feasible point $\bar x \in \mathcal{A}$.
    \item Update the upper bound $\bar v:=\min (\bar v, \Phi(\bar x))$.
    \item Update the reference value $\nu_R$.
    \item Compute a valid cut $c$ with respect to $\bar x$, $\mathcal{A}$ and $\nu_R$.
    \item Update ${\mathcal{A}}\leftarrow  {\mathcal{A}} \cap \{x\in \R^n: c^\top (x-\bar x)\geq 1\}$.
    \item Compute a lower bound $\underline v^r$ such that $\underline v^r \leq \Phi^*(\mathcal{A})$.
\end{enumerate}
until $\bar v-\underline v^r \leq \epsilon |\bar v| $.
\end{minipage}
}
\end{center}
In the cutting plane method described above, we maintain an upper bound $\bar{v} \geq \Phi^*(\mathcal{F})$ that decreases throughout the iterations. At each iteration, the method solves the optimization problem
\begin{equation}\label{eq:minPhiA}
\min \{ \Phi(x) : x \in \mathcal{A} \},
\end{equation}
where $\mathcal{A} \subset \mathcal{F}$ is a polyhedral set obtained by iteratively adding {valid cuts} such that
$
\Phi^*(\mathcal{F} \setminus \mathcal{A}) \geq \nu_R.
$
Therefore, at each iteration we always have
$
\bar{v} \geq \Phi^*(\mathcal{F}) \geq \min(\underline{v}^r, \nu_R),
$
where $\underline{v}^r$ is a lower bound for $\Phi^*(\mathcal{A})$. The choice of $\nu_R$ affects the quality of the final result. For example, if we update $\nu_R$ as $\bar{v} - \epsilon$ at each iteration, then upon termination we have
$
\bar{v} \geq \Phi^*(\mathcal{F}) \geq \bar{v} - \epsilon \max(1, |\bar{v}|).
$

The feasible point $\bar x\in \cA$ will be excluded by adding the cut $\{x \in \mathbb{R}^n : c^\top (x - \bar{x}) \geq 1\}$. Furthermore,    if $\mathcal{A}$ has a nonempty interior and $\bar{x} \in \mathcal{A}$, then for any vector $c \in \mathbb{R}^n$, the volume of
$
\mathcal{A} \cap \{x \in \mathbb{R}^n : c^\top (x - \bar{x}) \geq 1\}
$
is strictly smaller than that of $\mathcal{A}$. Therefore, the volume of $\mathcal{A}$ is expected to decrease at each iteration of the cutting plane method, and accordingly, we anticipate a reduction in the lower bound $\underline{v}^r$.

The above cutting plane method is only a generic framework, in implementation one has to concretize the following 
 two key steps:
\begin{itemize}
    \item[(a)] Computation of a lower bound $\underline{v}^r$ such that $\underline{v}^r \leq \Phi^*(\mathcal{A})$.
    \item[(b)] Generation of a valid cut $c$ with respect to $\bar{x}$, $\mathcal{A}$, and $\nu_R$.
\end{itemize}

For clarity of exposition, we now discuss these two steps for the case where $\mathcal{A} = \mathcal{F}$.

\begin{itemize}
\item[(a)] \textbf{Lower bound.} There is a wide selection of computationally tractable relaxations of~QP. We refer the reader to~\cite{nohra2021} and the references therein for a review of possible relaxations. In our numerical implementation, we adopt the lower bound obtained from the doubly nonnegative (DNN) relaxation of~\eqref{eq:QP}, which can be represented in the following form:
\begin{equation}
    \label{eq:SDP-lb}
    \begin{aligned}
    \min_{\substack{X,x} } \quad & \left\langle \begin{pmatrix}
Q & d \\
d^\top & 0
\end{pmatrix}, \begin{pmatrix}
 X & x \\
 x^\top  &1
\end{pmatrix} \right \rangle   \\
        \textrm{\rm s.t.}  \quad & 
\begin{pmatrix}
-A & b \\
\bzero^\top & 1 \\
\end{pmatrix}  \begin{pmatrix}
 X & x \\
 x^\top  &1
\end{pmatrix}\begin{pmatrix}
-A & b \\
\bzero^\top & 1 \\
\end{pmatrix} ^\top \geq 0 \\  
\quad &  \begin{pmatrix}
 X & x \\
 x^\top  &1
\end{pmatrix}\succeq 0.
    \end{aligned} 
\end{equation}
The optimal value of~\eqref{eq:SDP-lb} is referred to as the \textit{DNN bound} of the QP problem~\eqref{eq:QP}.
\item [(b)]  \textbf{Valid cut.}  Let $\bar x\in \cF$  and  $\nu_{_R}\leq \Phi(\bar x)$.
To compute a {valid cut respect to $\bar x$, $\cF$ and $\nu_{_R}$}, one search for a vector $c\in\R^n$ such that the following condition holds:
\begin{equation}\label{eq:validcut}
\begin{aligned}
    \nu_{_R} \leq \min_{x\in \R^{n}} \quad &   x^\top Q x + 2 d^\top x   \\
        \textrm{\rm s.t.} \quad &  Ax\leq b\\
        & \displaystyle  c^\top(x-\bar x) \leq 1.
    \end{aligned}
\end{equation}

However, even when $c \in \mathbb{R}^n$ is given, verifying that~\eqref{eq:validcut} holds is computationally intractable, as it essentially requires solving another non-convex QP problem, let alone searching for a vector $c \in \mathbb{R}^n$ that satisfies~\eqref{eq:validcut} in the first place. This difficulty motivates our work, in which we propose a method to generate valid cuts for a class of points $\bar{x}$ by solving linear semidefinite programming (SDP) problems.

\end{itemize}

\subsection{Main Contributions}

The first contribution of this paper, stated in~\Cref{thm:STc}, shows that if $\bar{x}$ is a Karush-Kuhn-Tucker (KKT) point satisfying the second-order condition~\eqref{eq:QHpd}, and either $\bar{x}$ is nondegenerate (as defined in~\Cref{def:nondegenerate}) or $\nu_{R} < {\Phi}(\bar{x})$, then a valid cut $c \in \mathbb{R}^n$ satisfying~\eqref{eq:validcut} can always be obtained by solving the linear SDP~\eqref{eq:SDPsearchofc}.

The second contribution, presented in~\Cref{thm:gmcfinite}, establishes that a KKT point $\bar{x}$ that meets the second order condition~\eqref{eq:QHpd} can be reached by solving a finite number of convex QP problems.

Finally, we implement the generic cutting plane framework described in~\Cref{subsec:cutpf}, which utilizes the DNN relaxation~\eqref{eq:SDP-lb} (to compute lower bounds) and the convex SDP~\eqref{eq:SDPsearchofc} (for cut generation). The resulting algorithm, \texttt{DCQP} (\Cref{alg:CuP-global}), does not rely on the classical branch-and-bound (B\&B) framework. However, it demonstrates superior performance compared to the state-of-the-art commercial solver \texttt{Gurobi} and the recent academic nonconvex QP solver \texttt{QPL}~\cite{LiuzziLocatelliPiccialli2022}, as shown in~\Cref{sec:experiments}. Specifically, on a standard desktop, \texttt{DCQP} successfully reduced the relative gap to $10^{-4}$ for 138 out of 140 instances of dimension $n=100$ within one hour. In contrast, both \texttt{Gurobi} and \texttt{QPL} failed to achieve this level of optimality for more than half of the instances. Moreover, on some selected instances, neither \texttt{Gurobi} nor \texttt{QPL} could reach a relative gap of $10^{-4}$ within 48 hours.

\subsection{Related Work}

In the literature on nonconvex QP, cutting planes are commonly introduced to strengthen relaxations, typically with respect to the relaxed formulation (for instance,~\eqref{eq:SDP-lb} when employing the DNN relaxation). The derivation of such cutting planes often involves analyzing the convex hull of the set $$\left\{(x,X)\in \R^n\times \cS^n:Ax\leq b, \enspace X_{ij}=x_ix_j,\enspace \forall i,j\right\}$$ to obtain valid linear inequalities that are satisfied by the pair $(x,X)$. A classical example of these valid inequalities is based on McCormick's under- and over-estimators~\cite{mccormick1976computability}.   For special cases such as box-constrained QP problems (i.e., when the polytope $\cF$ is the unit box) and standard QP problems (i.e., when $\cF$ is the unit simplex), the convex hull exhibits additional structure, allowing for the derivation of more sophisticated inequalities, such as the triangle inequalities~\cite{VandenbusscheNemhauser2005,VandenbusscheNemhauser2005a,Burer2009a,BonamiOktay18,Bonami2019}. Given an objective cut-off, i.e., an upper bound on the optimal value obtained from the best known solution, it is common to employ the so-called 'optimality-based bound tightening' (OBBT) technique to  further strengthen the relaxation bound~\cite{Coffrin2015,Gleixner17,Puranik17,Sundar2023}.  Note that the OBBT technique introduces cutting planes that are in the specific form of box constraints.

Our approach to generating cutting planes differs fundamentally from the aforementioned methods in two respects. First, we do not impose any a priori structure on the form of the linear inequalities. Second, and more importantly, while traditional cutting planes are typically incorporated into a branch-and-bound (B\&B) framework to guarantee global convergence, our main theorem (\Cref{thm:STc}) implies that the feasible region can be repeatedly reduced before certifying global optimality. As a result, it is possible to solely rely on the cutting plane technique introduced in this paper, as confirmed by our experiments.

The works most closely related to ours are those of Tuy~\cite{tuy1964concave} and Konno~\cite{konno1976cutting}, both of which focus on concave QP problems where $Q$ is required to be negative semidefinite. More recently, Qu et al.~\cite{QuZengLou2025} established a connection between Konno’s relaxation bound and the DNN bound, leading to an efficient cutting-plane-based global solver for concave QP. However, Konno's cut cannot be directly extended to the more general indefinite case, and the cutting-plane technique in~\cite{QuZengLou2025}  relies on the generation of Konno's cut. As a result, these methods are not applicable to indefinite QP problems. Our work, therefore, is not a straightforward extension of previous approaches, but rather addresses the new challenges that arise in the indefinite setting.

\subsection{Organization}
The remainder of this paper is organized as follows. In \Cref{sec:pre}, we introduce the notations, define nondegenerate KKT point, and review the second-order optimality conditions for QP problems. \Cref{sec:validcut} presents the development of valid cuts based on SDP programming. In \Cref{sec:localsearch}, we propose a local search method with finite convergence guarantees. The main algorithm \texttt{DCQP} is described in detail in \Cref{sec:algorithm}. Numerical results are provided in \Cref{sec:experiments}, and we conclude in \Cref{sec:conclusion}. Short proofs are included within the main text, while longer proofs are deferred to the appendix.

\section{Preliminaries}\label{sec:pre}
\subsection{Notations}

Throughout the paper, we define $[n]:=\{1,\ldots,n\}$.
We use $\bzero$, $\bone$ and $\bI$ to denote the zero vector or matrix, the vector of all ones, and the identity matrix, respectively, each with appropriate dimensions.
For a vector $x\in\R^n$, the notation $x\geq 0$ indicates that $x$ is elementwise nonnegative, and
$x\geq y$ means that $x-y\geq 0$.

We denote by $\cS^n$ the set of $n$-by-$n$ symmetric matrices, equipped with the Frobenius inner product $\left\langle A, B \right\rangle := \mathrm{tr}(A^\top B)$ for all $A, B\in\cS^n$. For $X\in\cS^n$,  we write $X\geq 0$ to indicate elementwise nonnegativity, and $X\succeq 0$ (resp. $X\succ 0$) to indicate that $X$ is positive semidefinite (resp. positive definite).
For two matrices $X, Y\in \cS^n$, the relations $X \geq Y$ and $X \succeq Y$ means that $X - Y { \geq 0}$ and $X - Y {\succeq 0}$, respectively. 

  For any nonempty subset $I\subset [m]$,  $A_I$ denotes  the submatrix of $A$ consisting of all rows with indices in $I$ and $b_I$ denotes the subvector of $b$ containing all elements with indices in $I$. Let
\begin{align*}&{\cH}_I:=\left\{x \in \R^n:A_I x=0\right\},\enspace \hat{\cH}_I:=\left\{x \in \R^n:A_I x=b_I\right\},\enspace \tilde{\cH}_I:=\left\{x \in \R^n:A_I x\le 0\right\}.
\end{align*}
For ease of notation, we also let $\cH_{\emptyset}:=\R^n$, $\hat\cH_{\emptyset}:=\R^n$, and $\tilde{\cH}_{\emptyset}:=\R^n$.

Denote by $a_i\in \R^n$ the transpose of the $i$th row vector of $A$. For any $x\in \cF$, let $I_x$ be the index set of the active constraints at point $x$, i.e.:
$$I_{x}:=\{i\in [m]:a_i^\top {x}=b_i\}.$$ 
Then $\cH_{I_x}$ corresponds to the null space of the active constraints at point $x$. 
Given $\lambda\in \R_+^m$, denote by $J_{\lambda}$ the index of positive elements in $\lambda$:
$$
J_{\lambda}:=\{i\in [m]:\lambda_i>0\}.
$$
For a finite set $S$, $|S|$ denotes the cardinality of $S$.

For any cone $\cH \subseteq \mathbb{R}^n$, 
we say that $Q$ is \textit{$\cH$-conditionally positive semidefinite} (resp.\ \textit{$\cH$-conditionally positive definite}) if $p^\top Q p \geq 0$ (resp.\ $p^\top Q p > 0$) for all nonzero $p \in \cH$. We write $Q|_{\cH} \succeq 0$ (resp.\ $Q|_{\cH} \succ 0$) to indicate that $Q$ is $\cH$-conditionally positive semidefinite (resp.\ definite), and $Q|_{\cH} \nsucc 0$ to indicate that $Q$ is not $\cH$-conditionally positive definite. If $\cH = \{\mathbf{0}\}$, we conventionally set $Q|_{\cH} \succ 0$ to be true for any  $Q\in \cS^n$.


\subsection{Nondegenerate KKT Point}
Let $\bar x$ be a feasible point of~\eqref{eq:QP}. It is a KKT point of~\eqref{eq:QP} if there is $\bar \lambda \in \R^m$ such that
\begin{equation}\label{eq:KKT}
\left\{
\begin{aligned}
&\bar\lambda \geq 0 \\
&Q\bar x+d =-A^\top \bar \lambda\\ 
&(A\bar x-b)^\top \bar \lambda=0 .
\end{aligned}
\right.
\end{equation}

\begin{definition}\label{def:nondegenerate}
We say that $\bar x$ is a nondegenerate KKT point of~\eqref{eq:QP} if there exist $\bar \lambda\in\R^m$ such that~\eqref{eq:KKT} holds and
\begin{equation}\label{eq:AIJ}
\rank(A_{I_{\bar x}})=\rank(A_{J_{\bar \lambda}})=|J_{\bar\lambda}|.
\end{equation}
\end{definition}
\begin{remark}
The definition of nondegeneracy presented above is less stringent than the standard definition used in nonlinear programming. Specifically, it does not require the uniqueness of the multiplier $\bar \lambda$. This relaxation is somewhat not surprising due to the particular linear structure of the feasible set. Note that if the standard constraint qualification condition holds, then~\eqref{eq:AIJ} corresponds to the strict complementary slackness condition. 
\end{remark} 

\subsection{Second Order Optimality Conditions}

We review below the necessary and sufficient conditions for the local optimality of~\eqref{eq:QP}; see also~\cite{Cottle1992,Lee2005}.

\begin{definition}[Local and locally unique solution~{\cite{Mangasarian1980}}]
A feasible point $\bar{x} \in \mathcal{F}$ such that
$
\bar{x}^\top Q \bar{x} + 2 d^\top \bar{x} \leq x^\top Q x + 2 d^\top x
$
for all feasible $x$ in some open Euclidean ball around $\bar{x}$ is called a \textit{local solution} of~\eqref{eq:QP}. If equality holds only for $x = \bar{x}$, then $\bar{x}$ is said to be a \textit{locally unique} solution of~\eqref{eq:QP}.
\end{definition}

\begin{theorem}[{\cite{Mangasarian1980,LuisContesse1980}}]\label{thm:soocu}
A point $\bar{x} \in \mathbb{R}^n$ is a local (resp.\ locally unique) solution of~\eqref{eq:QP} if and only if:
\begin{itemize}
    \item[(a)] $\bar{x}$ is feasible and satisfies the KKT conditions~\eqref{eq:KKT} with some $\bar{\lambda} \in \mathbb{R}^m$, and
    \item[(b)] the matrix $Q$ is $\mathcal{H}_{J_{\bar{\lambda}}} \cap \tilde{\mathcal{H}}_{I_{\bar{x}}}$-conditionally positive semidefinite (resp.\ positive definite).
\end{itemize}
\end{theorem}

 For $\bar{\lambda}$ satisfying~\eqref{eq:KKT}, we have
$
\mathcal{H}_{J_{\bar{\lambda}}} \supset \mathcal{H}_{I_{\bar{x}}}$ and $ \tilde{\mathcal{H}}_{I_{\bar{x}}} \supset \mathcal{H}_{I_{\bar{x}}}
$. Therefore, when $\mathcal{H}_{I_{\bar{x}}}\neq \emptyset$, a second-order necessary condition for $\bar{x}$ being a locally unique solution is:
\begin{equation}\label{eq:QHpd}
Q|_{\mathcal{H}_{I_{\bar{x}}}} \succ 0.
\end{equation}
By a slight abuse of language, in the remainder of the paper we refer to any KKT point $\bar{x}$ satisfying~\eqref{eq:QHpd} as a \textit{second-order KKT point}.

\begin{remark}
In general,~\eqref{eq:QHpd} is not a sufficient optimality condition. Here is an example. 
\begin{equation}
    \begin{aligned}
    \min_{x_1,x_2} \quad &   x^2_2+ x_1 x_2-x_2-\frac{1}{2}x_1+\frac{1}{4}  \\
        \textrm{s.t.} \quad &  
         x_1 \geq 0,\enspace x_2\geq 0,\enspace x_1+x_2 \leq 1.
    \end{aligned}
\end{equation}
Clearly $(x_1,x_2)=(0,1/2)$ is a KKT point with only one active constraint $\{x_1=0\}$. The quadratic term of the objective function $x_2^2+x_1x_2$ is clearly positive definite in the null space $\{(x_1,x_2): x_1=0\}$. However, $(x_1,x_2)=(0,1/2)$ is not a local minimum because $(x_1,x_2)=(\epsilon,(1-\epsilon)/2)$ has objective value $-{\epsilon^2}/{4}$.
\end{remark}
\begin{remark}\label{rem:KKT2ndlocal}
If $\bar x$ is a nondegenerate KKT point, then $\cH_{J_{\bar \lambda}}=\cH_{I_{\bar x}}$ and $\mathcal{H}_{J_{\bar{\lambda}}} \cap \tilde{\mathcal{H}}_{I_{\bar{x}}}=\cH_{I_{\bar x}}$. Hence, in this case,~\eqref{eq:QHpd} becomes sufficient and $\bar x$ is a locally unique solution. 
\end{remark}

\section{Construction of Valid Cut}\label{sec:validcut}
Throughout this section,  $\bar{x}$ is a given KKT point of~\eqref{eq:QP}, and  $\nu_R \leq \Phi(\bar{x})$ is a given reference value. We present a method to remove $\bar{x}$ from $\mathcal{F}$ while ensuring that the minimal value of $\Phi$ over the removed region is at least $\nu_R$. 
For brevity, we simply say that $c$ is a valid cut if it is valid with respect to $\bar{x}$, $\mathcal{F}$, and $\nu_R$, i.e., if~\eqref{eq:validcut} holds.

\subsection{Sufficient Conditions}
We start by making the following basic observation. 
\begin{lemma}\label{l:S1S2}
Let any $c\in \R^n$.
If there exist $ S\succeq 0$ and $T\geq 0$ such that
\begin{equation}\label{eq:SOS}
\begin{aligned}
\begin{pmatrix}
Q & d \\
d^\top & -\nu_{_R}
\end{pmatrix}&=S+\begin{pmatrix}
-A & b \\
\bzero^\top & 1 \\
\displaystyle -c^\top & 1+c^\top \bar x 
\end{pmatrix}^\top
T
\begin{pmatrix}
-A & b \\
\bzero^\top & 1 \\
\displaystyle -c^\top & 1+c^\top \bar x
\end{pmatrix},
\end{aligned}
\end{equation}
then $c$ is a valid cut.
\end{lemma}
\begin{proof}
If~\eqref{eq:SOS} holds, then for any $x\in \R^n$  we have
\begin{equation*}
\begin{aligned}
&x^\top Q x+2 d^\top x-\nu_{_R} 
=
\begin{pmatrix}
x^\top & 1
\end{pmatrix}S\begin{pmatrix}
x\\ 1
\end{pmatrix}+
\begin{pmatrix}
b-Ax \\
1\\
1+c^\top(\bar x-x)
\end{pmatrix}^\top
T
\begin{pmatrix}
b-Ax \\
1\\
1+c^\top(\bar x-x)
\end{pmatrix}
\end{aligned}
\end{equation*}
If $S\succeq 0$ and $T \geq 0$, then for any $x\in \R^n$ such that $Ax\leq b$ and $c^\top (x-\bar x)\leq 1$, the right-hand side of the above equation is nonnegative. So we have
$
x^\top Q x+2 d^\top x \geq \nu_{_R}
$ for any $x$ such that $Ax\leq b$ and $c^\top (x-\bar x)\leq 1$. Therefore,~\eqref{eq:validcut} holds. 
\end{proof}
\Cref{l:S1S2} provides a sufficient condition for a given vector  $c\in \R^n$ to qualify as a valid cut. This condition can be verified by solving a  conic linear program  when $c\in \R^n$ is given, as equation~\eqref{eq:SOS} is linear in terms of $T$ and $S$ for given $c$. If $c\in \R^n$ is not predetermined, then equation~\eqref{eq:SOS} becomes nonlinear and computationally challenging to solve. To address this issue, we propose a specific choice for the multiplier matrix $T$ to eliminate the nonlinearity.
\begin{lemma}\label{l:csdp}
Let any $c\in \R^n$. 
If there exist $ S\succeq 0$, $T_1\geq 0$ and $\beta\geq 0$ such that
\begin{equation}\label{eq:SOS-r}
\begin{aligned}
\begin{pmatrix}
Q & d \\
d^\top & -\nu_{_R}
\end{pmatrix}&=S+\begin{pmatrix}
-A & b \\
\bzero^\top & 1 \\
\end{pmatrix}^\top
T_1
\begin{pmatrix}
-A & b  \\
\bzero^\top & 1 \\
\end{pmatrix} \\ & \quad\quad+\frac{1}{2}\begin{pmatrix}
Q\bar x+d \\ 
-\bar x^\top Q \bar x- d^\top \bar x +\beta
\end{pmatrix} \begin{pmatrix}
-c^\top & 1+c^\top \bar x
\end{pmatrix} \\ &\quad\quad+ \frac{1}{2}\begin{pmatrix}
-c \\ 1+c^\top \bar x
\end{pmatrix} \begin{pmatrix}
(Q\bar x+d)^\top &
-\bar x^\top Q \bar x- d^\top \bar x +\beta
\end{pmatrix},
\end{aligned}
\end{equation}
then $c$ is a valid cut.
\end{lemma}
\begin{proof}
Since $\bar x$ is a KKT point, there exist $\bar \lambda \geq 0$ such that~\eqref{eq:KKT} holds, which implies:
$$
\begin{pmatrix}
Q\bar x+d \\ 
-\bar x^\top Q \bar x- d^\top \bar x +\beta
\end{pmatrix}=\begin{pmatrix}
-A & b \\
\bzero^\top & 1 \\
\end{pmatrix}^\top \begin{pmatrix}
\bar \lambda \\ \beta
\end{pmatrix}.
$$
Hence~\eqref{eq:SOS-r} is the same as~\eqref{eq:SOS} with
$$
  T = 
  \left( \begin{array}{ *{3}{c} }
    &  &  \frac{1}{2}\bar \lambda \\
    \multicolumn{2}{c}
{\raisebox{\dimexpr\normalbaselineskip+.7\ht\strutbox-1.1\height}[0pt][0pt]
        {\scalebox{1.2}{$T_1$}}} & \frac{1}{2} \beta \\
    ~~\frac{1}{2}\bar \lambda^\top~~  & \frac{1}{2}{\beta}~ & 0
  \end{array} \right).
$$

\end{proof}

\subsection{Computationally Tractable Construction of Valid Cuts}\label{subsec:ctc}
 In view of~\Cref{l:csdp}, we propose to construct a valid cut by solving the following conic linear program for certain prefixed $\beta\geq 0$. 
\begin{equation}
    \label{eq:SDPsearchofc}
    \begin{aligned}
    \min_{S, T, c} \quad &  c^\top (\bar z-\bar x)  \\
        \textrm{\rm s.t.~~~}  & \begin{pmatrix}
Q & d \\
d^\top & -\nu_{_R}
\end{pmatrix}=S+\begin{pmatrix}
-A & b \\
\bzero^\top & 1 \\
\end{pmatrix}^\top
T
\begin{pmatrix}
-A & b  \\
\bzero^\top & 1 \\
\end{pmatrix} \\& \qquad \qquad\qquad\enspace+\frac{1}{2}\begin{pmatrix}
Q\bar x+d \\ 
-\bar x^\top Q \bar x- d^\top \bar x +\beta
\end{pmatrix} \begin{pmatrix}
-c^\top & 1+c^\top \bar x
\end{pmatrix} \\ &\qquad\qquad \qquad \enspace + \frac{1}{2}\begin{pmatrix}
-c \\ 1+c^\top \bar x
\end{pmatrix} \begin{pmatrix}
(Q\bar x+d)^\top &
-\bar x^\top Q \bar x- d^\top \bar x +\beta
\end{pmatrix}
\\
& ~T\geq 0,\enspace S\succeq 0.
    \end{aligned}
\end{equation}
Here, $\bar z\in \R^n$ is a given vector. 
By~\Cref{l:csdp}, any feasible solution of~\eqref{eq:SDPsearchofc} defines a valid cut $c$. Moreover, the objective function $c^\top (\bar z-\bar x)$  has the following interpretation:  the smaller $c^\top (\bar z-\bar x)$ is, the more likely the cut $\{x\in \R^n: c^\top (x-\bar x)\geq 1\}$ will also exclude the point $\bar z$. This suggests choosing $\bar z$ as a feasible point that one particularly expects to eliminate. 

Suppose that $(X^*, x^*)$ is an optimal solution of~\eqref{eq:SDP-lb}.
Then $x^* \in \cF$, and we propose to set $\bar z = x^*$ in~\eqref{eq:SDPsearchofc}. As explained, choosing $\bar z = x^*$ in~\eqref{eq:SDPsearchofc} will generate a cut that is more likely to eliminate $x^*$ from the feasible set, which can help to further improve the relaxation bound.

We now justify the specific choice of the multiplier matrix by demonstrating the existence of a vector $c\in \R^n$ that satisfies all the constraints in~\eqref{eq:SDPsearchofc} under mild conditions. This constitutes one of the main results of this paper. The detailed proof is deferred to~\Cref{sec-appendix:proofofthmSTc} due to its length.
\begin{theorem}\label{thm:STc}
 Let $\bar x$ be a KKT point of~\eqref{eq:QP} such that $Q|_{\cH_{I_{\bar x}}}\succ 0$. Let $\nu_{_R}\leq \Phi(\bar x)$ and $ \beta \in [0, \Phi(\bar x)-\nu_{_R}] $.  Then the following assertions hold:
 \begin{itemize}
 \item [(a)]
If $\bar x$ is nondegenerate or $ \beta>0$, then~\eqref{eq:SDPsearchofc} is feasible. 
\item [(b)] If $ 0<\beta< \Phi(\bar x)-\nu_{_R}$, then~\eqref{eq:SDPsearchofc} is strictly feasible.
\end{itemize}
\end{theorem}

\begin{remark}
We mentioned in~\Cref{rem:KKT2ndlocal} that if $\bar x$ is a nondegenerate KKT point  satisfying $Q|_{\cH_{I_{\bar x}}}\succ 0$, then $\bar x$ is a locally unique optimal point. 
\Cref{thm:STc} asserts  further that  nondegenerate KKT point $\bar x$  satisfying $Q|_{\cH_{I_{\bar x}}}\succ 0$  are in principle ``easy" to remove. The remaining case of degenerate KKT point essentially reduces to the  \textit{copositive programming} problem, as illustrated by the following simple example. 
\begin{equation}\label{eq:degenerate}
\begin{aligned}
 \min_{x\in \R^{n}} \quad &   x^\top Q x    \\
        \textrm{\rm s.t.} \quad &  x\geq 0,\enspace \bone^\top x \le 1 .
    \end{aligned}
\end{equation}
Here $\bar x=\bzero$ is a degenerate KKT point, because
$
\rank (A_{I_{\bar x}})=n,
$
but  $\bar \lambda=\bzero$  is the only multiplier satisfying~\eqref{eq:KKT}. The search of valid cut for~\eqref{eq:degenerate} with respect to $\bar x=\bzero$ and $\nu_{_R}=0$ is equivalent to certifying the copositiveness of the matrix $Q$, which is known to be an NP-hard problem.
\end{remark}

\subsection{Feasibility of the Dual Program}

The Lagrange dual program of~\eqref{eq:SDPsearchofc} can be written as follows:
\begin{equation}
    \label{eq:SDPsearchofcdual}
    \begin{aligned}
    \max_{\substack{U}  } \quad & ~~\left\langle -\begin{pmatrix}
Q & d \\
d^\top & -\nu_{_R}
\end{pmatrix}, U \right \rangle  + \begin{pmatrix}
\bzero \\ 1
\end{pmatrix} ^\top U \begin{pmatrix}
Q\bar x+d \\
-\bar x^\top Q \bar x- d^\top \bar x +\beta
\end{pmatrix}  \\
        \textrm{\rm s.t.~~~}   & \quad 
\bar z-\bar x- 
\begin{pmatrix}
\bI & -\bar x
\end{pmatrix} U 
\begin{pmatrix}
Q\bar x+d  \\
-\bar x^\top Q \bar x- d^\top \bar x +\beta
\end{pmatrix} =0 \\ & \quad
\begin{pmatrix}
-A & b \\
\bzero^\top & 1 \\
\end{pmatrix}  U \begin{pmatrix}
-A & b \\
\bzero^\top & 1 \\
\end{pmatrix} ^\top \geq 0 \\ & \quad
U\succeq 0.
    \end{aligned}
\end{equation}

\begin{lemma}
Suppose that $\beta>0$.
If $\bar x$ is a KKT point of~\eqref{eq:QP}  and $\bar z\in \cF$, then~\eqref{eq:SDPsearchofcdual} is feasible. 
\end{lemma}
\begin{proof}Let $\alpha>0$ and
$$
U= \alpha\begin{pmatrix}
\bar z \\ 1 
\end{pmatrix}
\begin{pmatrix}
\bar z \\ 1 
\end{pmatrix}^\top \succeq 0. 
$$
Since $\bar z\in \cF$,   $U$ satisfies the second nonnegativity constraints in~\eqref{eq:SDPsearchofcdual}. 
Note that 
$$
\begin{pmatrix}
\bI & -\bar x
\end{pmatrix} U 
\begin{pmatrix}
Q\bar x+d  \\
-\bar x^\top Q \bar x- d^\top \bar x +\beta
\end{pmatrix} \\
 =  \alpha \left[(Q\bar x+d)^\top (\bar z-\bar x)+\beta\right](\bar z-\bar x).
$$
Since $\bar x$ is a KKT point and $\bar z\in \cF$, we know that $(Q\bar x+d)^\top (\bar z-\bar x) \geq 0$. Hence if
$$
\alpha=((Q\bar x+d)^\top (\bar z-\bar x)+\beta)^{-1},
$$
then $U$ is a feasible solution of~\eqref{eq:SDPsearchofcdual}.

\end{proof}
It is also possible to make~\eqref{eq:SDPsearchofcdual} strictly feasible for some specific choices of $\bar z$, as shown in the following lemma. We denote by $T_{x}(\cF)$ the tangent cone of $\cF$ at $x\in \cF$.
\begin{lemma}
If $\cF$ has nonempty interior, there exists $\bar z$ such that 
$
\bar z-\bar x \in T_{\bar x}(\cF)
$
and~\eqref{eq:SDPsearchofcdual} is strictly feasible. 
\end{lemma}
\begin{proof}
If $\cF$ has nonempty interior, then there exits $x^{1},\ldots,x^{n+1}\in \cF$ and $\alpha_1,\ldots,\alpha_{n+1}\geq 0$ such that 
$$
U=\sum_{i=1}^{n+1} \alpha_ i\begin{pmatrix}
x^{i} \\ 1 
\end{pmatrix}
\begin{pmatrix}
x^{i} \\ 1 
\end{pmatrix}^\top  \succ 0
$$
and $U$ satisfies the  nonnegativity constraints in~\eqref{eq:SDPsearchofcdual}.
Hence~\eqref{eq:SDPsearchofcdual} is strictly feasible for 
$$
\begin{aligned}
\bar z & =\bar x +
\begin{pmatrix}
\bI & -\bar x
\end{pmatrix} U 
\begin{pmatrix}
Q\bar x+d  \\
-\bar x^\top Q \bar x- d^\top \bar x +\beta
\end{pmatrix} \\
 &=\bar x + \sum_{i=1}^{n+1} \alpha_i \left[(Q\bar x+d)^\top (x^i-\bar x)+\beta\right](x^i-\bar x).
\end{aligned}
$$
Then we notice that $x^i-\bar x\in T_{\bar x} (\cF)$ and $(Q\bar x+d)^\top (x^i-\bar x)+\beta \geq 0$ for each $i\in [n+1]$.
\end{proof}

\section{A Local Search Method with Finite Convergence}\label{sec:localsearch}

Motivated by~\Cref{thm:STc}, we provide in this section a method to find a KKT point $\bar x$ satisfying $Q|_{\cH_{I_{\bar x}}}\succ 0$.

\subsection{Generalized Mountain Climbing Method}
We first write $\Phi$ as a difference of convex (DC) functions. 
Let $M\succeq 0$ and $ N\succeq 0$ such that $Q=M-N$. Given such $M$ and $N$, 
define:
\begin{align}\label{eq: DC Psi function}
\Psi(x,\tilde{x}):=\frac{1}{2}x^\top M x+\frac{1}{2}\tilde{x}^\top M\tilde{x}+d^\top x+d^\top \tilde{x}-x^\top N\tilde{x}.
\end{align}

 \begin{lemma}\label{l:biKKT}
 A vector $\bar x\in \cF$ is a KKT point of~\eqref{eq:QP} if and only if
\begin{align}\label{a:neoptc}
\bar x\in \argmin_{x\in \cF} \Psi(x,\bar x).
\end{align}
\end{lemma}
\begin{proof}
The program $\min_{x\in \cF} \Psi(x,\bar x)$ corresponds to the following QP: 
\begin{equation}
    \label{eq:psiqp}
    \begin{aligned}
    \min_{x\in \R^{n}} \quad &   \frac{1}{2}x^\top M x +  d^\top x -x^\top N \bar x  \\
        \textrm{s.t.} \quad &  Ax\leq b.
    \end{aligned}
\end{equation}
Since $M\succeq 0$,~\eqref{eq:psiqp} is a convex QP problem. Hence~\eqref{a:neoptc} holds if and only if there exists $\bar \lambda$ such that 
\begin{equation}
    \left\{\begin{aligned}
    & \bar \lambda \geq 0\\
&(M-N)\Bar{x}+d=-A^\top  \bar \lambda\\
 &   (A\Bar{x}-b)^\top \bar \lambda=0,
\end{aligned}\right.
\end{equation}
which is exactly~\eqref{eq:KKT} because $Q=M-N$.
\end{proof}

\begin{lemma}\label{l:fwer}
Let $\tilde x\in \cF$ and $\bar x\in \argmin_{x\in \cF} \Psi(x,\tilde x)$. If $\Phi(\bar x)\geq\Phi(\tilde x)$, then $\tilde{x}$ is a KKT point of~\eqref{eq:QP}.
\end{lemma}
\begin{proof}
    If $\tilde x\in \cF$ is not a KKT point of~\eqref{eq:QP}, by \Cref{l:biKKT} we have $$\min\limits_{x\in\cF}\Psi(x, \Tilde{x})<\Psi(\Tilde{x}, \Tilde{x})=\Phi (\Tilde{x}).$$
    For any $\Bar{x}\in \argmin_{x\in \cF} \Psi(x,\Tilde{x})$, $$\min\limits_{x\in\cF}\Psi(x, \Tilde{x})=\Psi(\bar x, \Tilde{x})\geq \frac{1}{2}\left (\Phi(\Tilde{x})+\Phi (\Bar{x})\right).$$  Combining the above two inequalities we deduce that $\Phi(\bar x)< \Phi(\tilde x)$.
\end{proof}
Based on~\Cref{l:fwer}, given $\tilde x \in\cF$, we can either determine that $\tilde x $ is a KKT point, or find a feasible point $\bar x\in \cF$  with a strictly smaller objective value than $\tilde x$. This leads naturally to the following procedure which strictly improves the objective value until arriving at a KKT point. 

\begin{center}
\fbox{
\begin{minipage}{0.6\textwidth}
\begin{center}
\textbf{Generalized Mountain Climbing Method}
\end{center}
Initialization: $k=0$.

Repeat:
\begin{enumerate}
\item $k=k+1$.
    \item Find $x_{k+1} \in \arg\min_{x\in \cF} \Psi(x, x_k)$.
\end{enumerate}
until $\Phi(x_{k+1})=\Phi(x_k) $.
\end{minipage}
}
\end{center}

If $Q \preceq 0$, we may set $M = 0$, in which case we recover the mountain climbing algorithm for concave QP as described in~\cite{konno1976maximization}. For this reason, we refer to the above procedure as the generalized mountain climbing \texttt{(GMC)} algorithm, as it also inherits the strict monotonicity property of the objective value prior to termination:
$$
\Phi(x_{k+1}) < \Phi(x_k), \quad  k = 0, 1, \ldots.
$$
A notable feature of the mountain climbing algorithm for concave QP is its finite convergence property~\cite{konno1976maximization}. 
In fact, readers familiar with the difference of convex literature should have already recognized \texttt{GMC} iteration as a special case of the Difference-of-Convex Algorithm (DCA) method applied to QP~\cite{Tao1997}. Tao and An~\cite{Tao1997} established that the DCA converges in finitely many iterations if one of the components in the DC decomposition is \textit{polyhedral convex}. While this condition indeed holds for the concave QP case, it cannot be fulfilled for indefinite QP
. Thus, it is not known whether the above \texttt{GMC} method converges in a finite number of iterations for indefinite QP.

Also,  the output $\bar x$ of the \texttt{GMC} method  may not satisfy the second-order condition $Q|_{\cH_{I_{\bar{x}}}} \succ 0$ required by~\Cref{thm:STc} for generating a valid cut. Therefore, additional modifications  are necessary.

\subsection{Some Implications of the Second Order Condition}

Let $\hat x\in \cF$.
Consider the following QP problem:
\begin{equation}
    \label{eq:changed convex problem}
    \begin{aligned}
    \Phi^*_{I_{\hat x}}:=\min_{x\in \R^{n}} \quad &   x^\top Q x + 2 d^\top x   \\
        \textrm{s.t.} \quad &  a_i^\top x=b_i, \forall i\in I_{\hat x}\\
        &a_i^\top x\leq b_i, \forall i\in[m]\setminus I_{\hat x}.
    \end{aligned}
\end{equation}
Note that any feasible solution of ~\eqref{eq:changed convex problem} is a feasible solution of the original QP problem~\eqref{eq:QP}. 
In addition, if $Q|_{\cH_{I_{\hat x}}}\succ 0$,~\eqref{eq:changed convex problem} is a convex QP problem and thus computationally tractable.

\begin{lemma}\label{l:no vertex no pd}
Let any $\hat x\in \cF$.  If  $Q|_{\cH_{I_x}}\nsucc0$, then there exists $\bar x\in \cF$ such that \begin{equation}\label{eq:C_I generating new x}
\Phi(\hat x)\geq \Phi(\bar x)\enspace \mathrm{and}~~ I_{\hat x}\subsetneqq I_{\Bar{x}}.
\end{equation}
\end{lemma}
\begin{proof}
 If  $Q|_{\cH_{I_x}}\nsucc0$, there exists nonzero $h\in \cH_{I_{\hat x}}$ such that $ h^\top Qh\leq 0$ and $(Q\hat{x}+d)^\top h\leq 0$.
Since $\cF$ is bounded, there is at least one $i\in[m]\setminus I_{\hat x}$ such that $a_i^\top h>0$. Define:
$$
 \alpha_{\hat{x}} :=\min\limits_{\substack{i\in[m]\setminus I_{\hat x},\\a_i^\top  h>0}}\frac{b_i-a_i^\top \hat x}{a_i^\top h}.
 $$
 Clearly $\alpha_{\hat{x}}>0$.
Then we consider $
\Bar{x}:=\hat{x}+\alpha_{\hat{x}} h.$
It is immediate from the definition of $\alpha_{\hat{x}}$ that $\bar x\in \cF$ and $I_{\hat x}\subsetneqq I_{\Bar{x}}$. Moreover,
we have
\begin{equation*}
\Phi(\bar x)-\Phi(\hat x)
=\alpha_{\hat{x}}^2h^\top Qh+2\alpha_{\hat{x}}(Q\hat{x}+d)^\top h
\le 0.
\end{equation*}
\end{proof}

\subsection{Local Search Method and its Convergence}
 We now propose a local‐search method in  
Algorithm~\ref{alg:generalized downhill} for computing second order KKT point.  The algorithm starts  
from an initial feasible point $x_{0}\in\cF$ for~\eqref{eq:QP} and  
maintains an iteration counter $k$.  At iteration $k$, denote  
the current iterate by $\hat x = x_{k}$.  The update consists of two cases:

\begin{enumerate}
  \item[1.] \textbf{$Q\bigl|_{\cH_{I_{\hat x}}}\;\succ\;0$.} 
  
  Solve the convex QP
    \eqref{eq:changed convex problem} to obtain its minimizer $\tilde x$
    (line~\ref{algl:findtildexPhi}); clearly,
    \begin{equation}\label{eq:barxb1}
      \Phi(\tilde x)\;\le\;\Phi(\hat x)\,.  
    \end{equation}
    Next we perform the mountain‐climbing step starting from $\tilde x$ to produce $\bar x$ (\Cref{algline:findbarx1}).  The while loop breaks if $\Phi(\tilde x)\leq \Phi(\bar x)$ and set $\bar x=\tilde x$ in this case. So we always have:
    \begin{equation}\label{eq:barxb2}
      \Phi(\bar x)\;\le\;\Phi(\tilde x)\,.  
    \end{equation}

  \item[2.] \textbf{$Q|_{\cH_{I_{\hat x}}}\not\succ0$.}
  
   Invoke
    Lemma~\ref{l:no vertex no pd} (line~\ref{algline:no vertex no pd}) to
    produce $\bar x\in\cF$ such that
    \begin{equation}\label{eq:barxb3}
      \Phi(\bar x)\;\le\;\Phi(\hat x)
      \quad\text{and}\quad
      I_{\bar x}\;\supsetneqq\;I_{\hat x}\,.  
    \end{equation}
\end{enumerate}
Finally, set $x_{k+1}=\bar x$.  Since each case guarantees
\eqref{eq:barxb1}–\eqref{eq:barxb3}, it follows that
$$
  \Phi(x_{k+1}) \;\le\; \Phi(x_{k}).
$$
Therefore Algorithm~\ref{alg:generalized downhill} generates a nonincreasing  
sequence of objective values until termination.

\begin{algorithm}
    \caption{\texttt{Finite-Terminating GMC}}\label{alg:generalized downhill}
    \begin{algorithmic}[1]
    \Require {$Q \in \mathcal{S}^n$, $d \in \R^n$, $A \in \R^{m \times n}$, $b \in \R^m$, $x_0 \in \cF$.}
    \State {$k \leftarrow 0$.}
    \While {true}
     \State $\hat x \leftarrow x_k$.
\If{$Q|_{\cH_{I_{\hat x}}}\succ0$}\label{algline:vertexorpd}
        \State{Find  $\tilde x \in \argmin_{x\in \cF\cap \hat\cH_{I_{\hat x}}} \Phi(x)$}\label{algl:findtildexPhi}.
        \Comment{Solve the convex QP problem \eqref{eq:changed convex problem}.}
        \State{Find $\bar x\in \argmin_{x\in\cF}\Psi(x, \tilde{x})$.} \Comment{See $\Psi(\cdot, \cdot)$ in \eqref{eq: DC Psi function}.} \label{algline:findbarx1}
        \If{$\Phi(\bar x)\geq  \Phi(\tilde x)$}\label{algline:KKT}
        \State{$\bar x\leftarrow \tilde x $.}
     \State{\algorithmicbreak}
     \EndIf
    \Else
       \State{Find $\Bar{x}\in \cF$ such that  $\Phi(\bar x)\leq \Phi(\hat x)$ and $I_{\bar x}\supsetneqq I_{\hat x}$.}
       \Comment{See the proof of \Cref{l:no vertex no pd}.}
       \label{algline:no vertex no pd}
    \EndIf 
    \State{$x_{k+1}\gets \Bar{x}$.}
    \State{$k\leftarrow k+1$.}
    \EndWhile
    \Ensure{$\bar x \in \R^n$.}
    \end{algorithmic}
    \end{algorithm}

If 
\begin{equation}\label{eq:tildex}
\tilde x \in \argmin_{x\in \cF\cap \hat\cH_{I_{\hat x}}} \Phi(x),
\end{equation}
then $I_{\tilde x} \supset I_{\hat x}$ and thus $\cH_{I_{\tilde x}} \subset \cH_{I_{\hat x}}$. So if  $Q|_{\cH_{I_{\hat x}}}\succ0$, and $\tilde x$ satisfies~\eqref{eq:tildex}, it is necessary that   $Q|_{\cH_{I_{\tilde x}}}\succ0$. 
Based on~\Cref{l:fwer}, when the while loop terminates, the output $\bar x$ must be a  KKT point that satisfies $Q|_{\cH_{I_{\bar x}}} \succ 0$.
Moreover,~\Cref{alg:generalized downhill} enjoys the finite termination property.
\begin{theorem}\label{thm:gmcfinite}
\Cref{alg:generalized downhill} always terminates in finitely many steps.
\end{theorem}
In view of~\Cref{thm:gmcfinite}, we refer to~\Cref{alg:generalized downhill} as the finite-terminating \texttt{GMC} algorithm.


\begin{remark}
 Numerous existing algorithms guarantee convergence to a KKT point $ \bar{x} $ fulfilling the second order necessary optimality condition  $ Q|_{\mathcal{H}_{I_{\bar{x}}}} \succeq 0 $. The output of Algorithm~\ref{alg:generalized downhill} achieves in finite number of steps a slightly stronger property: $ Q|_{\mathcal{H}_{I_{\bar{x}}}} \succ 0 $.  This is motivated by~\Cref{thm:STc}, which requires the positive definiteness of $ Q $ on the null space of the active constraints for constructing valid cuts. We emphasize that Algorithm~\ref{alg:generalized downhill} is not designed to compete with existing local methods in computational efficiency. Rather, it serves a distinct purpose: starting from an arbitrary feasible point (including outputs from other efficient local methods), the algorithm generates a KKT point satisfying $ Q|_{\mathcal{H}_{I_{\bar{x}}}} \succ 0 $ in a finite number of iterations, so that subsequent construction of cutting planes can go on.
\end{remark}

\section{Algorithm and Numerical Aspects}\label{sec:algorithm}

We have now collected all the necessary pieces to implement the generic cutting plane framework described in~\Cref{subsec:cutpf}. We rely on~\Cref{alg:generalized downhill} to obtain a KKT point $\bar{x}$ satisfying $Q|_{\mathcal{H}_{I_{\bar{x}}}} \succ 0$. Importantly,~\Cref{thm:gmcfinite} guarantees that this can be done by solving a finite number of convex QPs. Given such a point $\bar{x}$, and based on~\Cref{thm:STc}, a valid cut $c$ can be obtained by solving the convex SDP program~\eqref{eq:SDPsearchofc}--\eqref{eq:SDPsearchofcdual}.

The convergence theories developed in~\Cref{thm:STc} and~\Cref{thm:gmcfinite} rely on the exact solution of convex QPs and SDPs. However, due to the limitations of finite-precision arithmetic, we only have access to approximate solutions of these convex programs in practice. In particular, even determining the index set of active constraints $I_x$ for a given $x \in \mathcal{F}$ is nontrivial, and verifying whether $Q|_{\mathcal{H}_{I_{\bar{x}}}} \succ 0$ introduces additional complexity. These issues necessitate addressing inexactness at each intermediate step of Algorithm~\ref{alg:generalized downhill} as well as in the cutting plane method during practical implementation.

In this section, we briefly discuss relevant implementation details and describe our approach to handling these sources of inexactness. A comprehensive study of these challenges and their impact on the algorithm's convergence is beyond the scope of this paper.

\subsection{Intermediate Steps in Finite-Terminating \texttt{GMC}}
The intermediate steps of~\Cref{alg:generalized downhill} can only be performed approximately. This begins with the most basic operation of determining the index set of bounding constraints, $I_{x} = \{i \in [m] : a_i^\top x = b_i\}$. In practical implementation, we compute $I_x$ by
$
I_{x} = \{i \in [m] : |a_i^\top x - b_i| \leq \delta\}
$
for some accuracy parameter $0 < \delta \ll 1$. Similarly, the condition $Q|_{\cH_{I_x}} \succ 0$ is considered to hold if the numerically computed minimal eigenvalue of the matrix $Q|_{\cH_{I_x}}$ is greater than $-\delta$.

For steps involving the solution of convex QPs, we are content with a numerical solution returned by standard solvers. The inequality $\Phi(\bar x) \geq \Phi(\tilde x)$ is considered satisfied if $\Phi(\bar x) \geq \Phi(\tilde x) - \delta$.

These modifications are unavoidable due to the finite precision of floating-point arithmetic in practical implementations. The effect of these approximations on the actual convergence behavior of the algorithm is unknown. Nevertheless, in all our experiments, we consistently observe finite termination of~\Cref{alg:generalized downhill}.

\subsection{Valid Lower Bound}\label{subsec:valid lower}
The Lagrange dual program of~\eqref{eq:SDP-lb} can be written as follows:
\begin{equation}
    \label{eq:SDP-lb-d}
    \begin{aligned}
    \max_{S, T,\lambda} \quad &  \lambda \\
        \textrm{\rm s.t.} \quad & \begin{pmatrix}
Q & d \\
d^\top & -\lambda
\end{pmatrix}=S+\begin{pmatrix}
-A & b \\
\bzero^\top & 1 \\
\end{pmatrix}^\top
T
\begin{pmatrix}
-A & b  \\
\bzero^\top & 1 \\
\end{pmatrix} 
\\
& T\geq 0,\enspace S\succeq 0.
    \end{aligned}
\end{equation}
 When we solve~\eqref{eq:SDP-lb} and~\eqref{eq:SDP-lb-d} using an appropriate numerical solver, we obtain approximate optimal solution matrices $(S^*, T^*,\lambda^*)$. By adding a simple projection step if necessary, we can assume that $T^* \succeq 0$ and $S^* \succeq 0$.
Let $\Delta \in \mathcal{S}_n$ be the correction matrix such that:
\begin{equation*}
    \begin{aligned}
        \begin{pmatrix}
            Q & d \\
            d^\top & -\lambda^*
        \end{pmatrix}
        = \Delta + S^* 
        + \begin{pmatrix}
            -A & b \\
            \mathbf{0}^\top & 1
        \end{pmatrix}^\top
        T^*
        \begin{pmatrix}
            -A & b \\
            \mathbf{0}^\top & 1
        \end{pmatrix}.
    \end{aligned}
\end{equation*}
 Define
$
\delta := \min\left(0, \lambda_{\min}(\Delta)\right).
$
If there exists a constant $r_0>0$ such that $\cF \subset \{x \in\R^n : \|x\| \leq r_0\}$, then a valid lower bound of~\eqref{eq:QP} is given by
$
    \lambda^* + \delta (1 + r_0^2).$
For example, if $\cF \subset \R^n_+$ and $\max\{\sum_{i=1}^n x_i : x \in \cF\} \leq t$, we can set $r_0 = t$. Alternatively, if $\cF \subset [-t, t]^n$, we can set $r_0 = \sqrt{n}t$. 

The pseudocode for computing a valid lower bound, as described above, is summarized in~\Cref{alg:dnnbound} in~\Cref{sec-appendix:auxialgo}.

\subsection{Valid Cut}
For the same reasons discussed in~\Cref{subsec:valid lower}, the optimal solutions to~\eqref{eq:SDPsearchofc}--\eqref{eq:SDPsearchofcdual} are not accessible by any floating-point-based numerical solver. In fact, even obtaining a feasible solution to~\eqref{eq:SDPsearchofc} is not as straightforward as one might think. From a numerical SDP solver, we only have access to matrices $T^* \geq 0$ and $S^* \succeq 0$, and a vector $c^* \in \R^n$ such that
\begin{equation}
    \label{eq:SDPeq}
    \begin{aligned}
    & \begin{pmatrix}
    Q & d \\
    d^\top & -\nu_{_R}
    \end{pmatrix}
    = \Delta + S^* + \begin{pmatrix}
    -A & b \\
    \mathbf{0}^\top & 1 \\
    \end{pmatrix}^\top
    T^*
    \begin{pmatrix}
    -A & b \\
    \mathbf{0}^\top & 1 \\
    \end{pmatrix} \\
    & \qquad\qquad\qquad + \frac{1}{2}\begin{pmatrix}
    Q\bar{x} + d \\ 
    - \bar{x}^\top Q \bar{x} - d^\top \bar{x} + \beta
    \end{pmatrix}
    \begin{pmatrix}
    -(c^*)^\top & 1 + (c^*)^\top \bar{x}
    \end{pmatrix} \\
    & \qquad\qquad\qquad + \frac{1}{2}\begin{pmatrix}
    -c^* \\ 1 + (c^*)^\top \bar{x}
    \end{pmatrix}
    \begin{pmatrix}
    (Q\bar{x} + d)^\top & - \bar{x}^\top Q \bar{x} - d^\top \bar{x} + \beta
    \end{pmatrix},
    \end{aligned}
\end{equation}
where, again, $\Delta \in \mathcal{S}_n$ represents the  error term. Similar to~\Cref{subsec:valid lower}, we define
$
\delta := \min\left(0,\, \lambda_{\min}(\Delta)\right)
$ and let $r_0>0$ be a constant such that $\cF \subset \{x \in\R^n : \|x\| \leq r_0\}$.
A readily available lower bound for the cut region can be expressed as follows:
\begin{equation}\label{eq:validcut2}
\begin{aligned}
    \nu_{R} + \delta(1 + r_0^2) \leq \min_{x \in \R^{n}} \quad & x^\top Q x + 2 d^\top x \\
    \text{s.t.} \quad &  Ax \leq b \\
    & (c^*)^\top(x - \bar{x}) \leq 1.
\end{aligned}
\end{equation}
We refer to $\nu_R+\delta(1+r_0^2)$ as the \textit{cut lower bound} associated with $c^*$. Note that if there is no numerical error (i.e., $\delta=0$), then $c^*$ is a valid cut and~\eqref{eq:validcut} holds for $c=c^*$. However, if $\delta > 0$, we cannot guarantee a priori that $c^*$ is a valid cut; we can only assert the validity of~\eqref{eq:validcut2}.

If the lower bound $\nu_R+\delta(1+r_0^2)$ given in~\eqref{eq:validcut2} is not satisfactory—for example, if it is smaller than a prescribed threshold—we can further refine it now that $c^*$ is known. Specifically, we may directly compute the DNN bound of the QP program in~\eqref{eq:validcut2}.

In practical implementations, it is necessary to store a lower bound of the QP program in~\eqref{eq:validcut2} so that the lower bound for all regions that have been cut can be retrieved. We emphasize that this additional storage is required only due to numerical errors that may arise when solving the program~\eqref{eq:SDPsearchofc}--\eqref{eq:SDPsearchofcdual}. The pseudocode for generating the cut and computing the cut lower bound is provided in~\Cref{alg:dnncut} in~\Cref{sec-appendix:auxialgo}.

\subsection{Main Algorithm}
We are now ready to present the main algorithm of this paper in~\Cref{alg:CuP-global}. This algorithm implements the generic cutting plane method described in~\Cref{subsec:cutpf}, and relies on~\Cref{alg:generalized downhill} to obtain a second-order KKT point $\bar{x}$ and to generate a valid cut $c$ by solving the convex SDP problem~\eqref{eq:SDPsearchofc}--\eqref{eq:SDPsearchofcdual}. We refer to this algorithm as \texttt{DCQP}, which stands for \textbf{D}NN \textbf{C}utting Planes based \textbf{QP} solver.

\Cref{alg:CuP-global} maintains a matrix ${\bf{A}}$  that contains $A$ as submatrix and a vector ${\bf{b}}$ that contains $b$ as a subvector. 
It starts by computing an initial lower bound from the DNN relaxation of the program~\eqref{eq:QP} and an initial upper bound from the GMC method (\Cref{algline:dnnbound0} to~\Cref{algline:intialbarv}). If the initial relative gap is not sufficiently small, it proceeds to  add cutting planes until the criterion
$
\bar{v} - \underline{v}^r \leq \epsilon \max(|\bar{v}|,\epsilon)
$
is satisfied (\Cref{algline:termnation-criterion-g} to~\Cref{algline:endwhile}).   Here,  $\underline v^r$  is a lower bound of the minimal objective value over the reduced region $\cA:=\{x\in \R^n: {\bf{A}}x \leq {\bf{b}} \}$, i.e.,
\begin{equation}\label{eq:bAb}
\begin{aligned}
      \underline v^r\leq \min_{x \in \R^{n}} \quad &   x^\top Q x + 2 d^\top x   \\
    \textrm{s.t.} \quad &  x\in \cA .
\end{aligned}
\end{equation}
We also compute  a lower bound of the minimal objective value over the removed region $\cF\backslash \cA$, i.e., 
\begin{equation}\label{eq:bAbc}
\begin{aligned}
     \underline v^c \leq \min_{x \in \R^{n}} \quad &   x^\top Q x + 2 d^\top x   \\
    \textrm{s.t.} \quad &  x\in \cF, x \notin  \cA
\end{aligned}
\end{equation}
Hence, at any stage of~\Cref{alg:CuP-global}, $$\underline v:=\min\{\underline{v}^r,\, \underline{v}^c\}$$ is a lower bound for $\Phi^*(\mathcal{F})$. An upper bound $\overline{v}$ for $\Phi^*(\mathcal{F})$ is also maintained at each iteration as the smallest objective value among all feasible solutions found so far.

\begin{algorithm}
    \caption{\texttt{DNN\_Cutting\_Plane\_for\_QP (DCQP)}}
    \begin{algorithmic}[1]
    \Require $Q \in \cS^n$, $d \in \R^n$, $A \in \R^{m\times n}$, $b \in \R^m$, relative gap tolerance $\epsilon > 0$, parameter $0<\eta\leq \epsilon$.
    \State  $\underline{v^c} \gets +\infty$, ${\bf{A}}\leftarrow A$, ${\bf{b}}\leftarrow b$.
  \State $(\underline v^r,\bar z) \leftarrow$ \texttt{DNN\_Lower\_Bound($Q,d, \bf{A},\bf{b}$).}
  \Comment{See \Cref{alg:dnnbound}.} \label{algline:dnnbound0}
  \State
  $\bar x\leftarrow$ \texttt{Finite\_Terminating\_GMC($Q, d, {\bf{A}}, {\bf{b}}, \bar z$)}.
        \Comment{See~\Cref{alg:generalized downhill}.} \label{algline:ftgmc}
        \State $\bar v \gets  \bar x^\top Q \bar x + 2 d^\top \bar x$, $\underline v\leftarrow \underline v^r$. \label{algline:intialbarv}
\While{$\bar v -\underline v^r > \epsilon  \max( \vert \bar v\vert,\epsilon)$}\label{algline:termnation-criterion-g}
        \State $\nu_R\leftarrow \bar v-\eta \max(|\bar v|,\epsilon) $, $\nu \leftarrow   \bar v-\epsilon \max(\vert \bar v\vert, \epsilon)$.
        \State  $(c, w) \leftarrow$\texttt{DNN\_Cut($Q, d, {\bf{A}}, {\bf{b}},\bar x, \bar z, \nu_R, \nu $)}.
    \Comment{See~\Cref{alg:dnncut}.} \label{algline:dnncut}
    \State ${\bf{A}} \leftarrow \begin{pmatrix}
    {\bf{A}}\\
    -c^\top
    \end{pmatrix}$, ${\bf{b}} \leftarrow \begin{pmatrix}
    {\bf{b}}\\
    -c^\top \bar x-1
    \end{pmatrix}$. \label{algline:addcut}
    \Comment{Add a valid cut to $\cF$.}
       \State $\underline v^c \leftarrow \min \{w, \underline v^c\}$. \label{algline:vc}
       \State $(\underline v^r,\bar z) \leftarrow$ \texttt{DNN\_Lower\_Bound($Q,d, \bf{A},\bf{b}$).}
        \Comment{See \Cref{alg:dnnbound}.}
         \State $\bar x\leftarrow$ \texttt{Finite\_Terminating\_GMC($Q, d, {\bf{A}}, {\bf{b}}, \bar z$)}.
        \Comment{See~\Cref{alg:generalized downhill}.} 
        \State $\bar v \gets \min\left\{\bar v, \bar x^\top Q \bar x + 2 d^\top \bar x\right\}$. \Comment{Keep the best objective value ever reached.}
        \State $\underline v \leftarrow \min\{\underline v^r, \underline v^c\}$.
    \EndWhile \label{algline:endwhile}
    \Ensure Upper bound $\bar v$, and lower bound  $\underline v$.
    \end{algorithmic}
    \label{alg:CuP-global}
\end{algorithm}

\begin{remark}
\Cref{alg:CuP-global} is not an extension of the main algorithm \texttt{QuadProgCD-G} proposed in~\cite{QuZengLou2025} from special concave QPs to indefinite QPs, although both algorithms fit within the generic cutting plane framework. Specifically, the cut generated by \texttt{QuadProgCD-G} is always a scalar multiple of Konno's cut, which is computed by solving $n$ linear programs and is restricted to the concave case. In contrast, the cut in \Cref{alg:CuP-global} is obtained by solving a single semidefinite program. Even when specialized to the concave QP setting, the cut generated in \Cref{alg:CuP-global} remains fundamentally different from that of \texttt{QuadProgCD-G}.
\end{remark}

\begin{remark}\label{rem:swrr}
If the output $\bar{x}$ of~\Cref{alg:generalized downhill} is a second order KKT point, and if the SDP program~\eqref{eq:SDPsearchofc} is solved exactly to optimality, then $\underline{v}^c \geq \nu_R = \bar{v} - \eta \max(\vert \bar v \vert,\epsilon) \geq \bar{v} - \epsilon \max(\vert \bar v \vert,\epsilon)$.   In this idealized situation, at termination of~\Cref{alg:CuP-global} the output must satisfy
\begin{equation}\label{eq:relativgaplessepsiln}
\bar{v} - \underline{v} \leq \epsilon \max(|\bar v|,\epsilon).
\end{equation}
However, due to numerical errors, the output $\bar{x}$ of~\Cref{alg:generalized downhill} is only close to a second-order KKT point, and the SDP program~\eqref{eq:SDPsearchofc} is only solved approximately; therefore, in principle, we have no guarantee on $\underline{v}^c$ when~\Cref{alg:CuP-global} terminates. Nevertheless, if all numerical errors in the underlying steps are controlled to high precision, it is expected that $\underline{v}^c$ remains close to $\bar{v} - \epsilon\max(|\bar v|,\epsilon)$. This is consistent with our observations in numerical experiments, see~\Cref{rem:instance11}.
\end{remark}

\begin{remark}
\Cref{thm:STc} ensures that a cutting plane can always be generated as long as the termination criterion is not satisfied. However, even when assuming exact arithmetic (i.e., no numerical errors), this property alone is not enough to prove the convergence of~\Cref{alg:CuP-global}. A rigorous convergence analysis is left for future work; in the rest of this paper, we focus on presenting the numerical behavior of the algorithm.
\end{remark}

\section{Numerical Experiments}\label{sec:experiments}

In this section, we evaluate the numerical performance of the proposed cutting plane algorithm, \texttt{DCQP} (Algorithm~\ref{alg:CuP-global}). We compare \texttt{DCQP} with the popular commercial solver \texttt{Gurobi} (version 12.0.1)
, as well as with an open-source branch-and-bound-based academic QP solver \texttt{QPL}~\cite{LiuzziLocatelliPiccialli2022}.\footnote{ \url{https://github.com/gliuzzi/QPL}.}

Our algorithm \texttt{DCQP} is implemented in \texttt{MATLAB}, and \texttt{Gurobi} is run through its \texttt{MATLAB} interface. \texttt{QPL} is written and executed using \texttt{Julia} (version 1.11.4). All experiments are conducted on a Windows desktop equipped with an Intel(R) Core(TM) i7-14700 CPU (2.10 GHz, 20 cores) and 64GB RAM.
The implementation of the \texttt{DCQP} algorithm is available at \url{https://github.com/zhengqu-x/DCQP}.

The accuracy of the solutions is evaluated using the \textit{relative gap}, defined as
\begin{equation}\label{eq:rg}
\frac{\bar{v} - \underline{v}}{\max(|\bar{v}|, \epsilon)},
\end{equation}
where $\bar{v}$ denotes the best upper bound and $\underline{v}$ denotes the best lower bound obtained at any stage of the algorithm.

We set $\epsilon = 10^{-4}$ and  $\eta = 9\times 10^{-5}$ in \texttt{DCQP} (Algorithm \ref{alg:CuP-global}). The SDP subproblems are solved with \texttt{MOSEK} (version 11.0.12)
, and the convex QP subproblems are solved with \texttt{Gurobi} (version 12.0.1).  All the displayed computational time are measured in seconds. In the following, without further specification, we say that an algorithm \textit{successfully solves the instance} if it reduces the relative gap~\eqref{eq:rg} below $10^{-4}$.

\subsection{Instances}
We evaluate our approach on both the benchmark  QP dataset 
from~\cite{chen2012globally}, as well as on synthetically generated data. The instances used in our comparisons are publicly available at the above repository.  All instances fit into the following model\footnote{For notational simplicity, we present our approach for QP problems formulated with inequality constraints only. Extending the algorithm to QP problems with equality constraints is straightforward. The details are omitted for brevity.}:
\begin{equation}\label{eq:RandQP}
\begin{aligned}
\min_{x\in\R^n} \quad &  x^\top Q x +2 d^\top x  \\
\textrm{s.t.} \quad & A x \le b \\
& A_{\rm eq} x = b_{\rm eq} \\
& \bzero \le x \le \bone.
\end{aligned}
\end{equation}
In the following, $m_{\rm eq}$ denotes the number of equality constraints (i.e., the dimension of the vector $b_{\rm eq}$) in the model~\eqref{eq:RandQP}.

\subsubsection{Benchmark Data}
The benchmark dataset consists of 64 indefinite QP instances used in~\cite{chen2012globally}, and is downloaded from \href{https://sburer.github.io/projects.html}{https://sburer.github.io/projects.html}.

This dataset has been widely adopted for numerical comparisons in previous work~\cite{LiuzziLocatelliPiccialli2022,Xia2020}. These 64 instances are categorized into four groups: 
\begin{center}
\texttt{qp\_20\_10}, \texttt{qp\_30\_15}, \texttt{qp\_40\_20}, and \texttt{qp\_50\_25}.
\end{center}In these group names, the first number indicates the number of variables $n$, while the second number indicates the number of inequality constraints $m$.  The numbers of equality constraints for these four groups are $4$, $6$, $8$, and $10$, respectively. Each group contains 16 instances. The average density (proportions of nonzero entries) of matrices $Q$, $A_{\rm eq}$, and $A$ are \{0.62, 0.39, 0.38\}, \{0.61, 0.38, 0.38\}, \{0.6, 0.38, 0.37\}, \{0.6, 0.38, 0.38\}, respectively for the four groups.

\subsubsection{Synthetic Data}
To further evaluate performance on larger-scale problems, we generated 140 synthetic instances with $n = 100$ variables. Each instance follows the form of~\eqref{eq:RandQP}, where the data matrices $Q$, $A$, $b$, $A_{\rm eq}$, and $b_{\rm eq}$ are randomly generated. Additionally, the matrix $A$ includes a row of all ones. Each instance features 51 inequality constraints: 50 are randomly generated, and one corresponds to the normalization constraint $\bone^\top x \le t$. Further details on the data generation process can be found in the git repository \url{https://github.com/zhengqu-x/DCQP}.

The 140 instances are divided into seven groups based on the data distribution, the number of equality constraints, and the density of the matrices. The names of the seven groups are:
\begin{center}
    \texttt{qp\_n\_0\_1}, \texttt{qp\_n\_0\_3}, \texttt{qp\_n\_0\_9}, 
     \texttt{qp\_u\_0\_1}, \texttt{qp\_u\_0\_3}, \texttt{qp\_u\_0\_9}, and 
     \texttt{qp\_u\_25\_1}.
\end{center}
Here, \texttt{n} and \texttt{u} denote whether the random entries are drawn from a normal or uniform distribution, respectively. The first number indicates the number of equality constraints ($m_{\rm eq}$), and the last number specifies the density level—the larger this value, the denser the matrices $Q$ and $A$. For the \texttt{qp\_n\_0\_1}, \texttt{qp\_n\_0\_3}, and
\texttt{qp\_n\_0\_9} instances, the average proportions of nonzero entries in matrices $Q$ and $A$ are \{0.37, 0.11\}, \{0.97, 0.27\}, and \{1, 0.61\}, respectively.
For the \texttt{qp\_u\_0\_1}, \texttt{qp\_u\_0\_3}, and \texttt{qp\_u\_0\_9} instances, these proportions are \{0.82, 0.2\}, \{1, 0.46\}, and \{1, 0.84\}, respectively. For the \texttt{qp\_u\_25\_1} instances, the average proportions of nonzero entries in matrices $Q$, $A_{\rm eq}$, and $A$ are \{0.81, 0.18, 0.2\}.

Each group contains 20 instances. The first six groups correspond to instances without equality constraints, while the last group (\texttt{qp\_u\_25\_1}) contains 25 equality constraints.

\subsection{Results for Benchmark Data}\label{sec:percentageRandQP}


In this section, we compare \texttt{DCQP} with \texttt{Gurobi} and \texttt{QPL} on the 64 benchmark  instances.  In~\Cref{fig:comprandQP}, we show the plots of the percentage of instances solved versus the running time of the three methods, for each of the four groups \texttt{qp\_20\_10}, \texttt{qp\_30\_15}, \texttt{qp\_40\_20} and \texttt{qp\_50\_25}.  In~\Cref{table:randqp}, we provide additional details, including the initial relative gap (measured before the while loop begins in~\Cref{alg:CuP-global}), the final relative gap, and the number of cutting planes generated for each of the 64 instances.

From \Cref{fig:comprandQP}, we observe that for {all the four groups}, our \texttt{DCQP}  successfully reduced the relative gap below $10^{-4}$ for all the instances in the same group within the shortest wall-clock time.

\begin{figure}[ht]
    \centering
    \begin{subfigure}[t]{0.45\textwidth}
        \includegraphics[width=\textwidth]{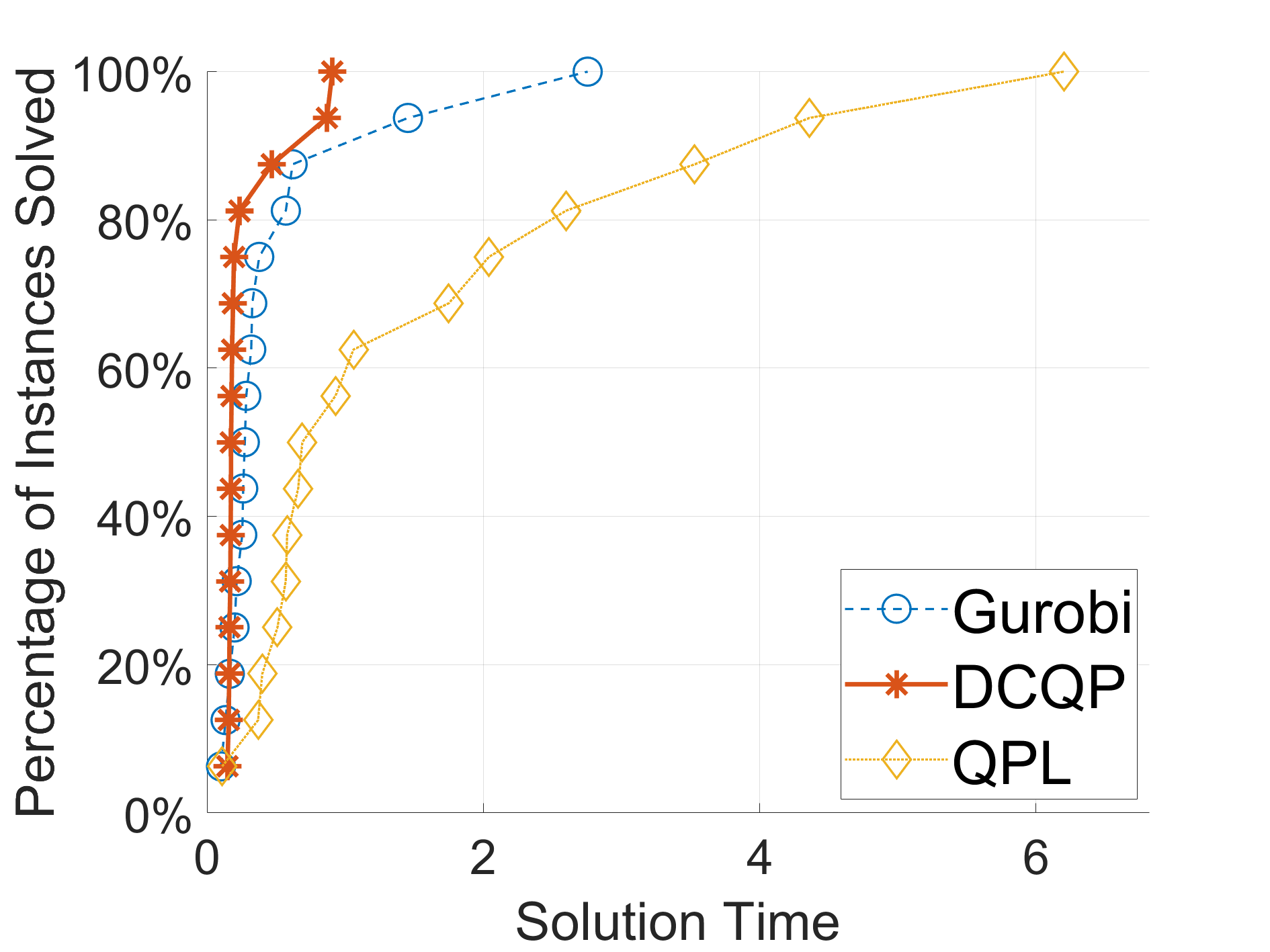}
        \caption{\texttt{qp}\_\texttt{20}\_\texttt{10} group.}
  \label{subfig1:QC100}
    \end{subfigure}
    \begin{subfigure}[t]{0.45\textwidth}
        \includegraphics[width=\textwidth]{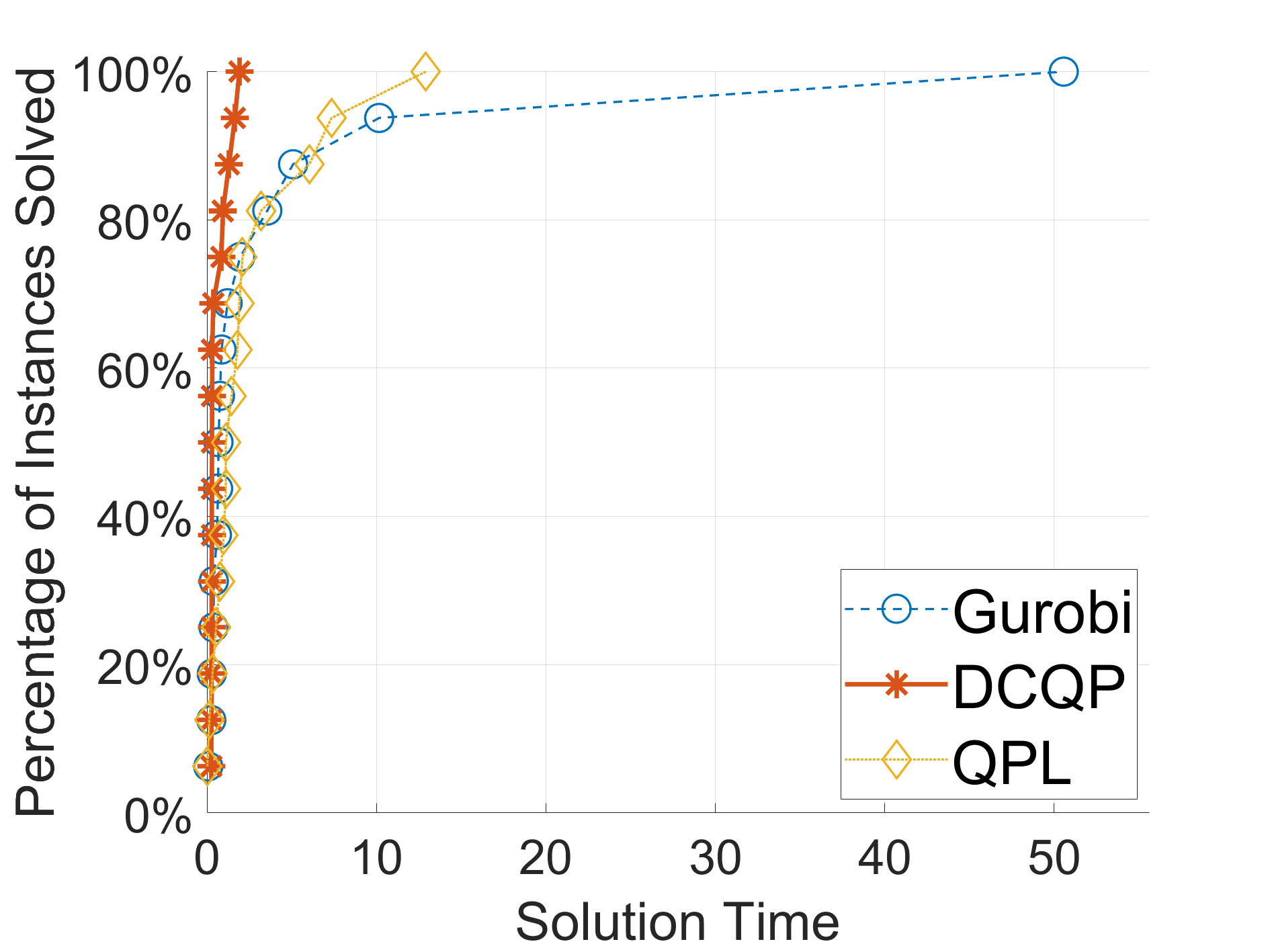}
        \caption{\texttt{qp}\_\texttt{30}\_\texttt{15} group.}
    \label{subfig2:QC100}
    \end{subfigure}
    \hfill
    \begin{subfigure}[t]{0.45\textwidth}
        \includegraphics[width=\textwidth]{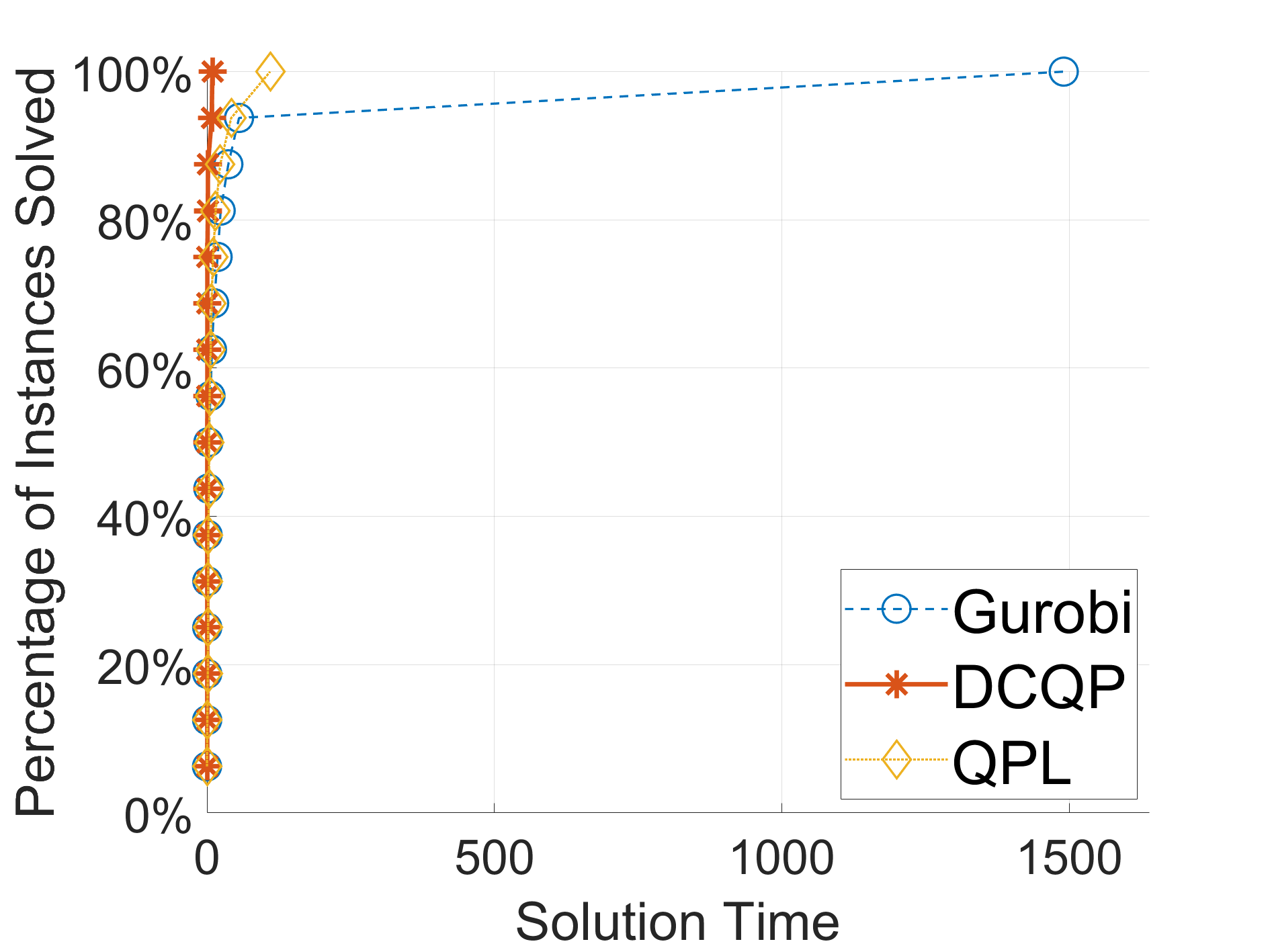}
        \caption{\texttt{qp}\_\texttt{40}\_\texttt{20} group.}
   \label{subfig3:QC100}
    \end{subfigure}
    \begin{subfigure}[t]{0.45\textwidth}
        \includegraphics[width=\textwidth]{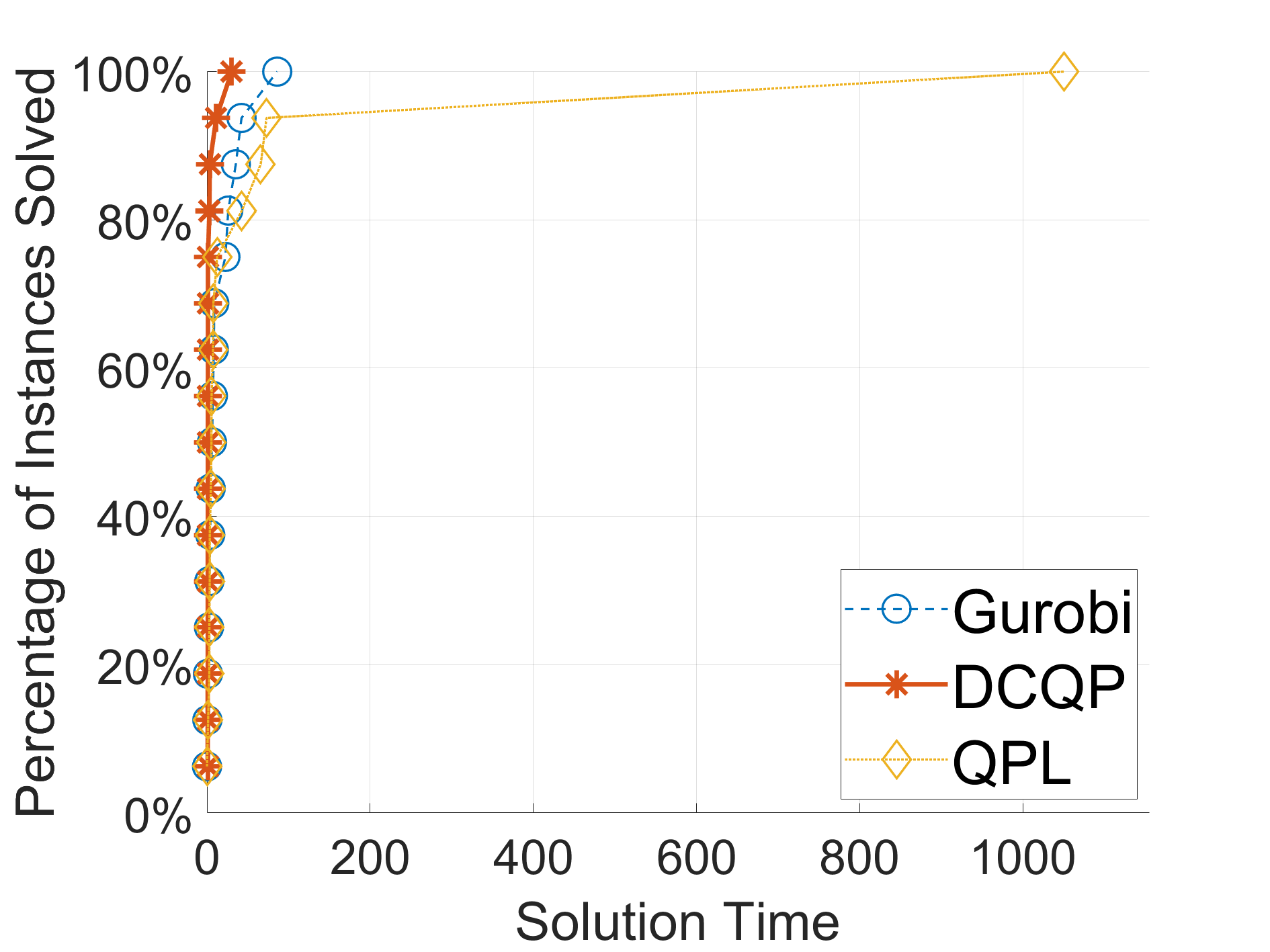}
        \caption{\texttt{qp}\_\texttt{50}\_\texttt{25} group.}
   \label{subfig4:QC100}
    \end{subfigure}
    \caption{Comparison of the percentage of instances solved versus computational time across four groups of benchmark instances, each comprising 16 instances.}
    \label{fig:comprandQP}
\end{figure}

By examining~\Cref{table:randqp}, we observe that, for most instances in the benchmark dataset, the number of generated cuts is zero. This suggests that, in these cases, the DNN bound is already sufficiently close to the optimal value, and no cutting planes are required to achieve the target relative gap accuracy of $10^{-4}$. Although the short computational times for these instances do not directly demonstrate the benefits of our approach, they do underscore the strong potential of using the DNN bound in globally solving QP problems. For instances where the DNN bound is not sufficiently tight, our approach becomes effective, successfully reducing the gap by adding only a small number of cutting planes.

Overall, the performance of our DNN-based cutting plane approach is largely comparable to the other two methods, but demonstrates greater stability on these 64 instances, exhibiting the smallest variation in computational time across all cases.


\subsection{Results for Synthetic Instances without Equality Constraints}\label{sec:percentagesynthetic}

In this section, we compare \texttt{DCQP}, \texttt{Gurobi}, and \texttt{QPL} on the six groups  of  synthetic instances with $m_{eq}=0$: \texttt{qp\_n\_0\_1}, \texttt{qp\_n\_0\_3}, \texttt{qp\_n\_0\_9}, \texttt{qp\_u\_0\_1}, \texttt{qp\_u\_0\_3} and \texttt{qp\_u\_0\_9}.

\subsubsection{Performance within One Hour}\label{subsubsec:synnoe1h}

We evaluate all three methods with a 3600s time limit, presenting the percentage of solved instances over time for each group in~\Cref{fig:comprandsynthetic}, and detailed gap statistics in~\Cref{table:merged-normalrandomqp} and~\Cref{table:merged-randomqp}. Key observations are as follows:
\begin{itemize}
    \item \textbf{Low-density} (\texttt{qp\_n\_0\_1}, \texttt{qp\_u\_0\_1}): Both \texttt{QPL} and \texttt{DCQP} efficiently solve all instances, while \texttt{Gurobi} performs well on \texttt{qp\_n\_0\_1} but is less effective on \texttt{qp\_u\_0\_1}. \texttt{DCQP} typically requires zero cuts, indicating tight DNN bounds for sparse cases.
    \item \textbf{Intermediate density} (\texttt{qp\_n\_0\_3}, \texttt{qp\_u\_0\_3}): All methods require more time. \texttt{DCQP} solves nearly all instances within 3600s, whereas \texttt{Gurobi} fails on 6 (\texttt{qp\_n\_0\_3}) and 10 (\texttt{qp\_u\_0\_3}) instances, and \texttt{QPL} fails on 4 and 17 instances, respectively.
    \item \textbf{High density} (\texttt{qp\_n\_0\_9}, \texttt{qp\_u\_0\_9}): Only \texttt{DCQP} reliably solves all instances within the time limit; both \texttt{Gurobi} and \texttt{QPL} fail on nearly all.
\end{itemize}
Overall, \texttt{DCQP} achieves the best solvability and computational efficiency, particularly for moderately sparse and dense nonconvex QP instances.

\begin{figure}[ht]
    \centering
    \begin{subfigure}[t]{0.45\textwidth}
        \includegraphics[width=\textwidth]{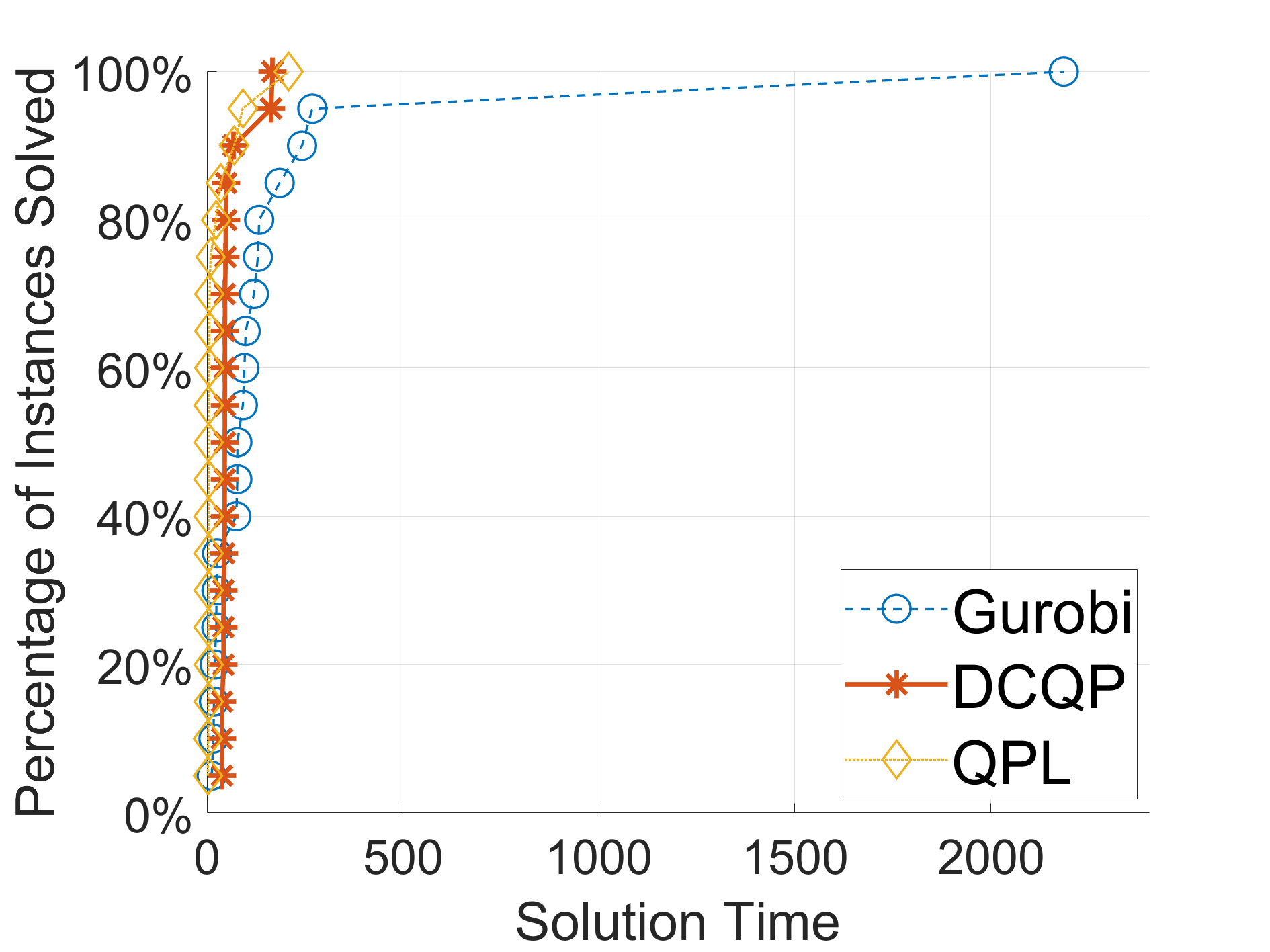}
        \caption{\texttt{qp\_u\_0\_1} group.}
  \label{subfig1:QC100-2}
    \end{subfigure}
    \begin{subfigure}[t]{0.45\textwidth}
        \includegraphics[width=\textwidth]{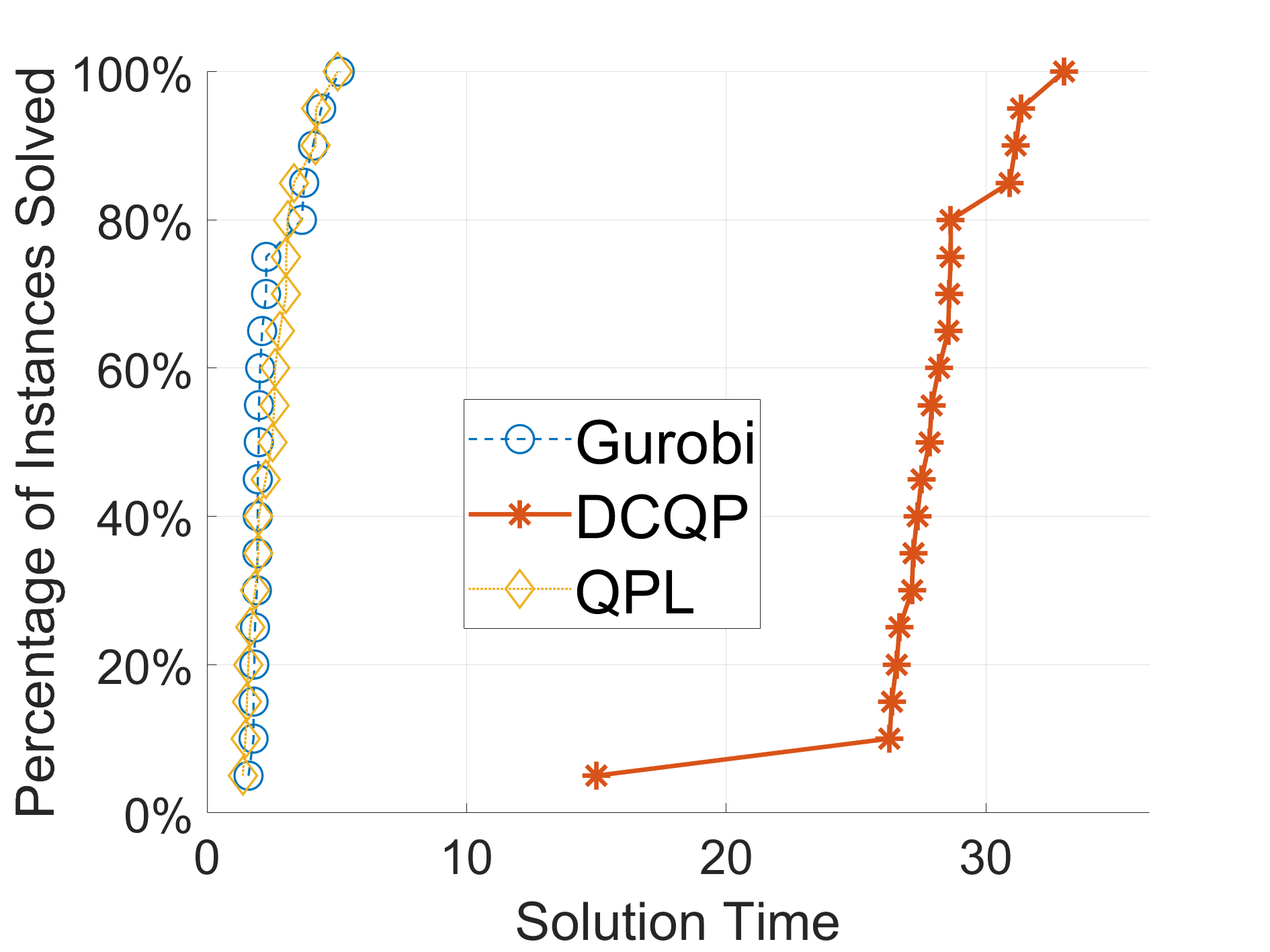}
        \caption{\texttt{qp\_n\_0\_1} group.}
    \label{subfig2:QC100-2}
    \end{subfigure}
    \hfill
    \begin{subfigure}[t]{0.45\textwidth}
        \includegraphics[width=\textwidth]{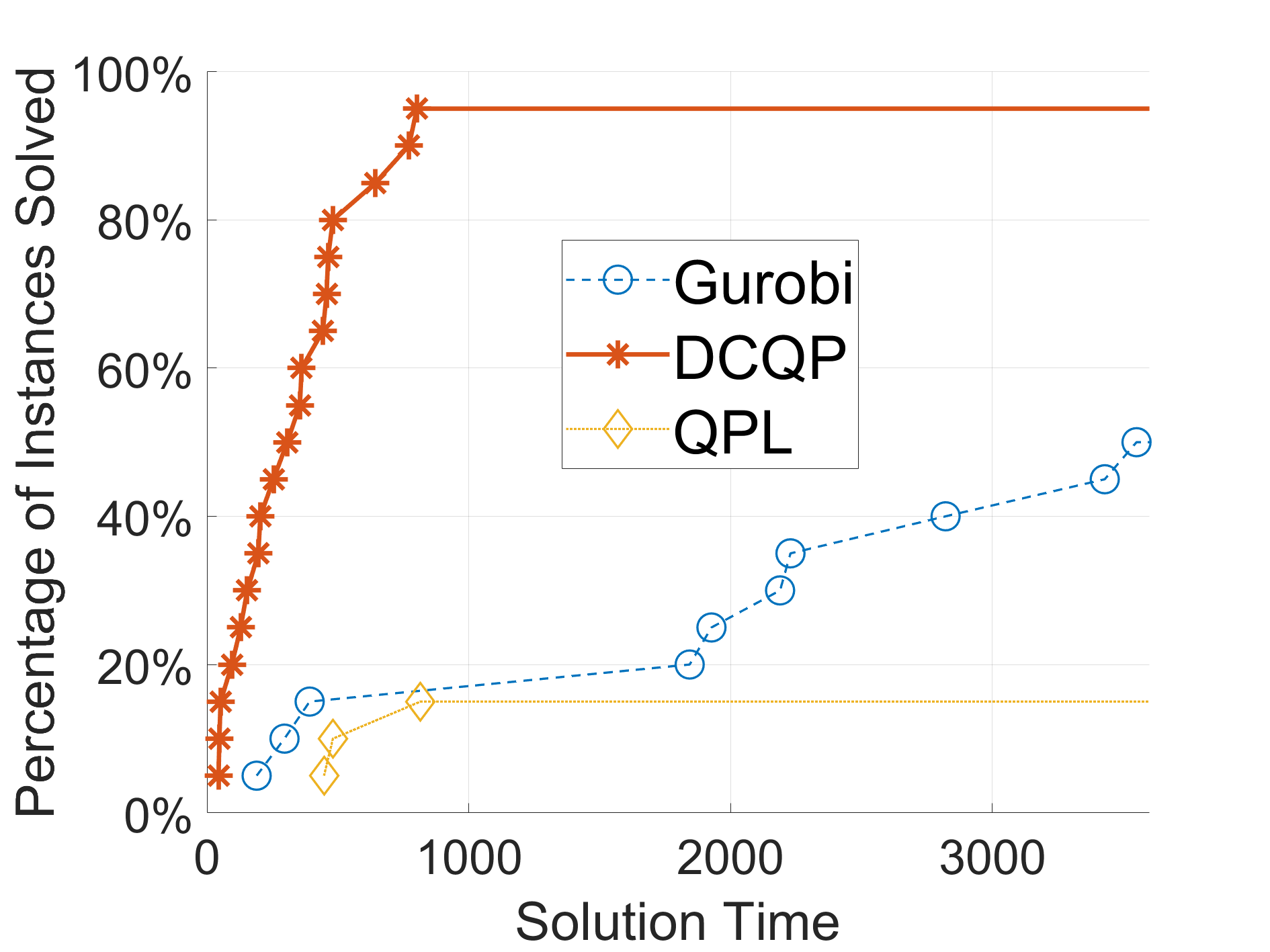}
        \caption{\texttt{qp\_u\_0\_3} group.}
   \label{subfig3:QC100-2}
    \end{subfigure}
    \begin{subfigure}[t]{0.45\textwidth}
        \includegraphics[width=\textwidth]{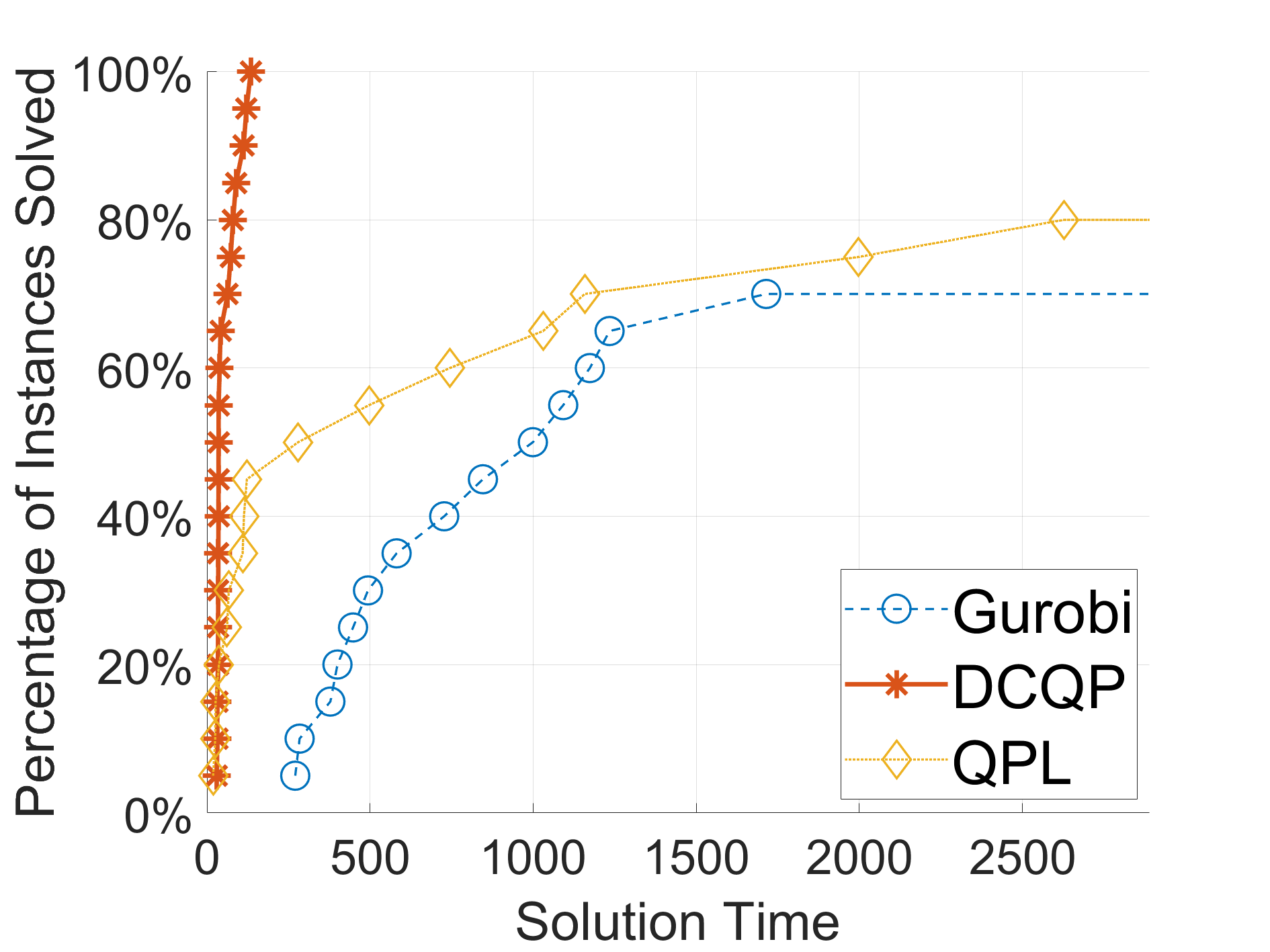}
        \caption{\texttt{qp\_n\_0\_3} group.}
   \label{subfig4:QC100-2}
    \end{subfigure}
    \begin{subfigure}[t]{0.45\textwidth}
        \includegraphics[width=\textwidth]{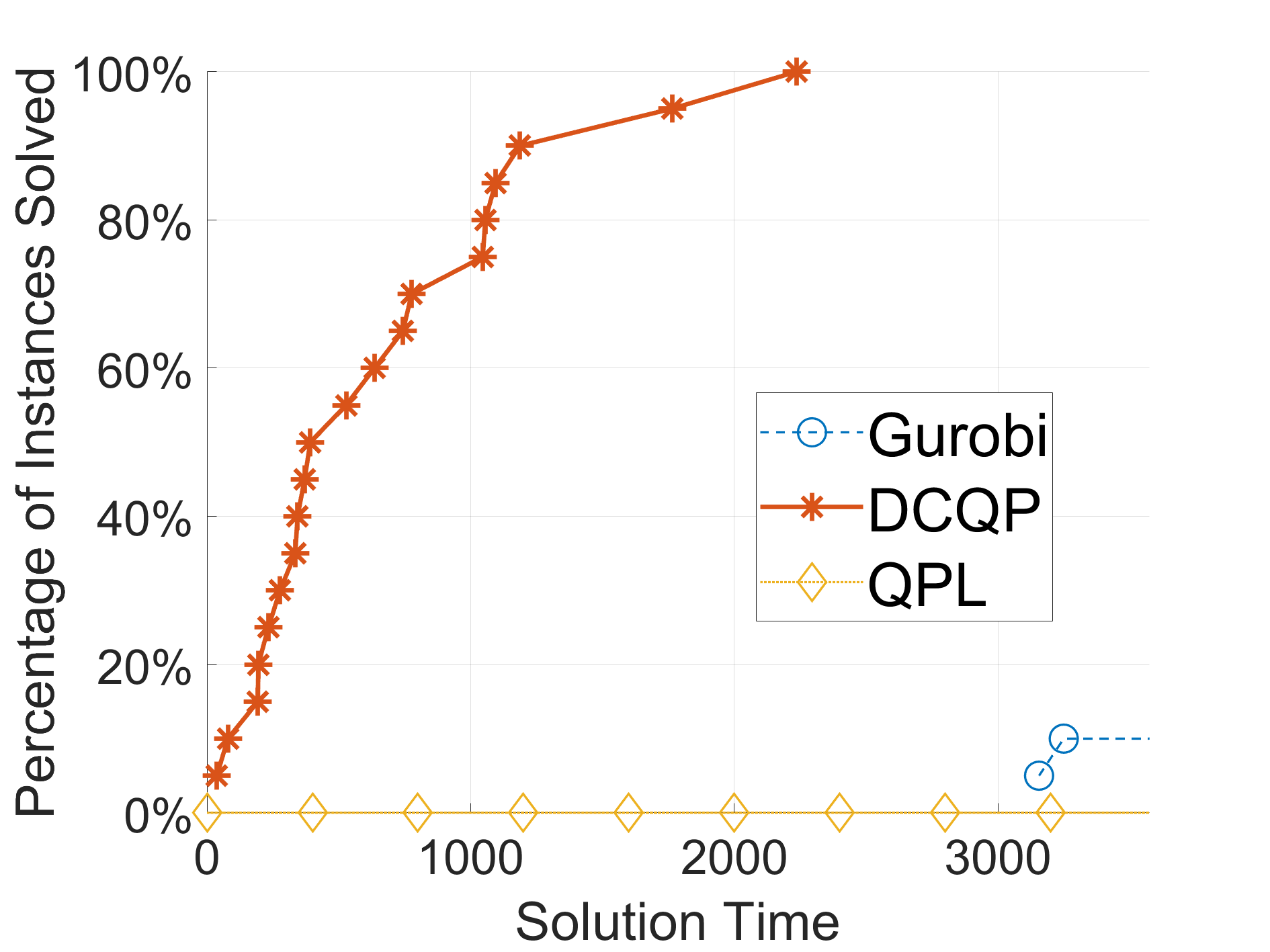}
        \caption{\texttt{qp\_u\_0\_9} group.}
   \label{subfig5:QC100-2}
    \end{subfigure}
    \begin{subfigure}[t]{0.45\textwidth}
        \includegraphics[width=\textwidth]{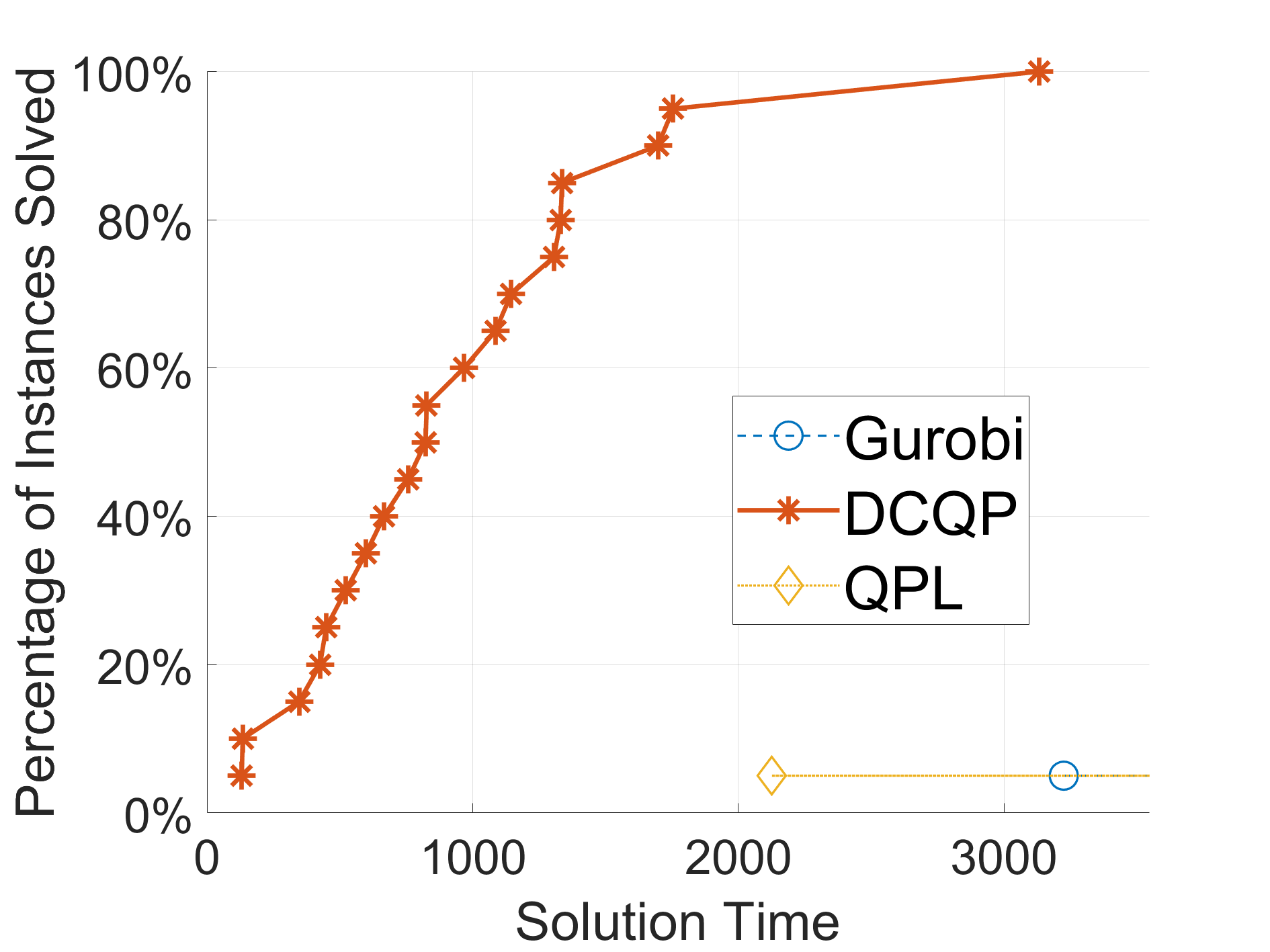}
        \caption{\texttt{qp\_n\_0\_9} group.}
   \label{subfig6:QC100-2}
    \end{subfigure}
    \caption{Comparison of the percentage of instances solved versus computational time, under a 3600-second time limit, across six groups of synthetic instances without equality constraints, each comprising 20 instances.}
\label{fig:comprandsynthetic}
\end{figure}

\subsubsection{Performance within 48 Hours}\label{sec:furthercomparison}

We have seen that our method \texttt{DCQP} can solve almost all the instances in the six groups within the time limit of one hour, while  \texttt{Gurobi} and \texttt{QPL} fail on many instances with moderately sparse and dense data. In this section, we extend the time limit to 48 hours to compare the convergence speed of the three methods on a few selected instances.

Specifically, we choose one instance for each group, where \texttt{Gurobi} and \texttt{QPL} did not converge within {3600 seconds}, and for which our \texttt{DCQP} has the longest runtime for all the instances in the group\footnote{For some groups, such instances may not exist.}. These instances are:
\begin{itemize}
  \item Instance 5 in the group \texttt{qp}\_\texttt{n}\_\texttt{0}\_\texttt{3};
  \item Instance 6  in the group \texttt{qp}\_\texttt{n}\_\texttt{0}\_\texttt{9};
  \item Instance 7 in the group
  \texttt{qp}\_\texttt{u}\_\texttt{0}\_\texttt{3};
  \item Instance 18 in the group
  \texttt{qp}\_\texttt{u}\_\texttt{0}\_\texttt{9};
\end{itemize}

 In~\Cref{fig:comp48h}, we present the plot of relative gap versus solution time (with a time limit of 48 hours) for the three methods applied to these four instances. The results clearly show that \texttt{Gurobi} and \texttt{QPL} often stagnate at relatively large gaps, failing to achieve the target accuracy of $10^{-4}$ within the allotted time. Notably, only for instance~7 in \texttt{qp\_u\_0\_3} does \texttt{Gurobi} succeed in reducing the relative gap below $10^{-4}$ within \textbf{48 hours}, but at the cost of a computational time approximately \textbf{20 times longer} than that required by \texttt{DCQP}.

In contrast, our \texttt{DCQP} consistently reduces the relative gap to the order of $10^{-4}$ for the four instances, with running times ranging from 121s to 3800s, thus achieving both faster convergence and significantly higher accuracy. These results underscore the robustness and efficiency of \texttt{DCQP} for solving nonconvex QP problems.

\begin{figure}[ht!]
    \centering
    \begin{subfigure}[t]{0.45\textwidth}
        \includegraphics[width=\textwidth]{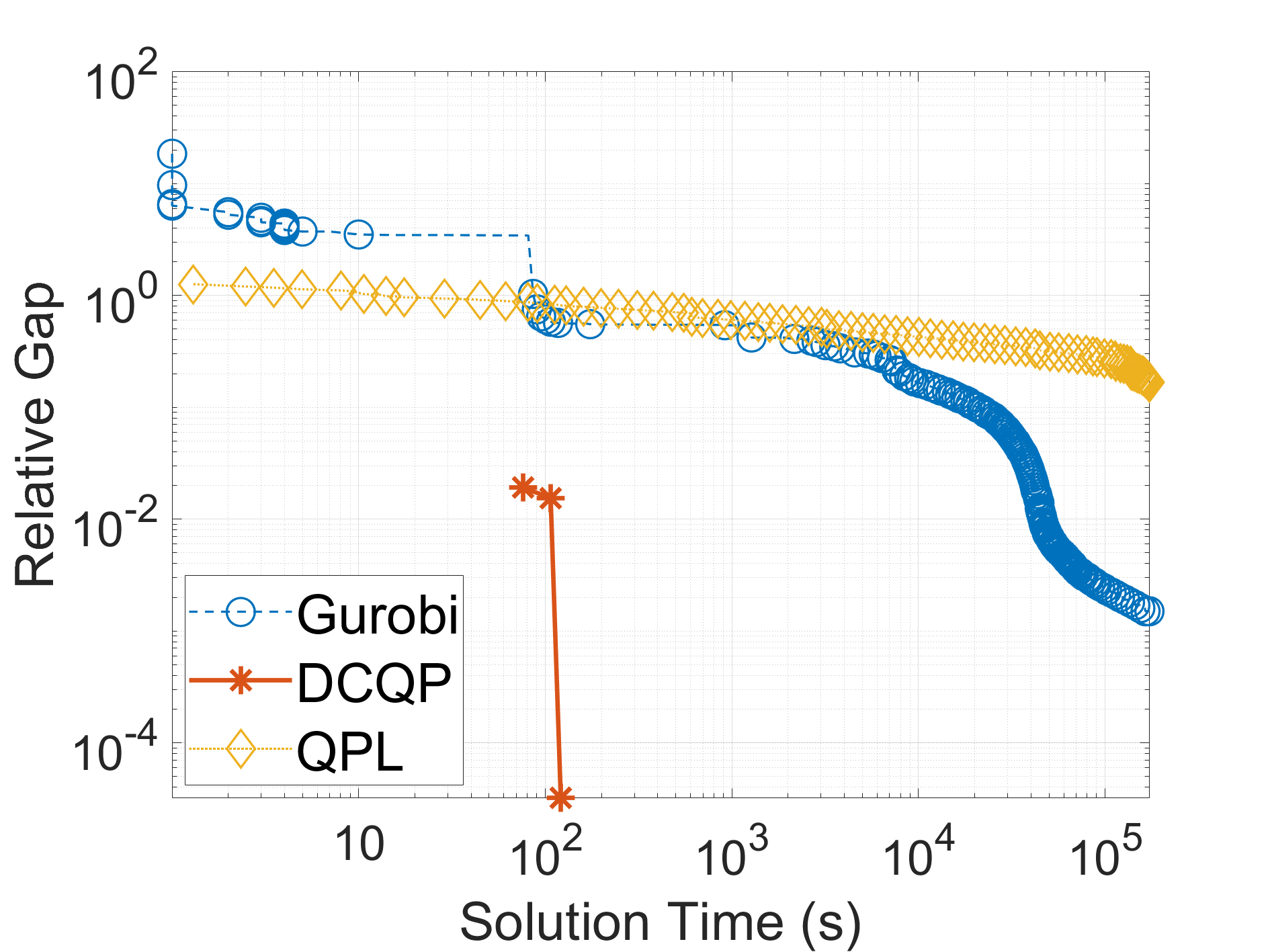}
        \caption{Instance 5 in \texttt{qp}\_\texttt{n}\_\texttt{0}\_\texttt{3}.}
  \label{subfig1:QC100-3}
    \end{subfigure}
    \begin{subfigure}[t]{0.45\textwidth}
        \includegraphics[width=\textwidth]{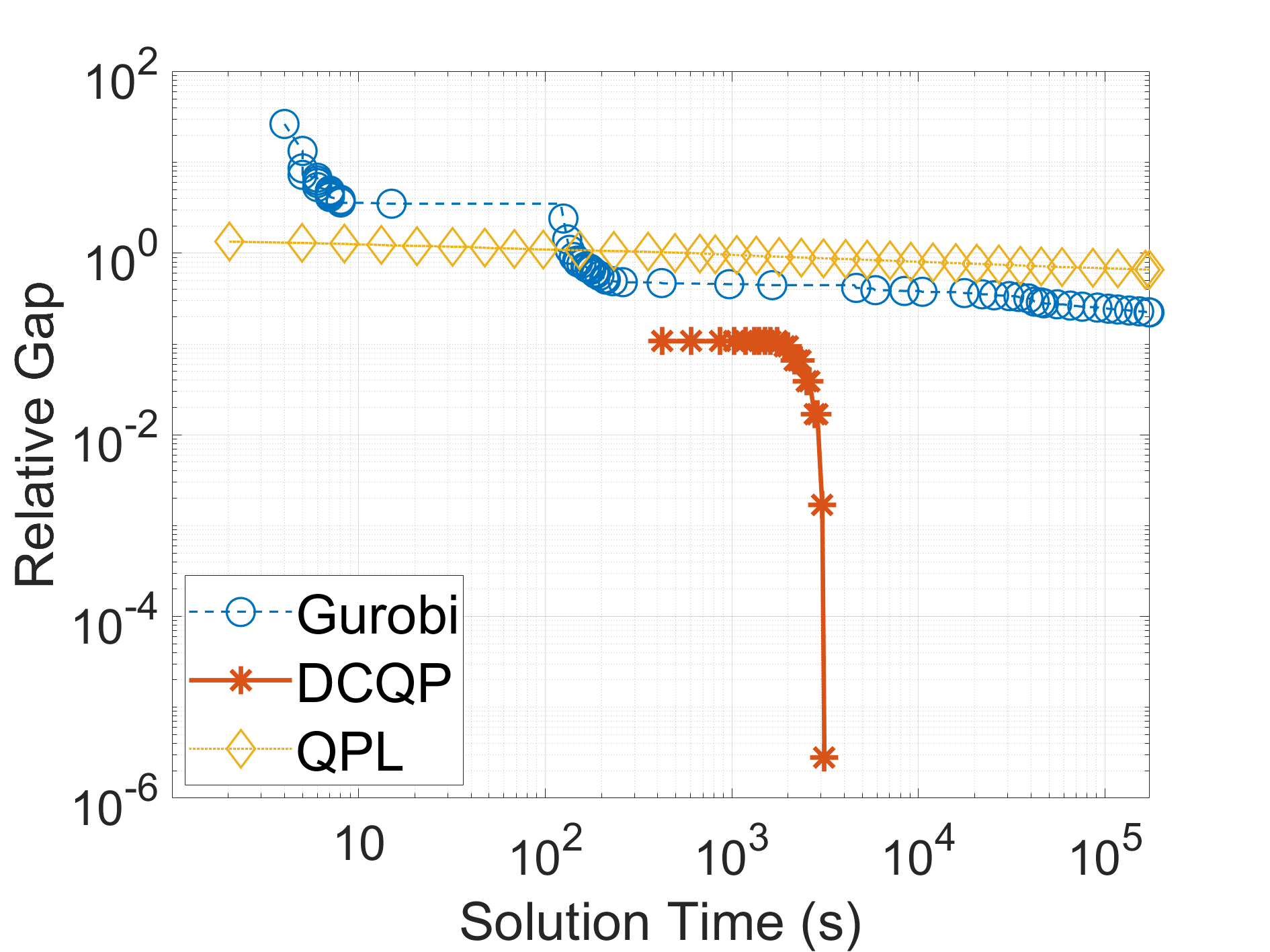}
        \caption{Instance 6 in \texttt{qp}\_\texttt{n}\_\texttt{0}\_\texttt{9}.}
    \label{subfig2:QC100-3}
    \end{subfigure}
    \hfill
    \begin{subfigure}[t]{0.45\textwidth}
        \includegraphics[width=\textwidth]{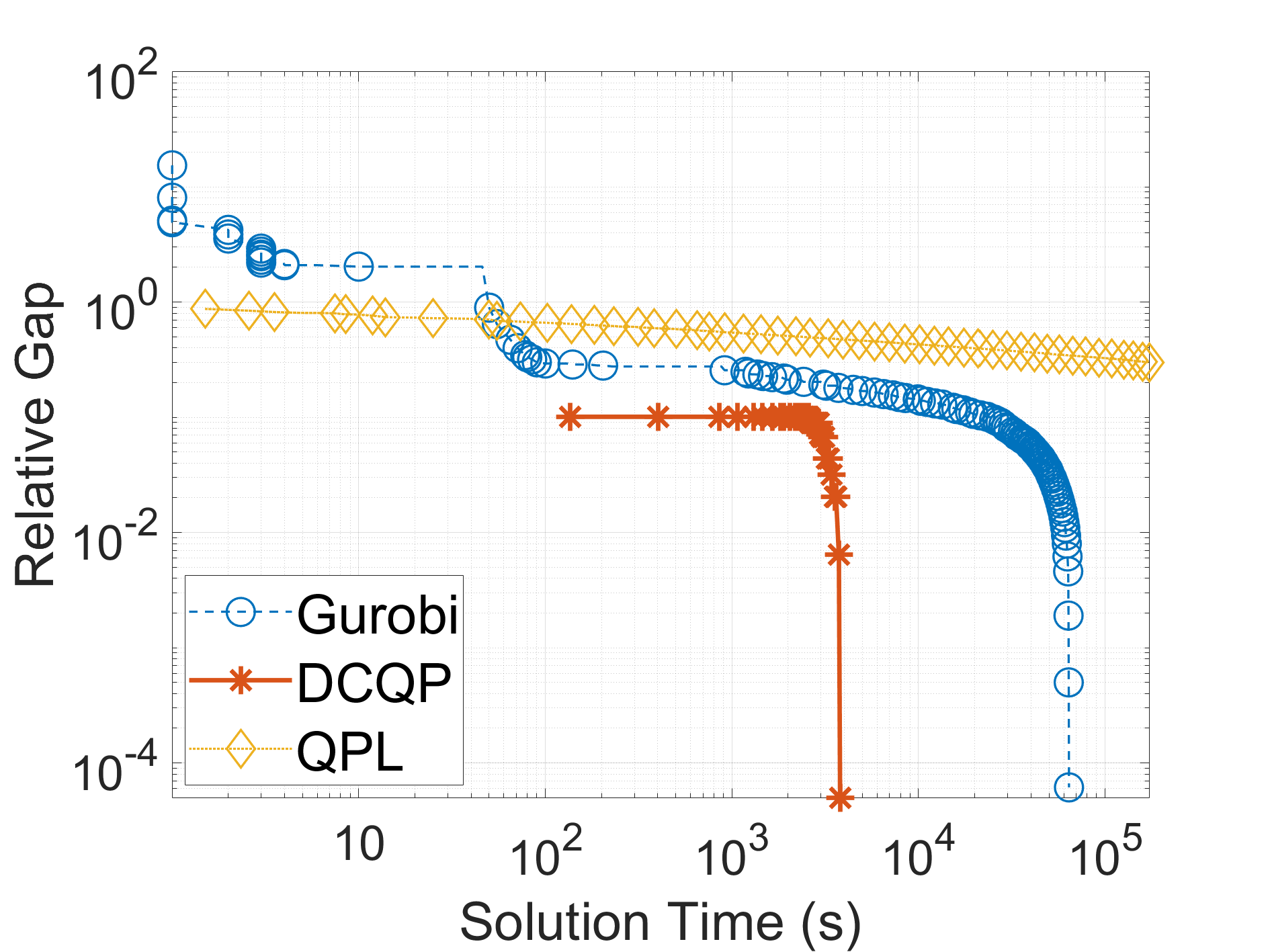}
        \caption{Instance 7 in \texttt{qp}\_\texttt{u}\_\texttt{0}\_\texttt{3}.}
   \label{subfig3:QC100-3}
    \end{subfigure}
    \begin{subfigure}[t]{0.45\textwidth}
        \includegraphics[width=\textwidth]{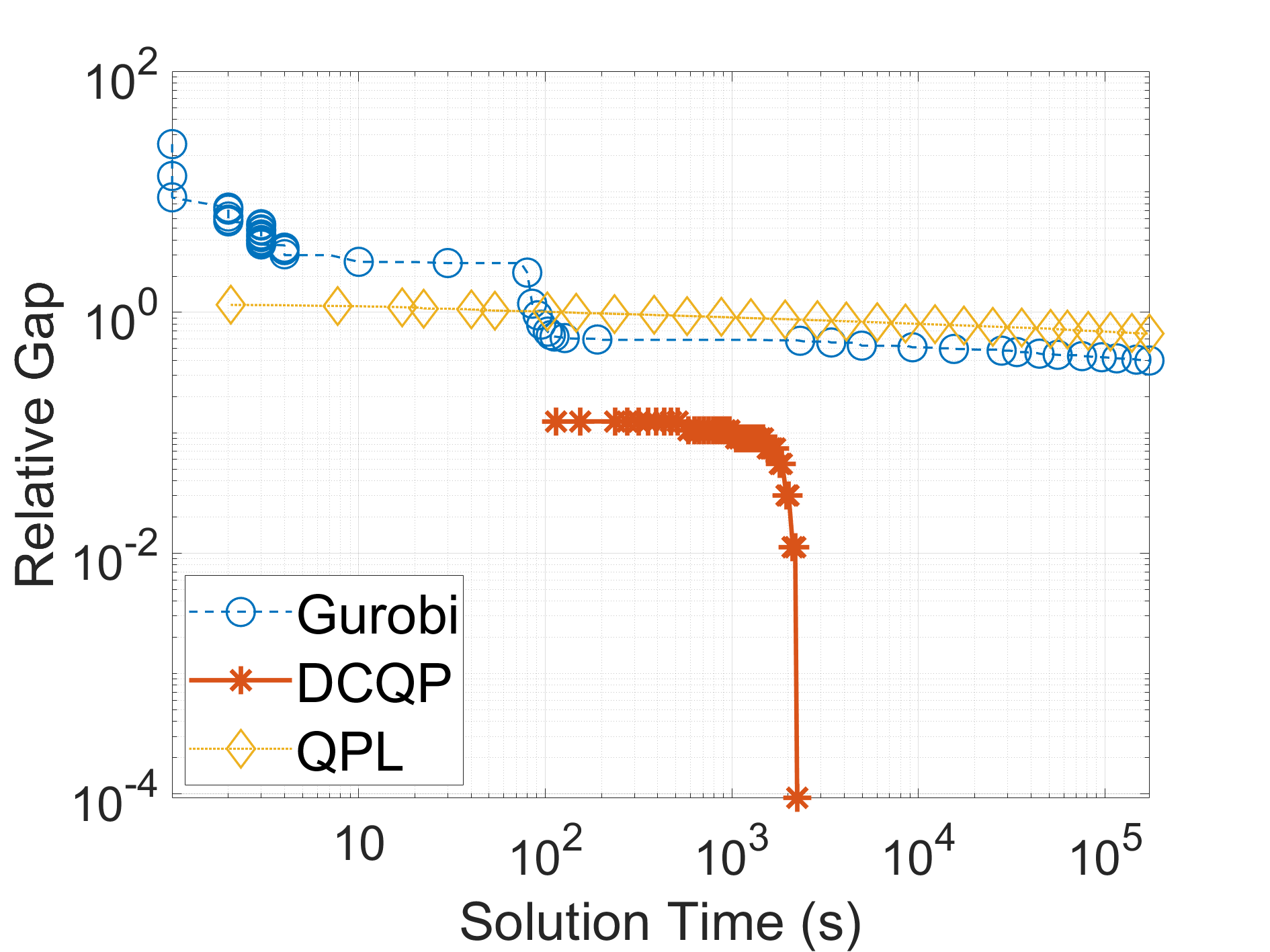}
        \caption{Instance 18 in \texttt{qp}\_\texttt{u}\_\texttt{0}\_\texttt{9}.}
   \label{subfig4:QC100-3}
    \end{subfigure}
    \caption{Comparison of the relative gap versus computational time for four selected instances under a 48-hour time limit.}
    \label{fig:comp48h}
\end{figure}

\subsection{Synthetic Instances with Equality Constraints}

In~\Cref{subsubsec:synnoe1h}, we observed that for sparse QP problems without equality constraints ($m_{eq}=0$), \texttt{Gurobi} and \texttt{QPL} are highly efficient and more advantageous than \texttt{DCQP}; additionally, the DNN bound appears to be mostly tight for such sparse instances. In this section, we show that the situation can completely change if we add some equality constraints while keeping the same sparsity level.

To illustrate this, we test the three methods on the group \texttt{qp\_u\_25\_1}, which has the lowest density level. Its generation follows exactly the same rule as for the group \texttt{qp\_u\_0\_1}, except that here we added 25 sparse equality constraints.

\subsubsection{Performance within One Hour}
We first examine the performance of the three methods for all instances within a time limit of 3600s. The percentage of instances in the group \texttt{qp\_u\_25\_1} solved by the three algorithms is plotted in~\Cref{fig:randomqp100_50_25_50_1_3600s}, and more computational details are given in~\Cref{table:randomqp100 50 25 50 1}. For this set of sparse QP problems containing equality constraints, neither \texttt{Gurobi} nor \texttt{QPL} succeeded in solving any case within the time limit of {3600 seconds}. In contrast, our \texttt{DCQP} successfully reduced the relative gap to $10^{-4}$ for 19 out of 20 instances in no more than \textbf{1500 seconds}.

These experiments indicate that, in the presence of equality constraints, both \texttt{Gurobi} and \texttt{QPL} may underperform on sparse instances, while our method consistently produces good results.

\begin{remark}\label{rem:instance11}
By examining~\Cref{table:randomqp100 50 25 50 1}, we find that instance~11 is the only case where our method, \texttt{DCQP}, did not solve successfully. On this instance, \texttt{DCQP} terminated after 432 seconds and returned a solution with a relative gap of $1.23 \times 10^{-4}$, which exceeds the prescribed threshold of $10^{-4}$. As discussed in~\Cref{rem:swrr}, the solution returned by~\Cref{alg:CuP-global} may occasionally fail to satisfy~\eqref{eq:relativgaplessepsiln} due to numerical errors in intermediate steps; this instance provides such an example. It is worth noting, however, that this is the only instance among all 204 tested where such behavior was observed.

\end{remark}
\begin{figure}[ht!]
    \centering
        \includegraphics[width=0.5\textwidth]{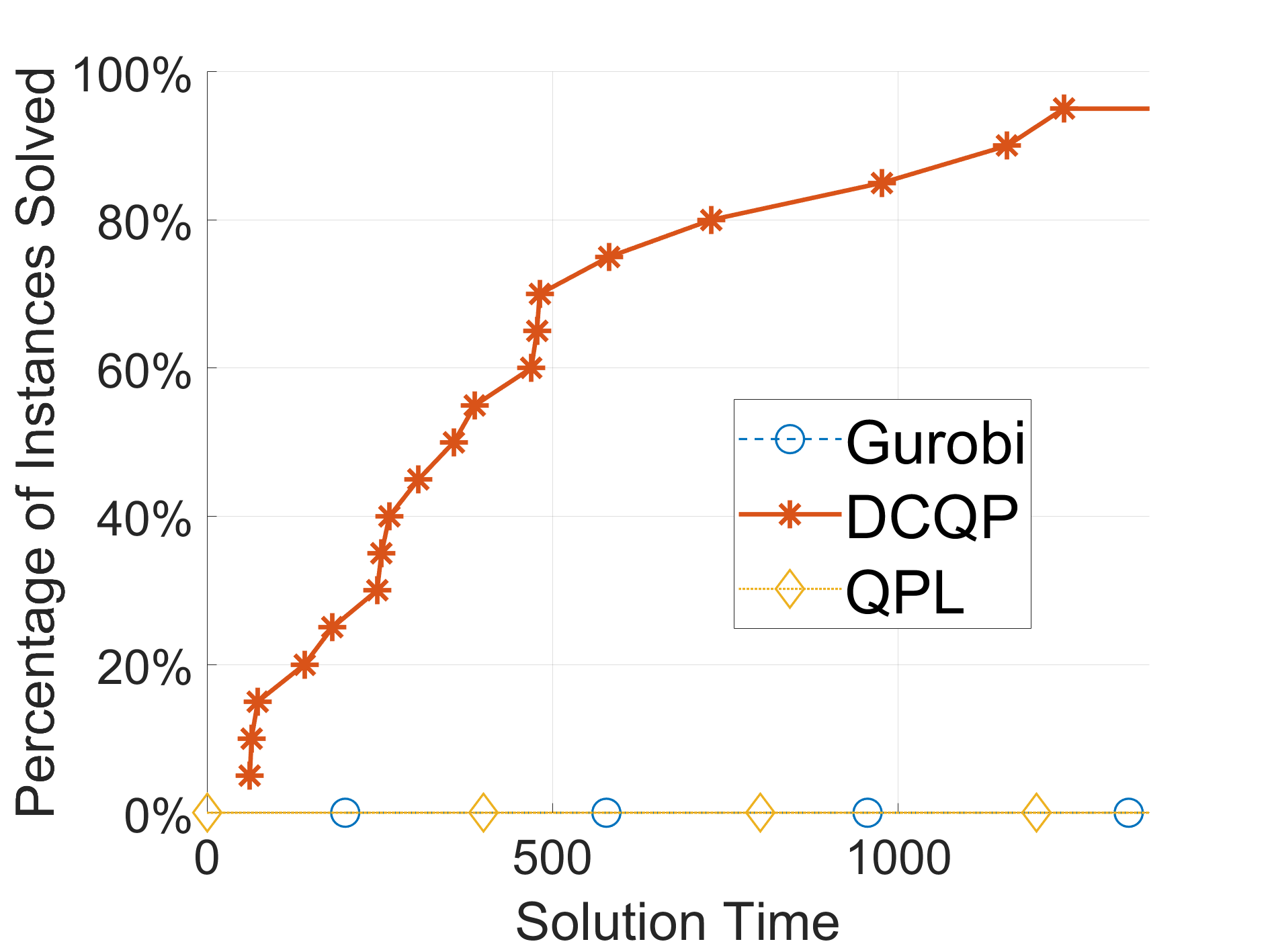}
        \caption{Comparison of the percentage of instances solved versus computational time, under a 3600-second time limit, for the 20 instances in the group \texttt{qp\_u\_25\_1}.}
        \label{fig:randomqp100_50_25_50_1_3600s}
\end{figure}

\subsubsection{Performance within 48 Hours}

We have extended the time limit to 48 hours for both \texttt{Gurobi} and \texttt{QPL} to evaluate their performance on instance 18 within the group \texttt{qp\_u\_25\_1}. As shown in~\Cref{fig:randomqp100_50_25_50_1_48h}, both \texttt{Gurobi} and \texttt{QPL} exhibit stagnation, with their relative gaps decreasing very slowly. Even after 48 hours, \texttt{Gurobi}'s relative gap remains around 1, while \texttt{QPL}'s is approximately 0.1. In contrast, our method successfully reduces the gap to $10^{-4}$ within just 1240 seconds.

\begin{figure}[ht!]
    \centering
        \includegraphics[width=0.5\textwidth]{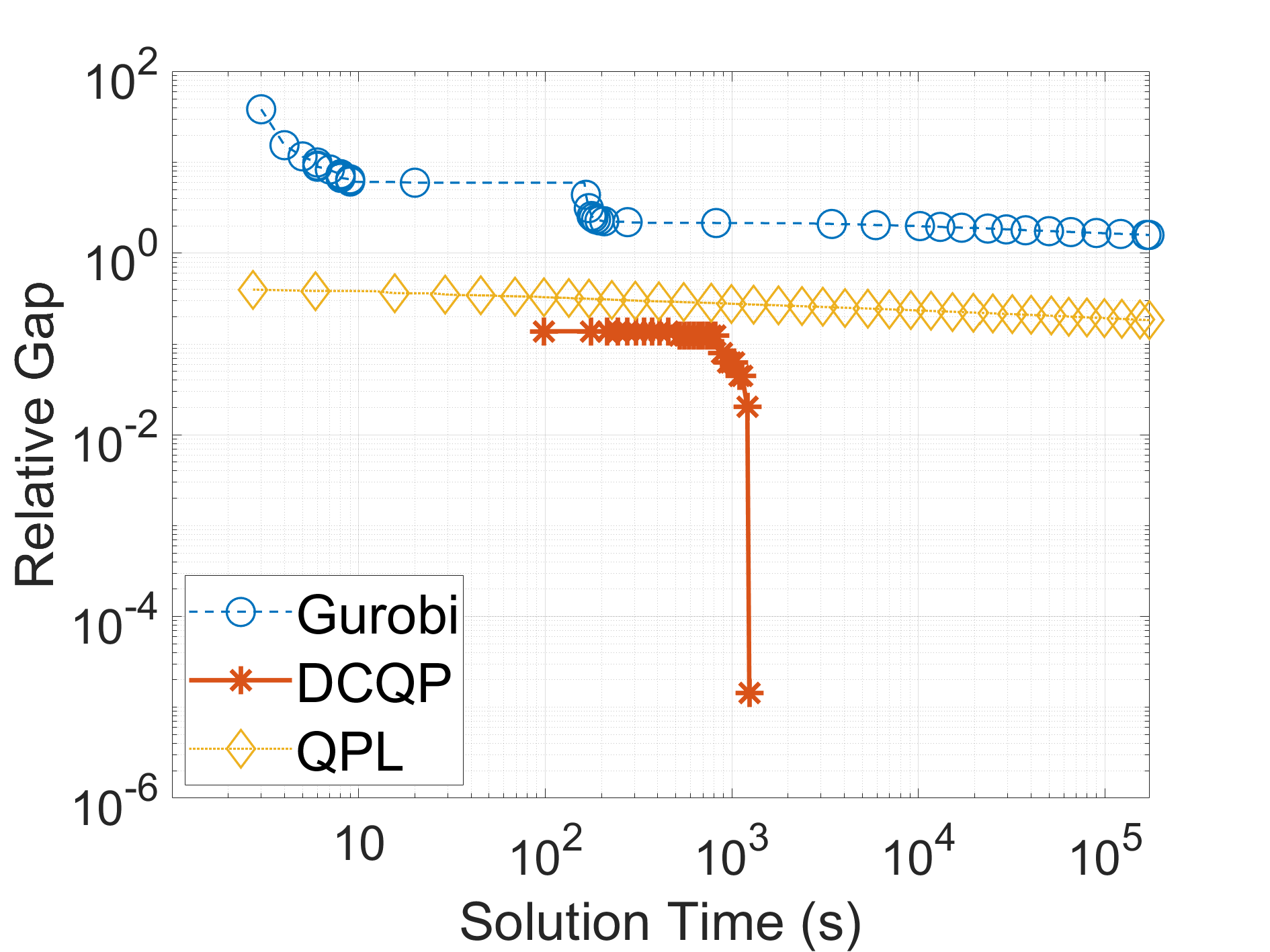}
        \caption{Comparison of the relative gap versus computational time for instance 18 in the group \texttt{qp\_u\_25\_3}.}
    \label{fig:randomqp100_50_25_50_1_48h}
\end{figure}

\section{Conclusion}\label{sec:conclusion}

In this paper, we introduced a new cutting plane framework for globally solving nonconvex QP problems. By developing both the theoretical foundation and practical algorithms for generating valid cuts via SDP relaxations, we addressed key challenges in tightening relaxation bounds and ensuring global convergence. Our finite-terminating local search method guarantees the identification of suitable KKT points for cut generation, and the overall algorithm achieves high-precision solutions across a wide range of problem instances.

Comprehensive computational results demonstrate that our approach consistently outperforms state-of-the-art commercial and academic QP solvers, especially on large-scale or difficult instances where existing methods often struggle. Notably, our algorithm is particularly effective in the presence of equality constraints or denser problem data, settings in which standard solvers may fail to deliver accurate solutions within practical time limits.

This work opens several avenues for future research, including the rigorous convergence analysis of the proposed cutting plane framework, the extension to more general classes of nonconvex quadratic or polynomial optimization problems, and the integration of advanced numerical techniques to further improve computational efficiency. Our results underscore the potential of SDP-based cutting plane methods as a powerful tool for global optimization in nonconvex quadratic programming.

\begin{table}[htbp!]
    \centering
    {\scriptsize
    \caption{Comparison of \texttt{DCQP}, \texttt{Gurobi}, and \texttt{QPL} on the 64 benchmark instances. The columns "{$\mathrm{Gap}_{\mathrm{initial}}$}", "{$\mathrm{Gap}_{\mathrm{final}}$}", and "{No. of cuts}" denote the initial relative gap, final relative gap, and the number of cuts added in \texttt{DCQP}, respectively. A hyphen ("-") indicates that the gap is less than $10^{-4}$. The shortest wall-clock time in each row is highlighted in bold font. Time is measured in seconds.}
    \label{table:randqp}
    \begin{tabular}{c|cccccccc}
    \toprule
&   \multicolumn{4}{c}{\texttt{DCQP}} & \multicolumn{2}{c}{\texttt{Gurobi}} & \multicolumn{2}{c}{\texttt{QPL}} \\
    \cmidrule(l{1pt}r{2pt}){2-5}
    \cmidrule(l{2pt}r{2pt}){6-7}
    \cmidrule(l{2pt}r{2pt}){8-9}
    Instance & $\mathrm{Gap}_{\mathrm{initial}}$ & No. of cuts & $\mathrm{Gap}_{\mathrm{final}}$ & Time & $\mathrm{Gap}_{\mathrm{final}}$ & Time & $\mathrm{Gap}_{\mathrm{final}}$ & Time  \\
\midrule
\texttt{qp20\_10\_1\_1} & - & 0 & - & 0.235 & - & \textbf{0.101} & - & 3.53 \\
\texttt{qp20\_10\_1\_2} & - & 0 & - & \textbf{0.195} & - & 0.62 & - & 4.36 \\
\texttt{qp20\_10\_1\_3} & 0.0488 & 3 & - & 0.868 & - & \textbf{0.199} & - & 0.69 \\
\texttt{qp20\_10\_1\_4} & 0.0587 & 1 & - & 0.471 & - & \textbf{0.255} & - & 0.93 \\
\texttt{qp20\_10\_2\_1} & - & 0 & - & 0.171 & - & \textbf{0.135} & - & 0.4 \\
\texttt{qp20\_10\_2\_2} & - & 0 & - & \textbf{0.189} & - & 1.45 & - & 6.2 \\
\texttt{qp20\_10\_2\_3} & - & 0 & - & \textbf{0.174} & - & 0.328 & - & 0.37 \\
\texttt{qp20\_10\_2\_4} & - & 0 & - & \textbf{0.163} & - & 0.284 & - & 2.6 \\
\texttt{qp20\_10\_3\_1} & 0.0221 & 3 & - & 0.907 & - & \textbf{0.215} & - & 2.04 \\
\texttt{qp20\_10\_3\_2} & - & 0 & - & \textbf{0.161} & - & 0.165 & - & 1.06 \\
\texttt{qp20\_10\_3\_3} & - & 0 & - & \textbf{0.148} & - & 0.262 & - & 0.58 \\
\texttt{qp20\_10\_3\_4} & - & 0 & - & \textbf{0.157} & - & 0.274 & - & 0.57 \\
\texttt{qp20\_10\_4\_1} & - & 0 & - & \textbf{0.178} & - & 0.32 & - & 1.75 \\
\texttt{qp20\_10\_4\_2} & - & 0 & - & 0.182 & - & 2.75 & - & \textbf{0.11} \\
\texttt{qp20\_10\_4\_3} & - & 0 & - & \textbf{0.169} & - & 0.378 & - & 0.51 \\
\texttt{qp20\_10\_4\_4} & - & 0 & - & \textbf{0.174} & - & 0.571 & - & 0.66 \\
\midrule
\texttt{qp30\_15\_1\_1} & - & 0 & - & 0.37 & - & 0.0732 & - & \textbf{0.02} \\
\texttt{qp30\_15\_1\_2} & 0.0016 & 1 & - & 0.85 & - & \textbf{0.256} & - & 1.91 \\
\texttt{qp30\_15\_1\_3} & - & 0 & - & \textbf{0.289} & - & 0.403 & - & 0.77 \\
\texttt{qp30\_15\_1\_4} & - & 0 & - & 0.315 & - & 0.679 & - & \textbf{0.14} \\
\texttt{qp30\_15\_2\_1} & - & 0 & - & \textbf{0.274} & - & 0.283 & - & 0.99 \\
\texttt{qp30\_15\_2\_2} & 0.0267 & 2 & - & \textbf{1.3} & - & 10.2 & - & 7.35 \\
\texttt{qp30\_15\_2\_3} & 0.3001 & 3 & - & \textbf{1.91} & - & 3.56 & - & 2.1 \\
\texttt{qp30\_15\_2\_4} & - & 0 & - & \textbf{0.306} & - & 0.587 & - & 12.9 \\
\texttt{qp30\_15\_3\_1} & 0.0441 & 3 & - & 1.65 & - & \textbf{0.859} & - & 6.04 \\
\texttt{qp30\_15\_3\_2} & - & 0 & - & \textbf{0.251} & - & 0.765 & - & 1.46 \\
\texttt{qp30\_15\_3\_3} & - & 0 & - & \textbf{0.297} & - & 50.6 & - & 0.33 \\
\texttt{qp30\_15\_3\_4} & - & 0 & - & \textbf{0.277} & - & 0.389 & - & 1.11 \\
\texttt{qp30\_15\_4\_1} & - & 0 & - & \textbf{0.268} & - & 1.96 & - & 0.54 \\
\texttt{qp30\_15\_4\_2} & 0.0128 & 1 & - & \textbf{0.946} & - & 1.22 & - & 1.8 \\
\texttt{qp30\_15\_4\_3} & - & 0 & - & \textbf{0.287} & - & 5.08 & - & 3.2 \\
\texttt{qp30\_15\_4\_4} & - & 0 & - & \textbf{0.278} & - & 0.657 & - & 1.11 \\
\midrule
\texttt{qp40\_20\_1\_1} & 0.0006 & 1 & - & 1.6 & - & \textbf{0.329} & - & 2.28 \\
\texttt{qp40\_20\_1\_2} & - & 0 & - & \textbf{0.557} & - & 0.638 & - & 42.1 \\
\texttt{qp40\_20\_1\_3} & 0.0084 & 1 & - & 1.65 & - & \textbf{0.393} & - & 1.97 \\
\texttt{qp40\_20\_1\_4} & - & 0 & - & \textbf{0.593} & - & 6.1 & - & 4.69 \\
\texttt{qp40\_20\_2\_1} & - & 0 & - & 0.527 & - & \textbf{0.456} & - & 0.63 \\
\texttt{qp40\_20\_2\_2} & - & 0 & - & \textbf{0.539} & - & 1490 & - & 10.6 \\
\texttt{qp40\_20\_2\_3} & 0.174 & 5 & - & \textbf{9.4} & - & 12.5 & - & 23.1 \\
\texttt{qp40\_20\_2\_4} & - & 0 & - & 0.659 & - & 0.0696 & - & \textbf{0.01} \\
\texttt{qp40\_20\_3\_1} & - & 0 & - & \textbf{0.449} & - & 1.06 & - & 1.99 \\
\texttt{qp40\_20\_3\_2} & - & 0 & - & \textbf{0.531} & - & 55.8 & - & 3.85 \\
\texttt{qp40\_20\_3\_3} & - & 0 & - & \textbf{0.544} & - & 2.51 & - & 1.42 \\
\texttt{qp40\_20\_3\_4} & - & 0 & - & \textbf{0.599} & - & 2.37 & - & 3.61 \\
\texttt{qp40\_20\_4\_1} & - & 0 & - & \textbf{0.555} & - & 37.2 & - & 7.81 \\
\texttt{qp40\_20\_4\_2} & 0.1039 & 7 & - & 8.23 & - & 8.63 & - & \textbf{6.56} \\
\texttt{qp40\_20\_4\_3} & - & 0 & - & \textbf{0.646} & - & 23.8 & - & 110 \\
\texttt{qp40\_20\_4\_4} & - & 0 & - & \textbf{0.622} & - & 18.5 & - & 15.2 \\
\midrule
\texttt{qp50\_25\_1\_1} & 0.0014 & 1 & - & 2.98 & - & \textbf{0.482} & - & 4.76 \\
\texttt{qp50\_25\_1\_2} & - & 0 & - & \textbf{0.909} & - & 2.95 & - & 5.64 \\
\texttt{qp50\_25\_1\_3} & 0.0082 & 2 & - & \textbf{11.4} & - & 35.3 & - & 42.8 \\
\texttt{qp50\_25\_1\_4} & - & 0 & - & \textbf{1.24} & - & 4.02 & - & 7.82 \\
\texttt{qp50\_25\_2\_1} & - & 0 & - & 0.94 & - & 0.0689 & - & \textbf{0.02} \\
\texttt{qp50\_25\_2\_2} & 0.0061 & 1 & - & \textbf{3.38} & - & 4.55 & - & 12.5 \\
\texttt{qp50\_25\_2\_3} & - & 0 & - & \textbf{1.1} & - & 8.38 & - & 2.86 \\
\texttt{qp50\_25\_2\_4} & - & 0 & - & \textbf{1.08} & - & 7.24 & - & 3.39 \\
\texttt{qp50\_25\_3\_1} & - & 0 & - & \textbf{1} & - & 2.45 & - & 73 \\
\texttt{qp50\_25\_3\_2} & - & 0 & - & \textbf{0.879} & - & 1.24 & - & 1.39 \\
\texttt{qp50\_25\_3\_3} & - & 0 & - & \textbf{1.03} & - & 86.1 & - & 1050 \\
\texttt{qp50\_25\_3\_4} & - & 0 & - & \textbf{1.22} & - & 26.1 & - & 3.45 \\
\texttt{qp50\_25\_4\_1} & - & 0 & - & \textbf{0.966} & - & 22.6 & - & 5.49 \\
\texttt{qp50\_25\_4\_2} & - & 0 & - & \textbf{1.05} & - & 8.95 & - & 7.42 \\
\texttt{qp50\_25\_4\_3} & 0.1213 & 13 & - & \textbf{30.2} & - & 42.3 & - & 65.9 \\
\texttt{qp50\_25\_4\_4} & - & 0 & - & \textbf{1.21} & - & 6.19 & - & 2.9 \\
\bottomrule
    \end{tabular}}
\end{table}

\begin{table}[htbp]
    \centering
    {\scriptsize
    \caption{Comparison of \texttt{DCQP}, \texttt{Gurobi} and \texttt{QPL} on  \texttt{qp}\_\texttt{n}\_\texttt{0}\_\texttt{1}, \texttt{qp}\_\texttt{n}\_\texttt{0}\_\texttt{3}, and \texttt{qp}\_\texttt{n}\_\texttt{0}\_\texttt{9} instances. 
   See the caption of~\Cref{table:randqp} for notation and formatting details.}
    \label{table:merged-normalrandomqp}
    \begin{tabular}{cccccccccc}
    \toprule
  \multirow{2}{*}{Group}  & \multirow{2}{*}{Index} & \multicolumn{4}{c}{\texttt{DCQP}} & \multicolumn{2}{c}{\texttt{Gurobi}} & \multicolumn{2}{c}{\texttt{QPL}} \\
    \cmidrule(l{1pt}r{2pt}){3-6}
    \cmidrule(l{2pt}r{2pt}){7-8}
    \cmidrule(l{2pt}r{2pt}){9-10}
      &   & $\mathrm{Gap}_{\mathrm{initial}}$ & No. of cuts & $\mathrm{Gap}_{\mathrm{final}}$ & Time & $\mathrm{Gap}_{\mathrm{final}}$ & Time & $\mathrm{Gap}_{\mathrm{final}}$ & Time  \\
    \midrule
    \multirow{20}{*}{\texttt{n}\_\texttt{0}\_\texttt{1}}
    & 1  & - & 0 & - & 15   & - & \textbf{2} & - & 5.03 \\
    & 2  & - & 0 & - & 27.5 & - & 3.74 & - & \textbf{1.97} \\
    & 3  & - & 0 & - & 28.5 & - & \textbf{1.85} & - & 1.86 \\
    & 4  & - & 0 & - & 28.6 & - & 1.79 & - & \textbf{1.66} \\
    & 5  & - & 0 & - & 26.7 & - & 1.96 & - & \textbf{1.59} \\
    & 6  & - & 0 & - & 28.6 & - & 5.11 & - & \textbf{1.99} \\
    & 7  & - & 0 & - & 27.8 & - & 2.28 & - & \textbf{2.26} \\
    & 8  & - & 0 & - & 28.6 & - & 1.60 & - & \textbf{1.54} \\
    & 9  & - & 0 & - & 26.6 & - & 2.00 & - & \textbf{1.38} \\
    & 10 & - & 0 & - & 26.3 & - & 2.05 & - & \textbf{1.48} \\
    & 11 & - & 0 & - & 31.3 & - & 4.08 & - & \textbf{3.11} \\
    & 12 & - & 0 & - & 33.0 & - & 4.40 & - & \textbf{4.17} \\
    & 13 & - & 0 & - & 31.1 & - & \textbf{1.96} & - & 2.51 \\
    & 14 & - & 0 & - & 27.9 & - & \textbf{2.12} & - & 2.62 \\
    & 15 & - & 0 & - & 26.4 & - & \textbf{2.28} & - & 2.81 \\
    & 16 & - & 0 & - & 28.2 & - & 3.65 & - & \textbf{3.35} \\
    & 17 & - & 0 & - & 27.2 & - & \textbf{1.94} & - & 3.05 \\
    & 18 & - & 0 & - & 27.1 & - & \textbf{1.82} & - & 3.05 \\
    & 19 & - & 0 & - & 30.9 & - & \textbf{1.92} & - & 4.20 \\
    & 20 & - & 0 & - & 27.4 & - & \textbf{1.79} & - & 2.59 \\
    \midrule
    \multirow{20}{*}{\texttt{n}\_\texttt{0}\_\texttt{3}}
    & 1  & -        & 0 & - & \textbf{32}   & -       & 846  & - & 61.7 \\
    & 2  & 0.0005 & 1 & - & \textbf{80.1} & 0.0016  & 3600 & 0.07 & 3600 \\
    & 3  & -        & 0 & - & \textbf{35}   & -       & 494  & - & 113 \\
    & 4  & -        & 0 & - & \textbf{34.5} & 0.0007  & 3600 & - & 744 \\
    & 5  & 0.0192 & 2 & - & \textbf{121}  & 0.3511  & 3600 & 0.49 & 3600 \\
    & 6  & -        & 0 & - & 33.3          & -       & 284  & - & \textbf{25.3} \\
    & 7  & -        & 0 & - & \textbf{35.5} & -       & 999  & - & 2000 \\
    & 8  & -        & 0 & - & \textbf{41.6} & 0.0463  & 3600 & 0.32 & 3600 \\
    & 9  & -        & 0 & - & \textbf{35.2} & -       & 1230 & - & 1030 \\
    & 10 & -        & 0 & - & \textbf{34.2} & -       & 378  & - & 36.6 \\
    & 11 & 0.0003 & 1 & - & \textbf{89}   & 0.0123  & 3600 & 0.20 & 3600 \\
    & 12 & 0.0008 & 1 & - & 72.5          & -       & 271  & - & \textbf{18.9} \\
    & 13 & 0.0002 & 1 & - & \textbf{63.8} & -       & 1170 & - & 280 \\
    & 14 & -        & 0 & - & \textbf{35.6} & -       & 581  & - & 66.9 \\
    & 15 & -        & 0 & - & \textbf{38.9} & 0.0014  & 3600 & - & 2630 \\
    & 16 & 0.0057 & 2 & - & 112           & -       & 400  & - & \textbf{109} \\
    & 17 & -        & 0 & - & 30.5          & -       & 448  & - & \textbf{26.4} \\
    & 18 & -        & 0 & - & \textbf{31.2} & -       & 1710 & - & 122 \\
    & 19 & 0.0197 & 2 & - & \textbf{134}  & -       & 1090 & - & 497 \\
    & 20 & -        & 0 & - & \textbf{35.7} & -       & 727  & - & 1160 \\
    \midrule
    \multirow{20}{*}{\texttt{n}\_\texttt{0}\_\texttt{9}}
    & 1  & 0.0353 & 3 & - & \textbf{448}  & 0.1831 & 3600 & 0.42 & 3600 \\
    & 2  & 0.0471 & 3 & - & \textbf{666}  & 0.2402 & 3600 & 0.46 & 3600 \\
    & 3  & 0.0027 & 2 & - & \textbf{757}  & 0.155 & 3600 & 0.37 & 3600 \\
    & 4  & 0.0121 & 3 & - & \textbf{1140} & 0.1504 & 3600 & 0.41 & 3600 \\
    & 5  & 0.0862 & 8 & - & \textbf{1340} & 0.3448 & 3600 & 0.67 & 3600 \\
    & 6  & 0.1072 & 20& - & \textbf{3130} & 0.4448 & 3600 & 0.9 & 3600 \\
    & 7  & 0.003 & 2 & - & \textbf{823}  & 0.0168 & 3600 & 0.2 & 3600 \\
    & 8  & 0.0722 & 7 & - & \textbf{1310} & 0.2461 & 3600 & 0.51 & 3600 \\
    & 9  & -        & 0 & - & \textbf{134}  & -    & 3220 & -   & 2130 \\
    & 10 & 0.0782 & 10& - & \textbf{1750} & 0.2882 & 3600 & 0.65 & 3600 \\
    & 11 & 0.0518 & 6 & - & \textbf{1090} & 0.3042 & 3600 & 0.72 & 3600 \\
    & 12 & -        & 0 & - & \textbf{131}  & 0.1287 & 3600 & 0.29 & 3600 \\
    & 13 & 0.1392 & 4 & - & \textbf{827}  & 0.2159 & 3600 & 0.62 & 3600 \\
    & 14 & 0.0158  & 2 & - & \textbf{598}  & 0.1742 & 3600 & 0.43 & 3600 \\
    & 15 & 0.009 & 3 & - & \textbf{967}  & 0.1386 & 3600 & 0.38 & 3600 \\
    & 16 & 0.0545 & 5 & - & \textbf{1330} & 0.2634 & 3600 & 0.59 & 3600 \\
    & 17 & 0.0673 & 8 & - & \textbf{1700} & 0.3363 & 3600 & 0.58 & 3600 \\
    & 18 & 0.0016 & 1 & - & \textbf{347}  & 0.1565 & 3600 & 0.45 & 3600 \\
    & 19 & 0.0266 & 2 & - & \textbf{523}  & 0.3057 & 3600 & 0.55 & 3600 \\
    & 20 & 0.0013 & 1 & - & \textbf{426}  & 0.1166 & 3600 & 0.5 & 3600 \\
    \bottomrule
    \end{tabular}}
\end{table}

\begin{table}[htbp]
    \centering
    {\scriptsize
    \caption{Comparison of \texttt{DCQP}, \texttt{Gurobi} and \texttt{QPL} on different \texttt{qp}\_\texttt{u}\_\texttt{0}\_\texttt{1}, \texttt{qp}\_\texttt{u}\_\texttt{0}\_\texttt{3}, and \texttt{qp}\_\texttt{u}\_\texttt{0}\_\texttt{9} instances. See the caption of~\Cref{table:randqp} for notation and formatting details.}
    \label{table:merged-randomqp}
    \begin{tabular}{cccccccccc}
    \toprule
  \multirow{2}{*}{Group}  & \multirow{2}{*}{Index} & \multicolumn{4}{c}{\texttt{DCQP}} & \multicolumn{2}{c}{\texttt{Gurobi}} & \multicolumn{2}{c}{\texttt{QPL}} \\
    \cmidrule(l{1pt}r{2pt}){3-6}
    \cmidrule(l{2pt}r{2pt}){7-8}
    \cmidrule(l{2pt}r{2pt}){9-10}
      &  & $\mathrm{Gap}_{\mathrm{initial}}$ & No. of cuts & $\mathrm{Gap}_{\mathrm{final}}$ & Time & $\mathrm{Gap}_{\mathrm{final}}$ & Time & $\mathrm{Gap}_{\mathrm{final}}$ & Time  \\
    \midrule
    \multirow{20}{*}{\texttt{u}\_\texttt{0}\_\texttt{1}}&
1 & 0.0008 & 1 & - & \textbf{68.5} & - & 2190 & - & 69.4 \\
& 2 & 0.0021 & 1 & - & \textbf{164} & - & 243 & - & 208 \\
& 3 & - & 0 & - & 49.2 & - & 91.5 & - & \textbf{4.54} \\
& 4 & - & 0 & - & 45.8 & - & 120 & - & \textbf{5.95} \\
& 5 & - & 0 & - & 48 & - & 133 & - & \textbf{9.21} \\
& 6 & - & 0 & - & 38.5 & - & 77.5 & - & \textbf{3.69} \\
& 7 & - & 0 & - & 45.5 & - & 19.5 & - & \textbf{5.22} \\
& 8 & - & 0 & - & \textbf{45.8} & - & 130 & - & 91 \\
& 9 & - & 0 & - & 47.8 & - & 12 & - & \textbf{4.35} \\
& 10 & - & 0 & - & 42.7 & - & 98.8 & - & \textbf{4.48} \\
& 11 & - & 0 & - & 45.9 & - & 95.5 & - & \textbf{6.72} \\
& 12 & - & 0 & - & 46.2 & - & 23.9 & - & \textbf{3.63} \\
& 13 & - & 0 & - & 45.2 & - & 25.6 & - & \textbf{3.93} \\
& 14 & - & 0 & - & 38.6 & - & 16.5 & - & \textbf{3.06} \\
& 15 & 0.0006 & 1 & - & 167 & - & 186 & - & \textbf{22.9} \\
& 16 & - & 0 & - & 37.8 & - & 77.5 & - & \textbf{3.96} \\
& 17 & - & 0 & - & 42.2 & - & 74.8 & - & \textbf{4.27} \\
& 18 & - & 0 & - & 45.2 & - & 269 & - & \textbf{35} \\
& 19 & - & 0 & - & 42.3 & - & 18 & - & \textbf{2.34} \\
& 20 & - & 0 & - & 43.9 & - & 24.3 & - & \textbf{3.62} \\
    \midrule
    \multirow{20}{*}{\texttt{u}\_\texttt{0}\_\texttt{3}}
&1 & 0.0106 & 2 & - & \textbf{307} & 0.038 & 3600 & 0.4 & 3600 \\
&2 & - & 0 & - & \textbf{47.3} & - & 392 & - & 816 \\
&3 & 0.0103 & 2 & - & \textbf{196} & - & 3430 & 0.4 & 3600 \\
&4 & 0.0238 & 3 & - & \textbf{204} & 0.175 & 3600 & 0.4 & 3600 \\
&5 & - & 0 & - & \textbf{52} & - & 190 & - & 447 \\
&6 & 0.1583 & 12 & - & \textbf{644} & 0.195 & 3600 & 0.5 & 3600 \\
&7 & 0.1895 & 38 & 0.02 & {3600} & 0.163 & {3600} & 0.5 & {3600} \\
&8 & 0.0358 & 8 & - & \textbf{802} & 0.18 & 3600 & 0.4 & 3600 \\
&9 & 0.0027 & 1 & - & \textbf{254} & - & 296 & - & 481 \\
&10 & 0.0012 & 1 & - & \textbf{362} & - & 3550 & 0.2 & 3600 \\
&11 & 0.0055 & 2 & - & \textbf{463} & - & 2230 & 0.2 & 3600 \\
&12 & 0.0046 & 1 & - & \textbf{355} & 0.014 & 3600 & 0.4 & 3600 \\
&13 & - & 0 & - & \textbf{95.3} & - & 1930 & 0.3 & 3600 \\
&14 & 0.0114 & 2 & - & \textbf{457} & 0.065 & 3600 & 0.4 & 3600 \\
&15 & 0.0262 & 2 & - & \textbf{443} & - & 1840 & 0.4 & 3600 \\
&16 & 0.0118 & 2 & - & \textbf{481} & 0.002 & 3600 & 0.4 & 3600 \\
&17 & 0.1175 & 8 & - & \textbf{771} & 0.201 & 3600 & 0.6 & 3600 \\
&18 & 0.004 & 1 & - & \textbf{129} & 0.067 & 3600 & 0.2 & 3600 \\
&19 & - & 0 & - & \textbf{44.2} & - & 2190 & 0.2 & 3600 \\
&20 & 0.0103 & 1 & - & \textbf{154} & - & 2820 & 0.2 & 3600 \\
    \midrule
    \multirow{20}{*}{\texttt{u}\_\texttt{0}\_\texttt{9}}
& 1 & - & 0 & - & \textbf{37.4} & 0.123 & 3600 & 0.4 & 3600 \\
&2 & 0.0672 & 10 & - & \textbf{776} & 0.246 & 3600 & 0.5 & 3600 \\
&3 & 0.1666 & 18 & - & \textbf{1090} & 0.486 & 3600 & 0.5 & 3600 \\
&4 & 0.3822 & 18 & - & \textbf{1050} & 0.549 & 3600 & 0.7 & 3600 \\
&5 & 0.0479 & 6 & - & \textbf{527} & 0.426 & 3600 & 0.7 & 3600 \\
&6 & 0.0138 & 2 & - & \textbf{343} & 0.187 & 3600 & 0.5 & 3600 \\
&7 & 0.2859 & 11 & - & \textbf{743} & 0.711 & 3600 & 0.7 & 3600 \\
&8 & - & 0 & - & \textbf{79.3} & 0.055 & 3600 & 0.3 & 3600 \\
&9 & 0.0034 & 1 & - & \textbf{195} & - & 3250 & 0.3 & 3600 \\
&10 & 0.0308 & 3 & - & \textbf{370} & 0.277 & 3600 & 0.6 & 3600 \\
&11 & 0.0127 & 2 & - & \textbf{277} & 0.259 & 3600 & 0.6 & 3600 \\
&12 & 0.0794 & 10 & - & \textbf{1060} & 0.255 & 3600 & 0.5 & 3600 \\
&13 & 0.0228 & 3 & - & \textbf{392} & 0.179 & 3600 & 0.4 & 3600 \\
&14 & 0.1236 & 21 & - & \textbf{1190} & 0.467 & 3600 & 0.9 & 3600 \\
&15 & 0.1257 & 25 & - & \textbf{1760} & 0.363 & 3600 & 0.7 & 3600 \\
&16 & 0.0077 & 1 & - & \textbf{193} & 0.169 & 3600 & 0.4 & 3600 \\
&17 & 0.0023 & 1 & - & \textbf{233} & - & 3160 & 0.3 & 3600 \\
&18 & 0.1586 & 40 & - & \textbf{2240} & 0.586 & 3600 & 0.9 & 3600 \\
&19 & 0.0364 & 5 & - & \textbf{635} & 0.274 & 3600 & 0.5 & 3600 \\
&20 & 0.0359 & 3 & - & \textbf{335} & 0.272 & 3600 & 0.5 & 3600 \\
    \bottomrule
    \end{tabular}}
\end{table}

\begin{table}[htbp]
    \centering
    {\scriptsize
    \caption{Comparison of \texttt{DCQP}, \texttt{Gurobi} and \texttt{QPL} on \texttt{qp}\_\texttt{u}\_\texttt{25}\_\texttt{1} instances. See the caption of~\Cref{table:randqp} for notation and formatting details.}
    \label{table:randomqp100 50 25 50 1}
    \begin{tabular}{cccccccccc}
    \toprule
\multirow{2}{*}{Group}& \multirow{2}{*}{Index}  &  \multicolumn{4}{c}{\texttt{DCQP}} & \multicolumn{2}{c}{\texttt{Gurobi}} & \multicolumn{2}{c}{\texttt{QPL}} \\
    \cmidrule(l{1pt}r{2pt}){3-6}
    \cmidrule(l{2pt}r{2pt}){7-8}
    \cmidrule(l{2pt}r{2pt}){8-10}
    &  & $\mathrm{Gap}_{\mathrm{initial}}$ & No. of cuts & $\mathrm{Gap}_{\mathrm{final}}$ & Time & $\mathrm{Gap}_{\mathrm{final}}$ & Time & $\mathrm{Gap}_{\mathrm{final}}$ & Time  \\
\midrule
 \multirow{20}{*}{\texttt{u}\_\texttt{25}\_\texttt{1}} &1 & 0.0198 & 3 & - & \textbf{306} & 0.8 & 3600 & 0.8 & 3600 \\
& 2 & 0.0569 & 8 & - & \textbf{477} & 1.3 & 3600 & 0.9 & 3600 \\
& 3 & 0.143 & 21 & - & \textbf{977} & 1.6 & 3600 & 1 & 3600 \\
& 4 & 0.0001 & 1 & - & \textbf{181} & 0.8 & 3600 & 0.9 & 3600 \\
& 5 & 0.0529 & 5 & - & \textbf{357} & 1.9 & 3600 & 1 & 3600 \\
& 6 & - & 0 & - & \textbf{73.6} & 0.6 & 3600 & 0.8 & 3600 \\
& 7 & 0.1618 & 27 & - & \textbf{1160} & 2.3 & 3600 & 1 & 3600 \\
& 8 & - & 0 & - & \textbf{62.1} & 0.6 & 3600 & 0.6 & 3600 \\
& 9 & 0.0404 & 5 & - & \textbf{582} & 1.3 & 3600 & 0.7 & 3600 \\
& 10 & 0.0833 & 7 & - & \textbf{388} & 2.4 & 3600 & 0.8 & 3600 \\
& 11 & 0.0644 & 4 & 0.000123 & {431} & 1.4 & 3600 & 0.9 & 3600 \\
& 12 & 0.0447 & 5 & - & \textbf{469} & 1.5 & 3600 & 0.9 & 3600 \\
& 13 & 0.0004 & 1 & - & \textbf{141} & 0.8 & 3600 & 1 & 3600 \\
& 14 & 0.0432 & 5 & - & \textbf{481} & 1 & 3600 & 0.9 & 3600 \\
& 15 & 0.0927 & 12 & - & \textbf{730} & 1.1 & 3600 & 1 & 3600 \\
& 16 & 0.0172 & 2 & - & \textbf{246} & 1.8 & 3600 & 0.9 & 3600 \\
& 17 & 0.0132 & 2 & - & \textbf{264} & 1.8 & 3600 & 1 & 3600 \\
& 18 & 0.2326 & 26 & - & \textbf{1240} & 2.1 & 3600 & 0.8 & 3600 \\
& 19 & - & 0 & - & \textbf{64.4} & 0.8 & 3600 & 1 & 3600 \\
& 20 & 0.0068 & 2 & - & \textbf{252} & 1.1 & 3600 & 1 & 3600 \\
\bottomrule
    \end{tabular}}
\end{table}

\section*{Acknowledgments}
The authors would like to thank  Chiyu Ma for his advice in coding.

\newpage
\appendix

\section{Proof of~\Cref{thm:STc}}\label{sec-appendix:proofofthmSTc}

In this section we present the proof of~\Cref{thm:STc}. The following notations are used throughout. For any matrix $X$, $X_i$ denotes its $i$-th column vector; $X_{p:q}$ denotes the submatrix formed by rows $p$ through $q$; and $X_{p:q,r:s}$ denotes the submatrix corresponding to the intersection of rows $p$ to $q$ and columns $r$ to $s$. For any vector $x$, $x_{p:q}$ denotes the subvector containing elements with indices from $p$ to $q$.

Throughout this section, $\bar x$ is a KKT point of~\eqref{eq:QP} such that $Q|_{\cH_{I_{\bar x}}}\succ 0$. The values $\nu_{_R}\leq \Phi(\bar x)$ and $ \beta \in [0, \Phi(\bar x)-\nu_{_R}] $ are given.

\subsection{Change of Coordinate}
Before presenting the proof of~\Cref{thm:STc} in~\Cref{subsec:proofmainthm}, we first perform an appropriate change of coordinates. Since at least one eigenvalue of $Q$ is negative and $Q|_{H_{I_{\bar x}}}\succ 0$, we have $I_{\bar x}\neq \emptyset$. 
By permuting the rows of $A$, we assume without loss of generality that there is $k\in [n]$ such that  $\bar \lambda \geq 0$, $\bar \lambda_{k+1:m}=\bzero$ and
\begin{equation}\label{eq:KKTI}
\cH_{I_{\bar x}}=\cH_{[k]},\enspace
A_{1:k} \bar x =b_{1:k},\enspace Q\bar x+d ={-}\sum_{i=1}^k \bar \lambda_i a_i, \enspace 
\rank(A_{1:n})=n.
\end{equation}
We then make the following change of coordinate:
\begin{align}\label{a:cc}
y=b_{1:n}-A_{1:n}x,
\end{align}
so that~\eqref{eq:validcut} is equivalent to
\begin{equation}\label{eq:validcut-y}
\begin{aligned}
    \nu_{_R} \leq \min_{y\in \R^{n}} \quad &   y^\top R y + 2 p^\top y+r   \\
        \textrm{\rm s.t.} \quad &  Fy\leq w\\
        & \displaystyle  y\geq 0\\
        & \theta^\top (y-\bar y) \leq 1
    \end{aligned}
\end{equation}
where
\begin{subequations}\label{sub:ccor}
\begin{align}
R&:=(A_{1:n}^{-1})^\top Q (A_{1:n}^{-1}),\enspace p:=-(A_{1:n}^{-1})^\top d-(A_{1:n}^{-1})^\top QA_{1:n}^{-1}b_{1:n},\label{eq:p}
\\ 
F&:= -A_{n+1:m} A_{1:n}^{-1}
,\enspace 
w:= b_{n+1:m}-A_{n+1:m}A_{1:n}^{-1}b_{1:n},\label{eq:w}\\
r&:= \left(Q A_{1:n}^{-1} b_{1:n} + 2d\right)^\top A_{1:n}^{-1} b_{1:n}, \enspace \theta:=-(A_{1:n}^{-1})^\top c, \enspace \bar y:=b_{1:n}-A_{1:n}\bar x.\label{eq:c}
\end{align}
\end{subequations}
It is important to note that 
$$
\bar y_1=\cdots=\bar y_k=0.
$$

In the following two subsections, we provide different ways to generate a vector $\theta$  that satisfies~\eqref{eq:validcut-y}. Note that the vector $\theta$ is in one-to-one correspondence with the cut $c$. Any $\theta$ that satisfies~\eqref{eq:validcut-y} results in a valid cut $c=-A_{1:n}^\top \theta$. The generation of the vector $\theta$, along with the associated technical results, forms the core components of the proof of~\Cref{thm:STc}.

\subsection{Some Relations}
Let us first derive some basic properties from~\eqref{eq:KKTI} and the fact that $Q|_{\cH_{I_{\bar x}}}\succ 0$.
When $k<n$,
we partition $R$, $p$ and $F$ into blocks:
\begin{align}\label{a:pwdUb}
R=\begin{pmatrix}
R_{11} & ~R_{12} \\
R_{12}^\top & ~R_{22}
\end{pmatrix}, \enspace p=\begin{pmatrix}
 p_1 \\ p_2
\end{pmatrix},\enspace F=\begin{pmatrix}F_1 &~ F_2\end{pmatrix},
\end{align}
where $R_{11}\in \cS^k$, $R_{22}\in \cS^{n-k}$, $R_{12}\in \R^{k\times (n-k)}$, $p_1\in \R^k$, $p_2\in \R^{n-k}$,
$F_1 \in \R^{(m-n)\times k}$ and $F_2 \in \R^{(m-n)\times (n-k)}$. We first make a simple yet important observation. 
\begin{lemma}
The condition $Q|_{\cH_{I_{\bar x}}}\succ 0$ directly translates into $R_{22}\succ 0$. 
\end{lemma}
\begin{proof}
Let $y\neq 0$ satisfying $y_1=\cdots=y_k=0$. By the definition of $R$ in~\eqref{eq:p}, 
$$
y^\top Ry = (A_{1:n}^{-1} y)^\top Q (A_{1:n}^{-1} y).
$$
Let  $x=A_{1:n}^{-1} y$. If $y_1=\cdots=y_k=0$, then $A_{1:k} x=0$ and $x\in \cH_{I_{\bar x}}$ because $\cH_{I_{\bar x}}=\cH_{[k]}$. Hence, if $Q|_{\cH_{I_{\bar x}}}\succ 0$, we must have $y^\top R y>0$.

\end{proof}
Given that $R_{22}\succ 0$ when $k<n$, we define the matrix $D\in \cS_n$, the vector $q\in \R^n$ and the scalar $\upsilon\in \R$ as follows:
\begin{align}\label{eq:D}
D := \left\{\begin{array}{ll}
\begin{pmatrix}
    R_{11}-R_{12}R_{22}^{-1}R_{12}^\top & \bzero\\
    \bzero^\top
    & \bzero
\end{pmatrix},& \enspace\mathrm{~if~} k<n, \\
R, & \enspace\mathrm{~if~} k=n,
\end{array}
\right.
\end{align}
\begin{align}\label{eq:q}
q := \left\{\begin{array}{ll}
\begin{pmatrix}
    p_1-R_{12}R_{22}^{-1}p_2 \\
    \bzero
\end{pmatrix},& \enspace\mathrm{~if~} k<n ,\\
p, & \enspace\mathrm{~if~} k=n,
\end{array}
\right.
\end{align}
and
\begin{align}\label{eq:upsilon}
\upsilon := \left\{\begin{array}{ll}
    r - p_2^\top R_{22}^{-1}p_2~,
& \enspace\mathrm{~if~} k<n, \\
r~, & \enspace\mathrm{~if~} k=n.
\end{array}
\right.
\end{align} 
\begin{lemma}\label{l:q=lambda}
We have
$q=\bar \lambda_{1:n}$ and $\upsilon=\Phi(\bar x)$. 
\end{lemma}
\begin{proof}
Under the change of coordinate~\eqref{a:cc},
the equation $Q\bar x+d=-\sum_{i=1}^k \bar \lambda_i a_i$ in~\eqref{eq:KKTI} corresponds to: 
$$
(A_{1:n}^{-1})^\top Q(A_{1:n}^{-1}(b_{1:n}-\bar y))+ (A_{1:n}^{-1})^\top d=-\sum_{i=1}^k \bar \lambda_i (A_{1:n}^{-1})^\top a_i,
$$
By~\eqref{eq:p}, the above equation can be simplified to
\begin{equation}\label{eq:pRbary}
p+R\bar y=\sum_{i=1}^k \bar \lambda_i e_i=\begin{pmatrix}
\bar \lambda_{1:k} \\
\bzero
\end{pmatrix}.
\end{equation}
If $k=n$, then $\bar y=\bzero$ and thus $q=p=\bar \lambda_{1:n}$. Otherwise,~\eqref{eq:pRbary} yields
$$
\begin{pmatrix}
p_1 + R_{12} \bar y_{k+1:n} \\
p_2+R_{22} \bar y_{k+1:n}
\end{pmatrix}=\begin{pmatrix}
\bar \lambda_{1:k} \\
\bzero
\end{pmatrix},
$$
so that 
$$
\bar y_{k+1:n}=-R_{22}^{-1}p_2,\enspace p_1-R_{12}R_{22}^{-1}p_2=\bar \lambda_{1:k}.
$$
Hence $q=\bar \lambda_{1:n}$. Note that 
$
\Phi(\bar x)=\bar y^\top R \bar y+2p^\top \bar y +r.
$ If $k=n$, then $\bar y=\bzero$ so we have $\upsilon=r=\Phi(\bar x)$. If $k<n$,
using~\eqref{eq:pRbary} we obtain 
$$\bar y^\top R \bar y+2p^\top \bar y=p^\top \bar y+ y^\top \bar \lambda_{1:n}=p^\top \bar y =p_2 ^\top \bar y_{k+1:n}=-p_2^\top R_{22}^{-1} p_2.$$
Therefore 
$
\Phi(\bar x)=r-p_2^\top R_{22}^{-1} p_2=\upsilon.
$

\end{proof} 
Denote 
\begin{equation}\label{eq:defH}
H= \left\{\begin{array}{ll}
   \begin{pmatrix}
     R_{12} \\ ~ R_{22} \\ ~p_2^\top
    \end{pmatrix}
    R^{-1}_{22}\begin{pmatrix}
     R_{12}^\top &~ R_{22} & ~p_2
    \end{pmatrix},
& \enspace \mathrm{~if~} k<n, \\
\bzero~, & \enspace \mathrm{~if~} k=n.
\end{array}
\right.
\end{equation}
It can be straightforwardly verified that:
\begin{equation}\label{eq:R=D}
\begin{pmatrix}
R & p \\p^\top & r
\end{pmatrix}= \begin{pmatrix}
D & q \\q^\top & \upsilon
\end{pmatrix} + H.
\end{equation}

\subsection{Valid Cut Through Linear Programming}
In this section, we show how to construct valid cut through linear programming.  
For each $i\in [k]$, let
\begin{equation}\label{eq:mustar}
\begin{aligned}
         \mu^*_i:=\min  &  \qquad    \<q,y>+ \beta y_0   \\
        ~ \textrm{s.t.} &\qquad  \<D_i,y>+q_i y_0 =-1\\
&\qquad F y \leq wy_0 \\
& \qquad  y,y_0\geq 0.
    \end{aligned}
\end{equation}
By~\Cref{l:q=lambda}, we have $q\geq 0$.
Hence, $\mu_i^*$ lies in $[0,+\infty]$, with the convention that $\mu_i^*=+\infty$ if~\eqref{eq:mustar} has no feasible solution.

\begin{lemma}\label{l:lambdaistar}
If $\mu_i^*\in (0,+\infty]$, then there exits $v^i\geq 0$, $s^i\geq 0$ and $\delta_i\geq 0$ such that 
\begin{equation}\label{eq:lHs}
\begin{pmatrix}\bzero
& (\mu_i^*)^{-1} q+ D_i+F^\top  v^{i}- s^{i} \\ ((\mu_i^*)^{-1} q+ D_i+F^\top  v^{i}- s^{i})^\top&  2((\mu_i^*)^{-1} \beta+ q_i-w^\top v^{i}-\delta_i)\end{pmatrix} \succeq 0
\end{equation}
\end{lemma}
\begin{proof}
\begin{enumerate}
\item If $\mu_i^*=+\infty$, the feasible region of~\eqref{eq:mustar} is empty, which implies that the linear system
$$\left\{
\begin{array}{l}
\<D_i,y>+q_i y_0 =-1\\
 F y \leq wy_0 \\
  y\geq 0, y_0\geq 0
\end{array}
\right.
$$
has no solution. By Farkas' Lemma, there exists $v^i\geq 0$, $s^i\geq 0$ and $\delta_i \geq 0 $ such that 
$$
D_i+F^\top v^i=s^i,\enspace q_i - w^\top v^i=\delta_i. 
$$
Therefore~\eqref{eq:lHs} holds  because $(\mu_i^*)^{-1}=0$. 
\item If $\mu_i^* \neq +\infty$, then~\eqref{eq:mustar} has an optimal solution. By strong duality, the following dual problem of~\eqref{eq:mustar} 
\begin{equation}\label{eq:mustardual}
\begin{aligned}
         \max_{\lambda,v,s}  &  \qquad    \lambda   \\
        ~ \textrm{s.t.} &\qquad  
        q+\lambda D_i+F^\top v-s = 0 \\
        & \qquad  \beta+\lambda q_i-w^\top v\geq 0\\
        &\qquad  v\geq 0, \enspace s \geq 0
    \end{aligned}
\end{equation}
can reach the optimal value $\mu_i^*$. In other words, there exists $\bar v^i$ and $\bar s^i$ such that 
$$
\begin{pmatrix}\bzero
& q+\mu^*_i D_i+F^\top \bar v^{i}-\bar s^{i} \\ (q+\mu^*_i D_i+F^\top \bar v^{i}-\bar s^{i})^\top&  2(\beta+\mu^*_i q_i-w^\top \bar v^{i})\end{pmatrix} \succeq 0.
$$
Hence~\eqref{eq:lHs} holds with 
$v^i=(\mu_i^*)^{-1}\bar v_i$, $s^i=(\mu_i^*)^{-1}\bar s_i$ and $\delta_i=0$.
\end{enumerate}

\end{proof}
\begin{proposition}\label{prop:Konno} 
If $\mu^*_i \in (0,+\infty]$ for each $i\in [k]$, then~\eqref{eq:validcut-y} holds with 
\begin{align}\label{eq:theta-lp}
\theta_i := \left\{\begin{array}{ll}
    (\mu^*_i)^{-1},
& \enspace\mathrm{~if~} i\leq k \\
0~, & \enspace\mathrm{~if~}  k<i\leq n
\end{array}
\right.
\end{align}
\end{proposition}
\begin{proof}
In view of~\Cref{l:lambdaistar},
there exists $v^1\geq 0,\ldots,v^k \geq 0, s^1\geq 0, \ldots,s^k \geq 0$ and $\delta_1\geq 0,\ldots, \delta_k \geq 0$ such that
$$
 G^i:=\frac{1}{2}\begin{pmatrix}\bzero
& \theta_i q+ D_i+F^\top  v^{i}- s^{i} \\ ( \theta_i q+ D_i+F^\top  v^{i}- s^{i})^\top&  2( \theta_i\beta+ q_i-w^\top v^{i}-\delta_i)\end{pmatrix} \succeq 0,\enspace i\in [k].
$$
We have
$$
\begin{aligned}
& \begin{pmatrix}
y^\top  &~1 
\end{pmatrix} 
 G^i 
\begin{pmatrix}
y \\1 
\end{pmatrix} \\ 
&=  D_i^\top y +q_i +
(Fy-w)^\top   v^i -y^\top s^i   + \theta_i\left( q^\top y+\beta \right)-\delta_i,\enspace \forall y \in \R^n.
\end{aligned}
$$
Multiplying each side by $ y_i$ and summing over $i$ from 1 to $k$ we obtain 
$$
\begin{aligned}
& \begin{pmatrix}
y^\top  &~1 
\end{pmatrix} 
\left(\sum_{i=1}^k  y_i G^i \right)
\begin{pmatrix}
y \\1 
\end{pmatrix} \\ 
&=\sum_{i=1}^k y_i (D_i^\top y +q_i) +
(Fy-w)^\top \left(\sum_{i=1}^k  y_i v^i \right )-y^\top \left(\sum_{i=1}^k y_i s^i \right) \\ &  \qquad +\sum_{i=1}^k y_i \theta_i \left(q^\top y+\beta\right) -\sum_{i=1}^k \delta_i y_i,\enspace \forall y \in \R^n.
\end{aligned}
$$
Recall from the definition of $D$ and $q$ in~\eqref{eq:D} and~\eqref{eq:q} that $D_{i,j}=0$ if $i>k$ or $j>k$ and $q_i=0$ if $i>k$.
It follows that
\begin{equation}\label{eq:yDy}
\begin{aligned}
& y^\top D y+2q^\top y+\beta   
\\ & = \begin{pmatrix}
y^\top  &~1 
\end{pmatrix} 
\left(\sum_{i=1}^k  y_i G^i \right)
\begin{pmatrix}
y \\1 
\end{pmatrix} + (w-Fy)^\top \left(\sum_{i=1}^k y_i v^i \right) +y^\top \left(\sum_{i=1}^k  y_i s^i \right ) \\ & \quad +\left( 1-\sum_{i=1}^k \theta_iy_i \right)\left(q^\top y+\beta\right)+\sum_{i=1}^k \delta_i y_i,\enspace \forall y\in \R^n.
\end{aligned}
\end{equation}
Note that $\sum_{i=1}^k y_i G^i \succeq 0$, $\sum_{i=1}^k  y_i v^i\geq 0$, $\sum_{i=1}^k  y_i s^i\geq 0$,  $q^\top y+\beta \geq 0$ and $\sum_{i=1}^k \delta_i y_i \geq 0$ for all $y\geq 0$. 
Recall that $\bar y_i=0$ for $i\in [k]$. Hence,
\begin{equation}\label{eq:validcut-y-D}
\begin{aligned}
    0 \leq \min_{y\in \R^{n}} \quad &   y^\top D y + 2 q^\top y+\beta   \\
        \textrm{\rm s.t.} \quad &  Fy\leq w\\
        & \displaystyle  y\geq 0\\
        & \theta^\top (y-\bar y) \leq 1.
    \end{aligned}
\end{equation}
In view of~\eqref{eq:R=D} and the fact that $\beta \leq  \Phi(\bar x)-\nu_R =\upsilon-\nu_R$, we deduce~\eqref{eq:validcut-y}.

\end{proof}
\begin{remark}
In the scenario where $\bar{x}$ is a vertex,~\eqref{eq:theta-lp} corresponds precisely to the cut proposed by Konno in~\cite{konno1976cutting} for concave QP (i.e., when $Q \preceq 0$). Note that Konno's cut in its original form was only applicable when $\bar{x}$ is a vertex, which is why it was restricted to the case when $Q \preceq 0$. Interestingly, as we will demonstrate below, the particular choice of the multiplier matrix
$
\begin{pmatrix}
Q \bar{x} + d \\
-\bar{x}^\top Q \bar{x} - d^\top \bar{x} + \beta
\end{pmatrix}
$
which we have selected to ensure the feasibility of the program~\eqref{eq:SDPsearchofc}, is fundamentally rooted in Konno's cut~\eqref{eq:mustar}.

\end{remark}
\begin{proposition}\label{prop:LP}
Assume that $\mu^*_i \in (0,+\infty]$ for each $i\in [k]$ and let $\theta\in \R^n$ be given by~\eqref{eq:theta-lp}.  There exist $\Lambda \geq 0$, $L\geq 0$ and $\ell \geq 0$ such that 
\begin{equation}\label{eq:Rppc3}
\begin{aligned}
\begin{pmatrix}
D & q \\
q^\top &  \beta
\end{pmatrix} &=
\frac{1}{2}\begin{pmatrix}
 \bI_n & \bzero \\ -F & w \\
\bzero^\top & 1
\end{pmatrix}^\top
\begin{pmatrix}
L+L^\top & \Lambda^\top  & \ell    \\
\Lambda & \bzero^{~~} & \bzero \\
\ell^\top & \bzero^\top &  0
\end{pmatrix}
\begin{pmatrix}
 \bI_n & \bzero \\ -F & w \\
\bzero^\top & 1
\end{pmatrix} \\ &\qquad +
\frac{1}{2}\begin{pmatrix}
q \\ \beta
\end{pmatrix}
\begin{pmatrix}
- \theta^\top & 1\end{pmatrix}+\frac{1}{2}\begin{pmatrix}
- \theta \\  1\end{pmatrix}\begin{pmatrix}
q^\top & \beta
\end{pmatrix}.
\end{aligned}
\end{equation}
\end{proposition}
\begin{proof}
Here, we follow the notation and leverage parts of the proof of~\Cref{prop:Konno}.
Given the special structure of $G^i$ and the fact that $G^i \succeq 0$, we have:
$$
\begin{pmatrix}
y^\top  &~1 
\end{pmatrix} 
\left(\sum_{i=1}^k  y_i G^i \right)
\begin{pmatrix}
y \\1 
\end{pmatrix}=\sum_{i=1}^k y_i\left( \theta_i\beta+ q_i-w^\top v^{i}-\delta_i\right),
$$
and so~\eqref{eq:yDy} reduces to
\begin{equation}\label{eq:yDyLP}
\begin{aligned}
& y^\top D y+2q^\top y+\beta
\\ & =  (w-Fy)^\top \left(\sum_{i=1}^k y_i v^i \right) +y^\top \left(\sum_{i=1}^k  y_i s^i \right )  +\left( 1-\sum_{i=1}^k \theta_iy_i \right)\left(q^\top y+\beta\right)\\& \qquad+\sum_{i=1}^k \left( \theta_i\beta+ q_i-w^\top v^{i}\right) y_i,\enspace \forall y\in \R^n.
\end{aligned}
\end{equation}
Then~\eqref{eq:yDyLP} is equivalent to~\eqref{eq:Rppc3} with
$$
\Lambda=\begin{pmatrix}
v^1 &
 \cdots &
 v^k &
 \bzero
\end{pmatrix},\enspace \ell =
\begin{pmatrix}
\theta_1\beta+ q_1-w^\top v^{1}\\
\vdots\\
\theta_k\beta+ q_k-w^\top v^{k}\\
 \bzero
\end{pmatrix},\enspace L= \begin{pmatrix}
s^1 & \cdots & s^k & \bzero
\end{pmatrix}.
$$
\end{proof}

\begin{proposition}\label{prop:LPexistence}
If $\beta>0$ or $\min_{i\in [k]} q_i>0$, then there exist $T\geq 0$ and $\theta\in \R^n$ such that 
\begin{equation}
    \label{eq:LPtheta}
    \begin{aligned}
     & \begin{pmatrix}
D & q \\
q^\top &  \beta
\end{pmatrix} =
\begin{pmatrix}
 \bI & \bzero \\ -F & w \\
\bzero^\top & 1
\end{pmatrix}^\top
T
\begin{pmatrix}
 \bI & \bzero \\ -F & w \\
\bzero^\top & 1
\end{pmatrix} \\ &\qquad\qquad\qquad +
\frac{1}{2}\begin{pmatrix}
q \\ \beta
\end{pmatrix}
\begin{pmatrix}
- \theta^\top & 1+\theta^\top \bar y\end{pmatrix}+\frac{1}{2}\begin{pmatrix}
- \theta \\  1+\theta^\top \bar y\end{pmatrix}\begin{pmatrix}
q^\top & \beta
\end{pmatrix}.
    \end{aligned}
\end{equation} 
\end{proposition}
\begin{proof}
If $\beta>0$ and $\mu_i^*=0$, then the optimal solution of~\eqref{eq:mustar} must satisfy $y_0=0$, which is impossible due to the boundedness of $\cF$. If $\min_{i\in [k]} q_i>0$ and $\mu_i^*=0$, then the optimal solution of~\eqref{eq:mustar} must satisfy $y_1=\cdots=y_k=0$, which implies that $\<D_i,y>=0$ and thus $\<D_i,y>+q_i y_0\geq 0$. This contradicts with $\<D_i,y>+q_i y_0=-1$. 
Therefore, if $\beta>0$ or $\min_{i\in [k]} q_i>0$, we have $\mu^*_i\in (0,+\infty]$ for each $i\in [k]$. Then it suffices to apply~\Cref{prop:LP} and use the fact that $\theta$ given by~\eqref{eq:theta-lp} satisfies $\theta^\top \bar y=0$.

\end{proof}

\subsection{Proof of~\Cref{thm:STc}}\label{subsec:proofmainthm}
\begin{proof}
If $\bar x$ is nondegenerate, then up to a permutation of the linear constraints, there exist  $\bar \lambda$ such that~\eqref{eq:KKTI} holds with $$
\min_{i\in [k]} \bar \lambda_i>0.
$$
From~\Cref{l:q=lambda} we know that the nondegeneracy of $\bar x$ implies $\min_{i\in [k]} q_i>0$. 
Let $S\in \cS_+^n$ such that $S_{1:k,1:k}\succ 0$ and $S_{k+1:n,k+1:n}=\bzero$. Using~\eqref{eq:R=D} we write
$$
\begin{pmatrix}
R & p \\
p^\top &  r- \nu_{_R}
\end{pmatrix} = H+  \begin{pmatrix}
S & \bzero \\
\bzero^\top &  \upsilon-\nu_{_R}-\beta
\end{pmatrix}+ \begin{pmatrix}
D-S & q \\
q^\top &  \beta
\end{pmatrix}.
$$
Denote 
$$
\hat H:=
H+\begin{pmatrix}
S & \bzero \\
\bzero^\top & \upsilon-\nu_{_R}-\beta
\end{pmatrix}.
$$
By~\Cref{l:q=lambda}, $\upsilon=\Phi(\bar x)$ and thus $\upsilon-\nu_{_R}-\beta\geq 0$. Therefore
 $\hat H\succeq 0$.

Note that~\Cref{prop:LPexistence} does not impose any condition on the matrix $D$, except that 
$
D_{i,j}=0
$ when $i>k$ or $j>k$. Hence, the same conclusion holds for $D-S$. Therefore, if $\min_{i\in [k]} q_i>0$ or $\beta>0$, 
there exists  $T\geq 0$ and $\theta\in \R^n$ such that
\begin{equation}\label{eq:Rppc34}
\begin{aligned}
&\begin{pmatrix}
R & p \\
p^\top &  r- \nu_{_R}
\end{pmatrix} = \hat H+
\begin{pmatrix}
 \bI & \bzero \\ -F & w \\
\bzero^\top & 1
\end{pmatrix}^\top
T
\begin{pmatrix}
 \bI & \bzero \\ -F & w \\
\bzero^\top & 1
\end{pmatrix} \\ &\qquad +
\frac{1}{2}\begin{pmatrix}
q \\ \beta
\end{pmatrix}
\begin{pmatrix}
- \theta^\top  & 1+\theta^\top \bar y \end{pmatrix}+\frac{1}{2}\begin{pmatrix}
- \theta \\  1+\theta^\top \bar y\end{pmatrix}\begin{pmatrix}
q^\top & \beta
\end{pmatrix}.
\end{aligned}
\end{equation}
Note that
$$
\begin{pmatrix}
Q & d \\
d^\top &  -\nu_{_R}
\end{pmatrix} =
\begin{pmatrix}
-A_{1:n} & b_{1:n}\\
\bzero^\top & 1
\end{pmatrix}^\top
\begin{pmatrix}
R & p \\
p^\top & \quad r-\nu_{_R}
\end{pmatrix} \begin{pmatrix}
-A_{1:n} & b_{1:n}\\
\bzero^\top & 1
\end{pmatrix}.
$$
Therefore,
$$
\begin{aligned}
\begin{pmatrix}
Q & d \\
d^\top &  -\nu_{_R}
\end{pmatrix} &
\overset{\eqref{eq:Rppc34}}{=}
\begin{pmatrix}
-A_{1:n} & b_{1:n}\\
\bzero^\top & 1
\end{pmatrix}^\top \Bigg [  \hat H+
\begin{pmatrix}
 \bI & \bzero \\ -F & w \\
\bzero^\top & 1
\end{pmatrix}^\top
T
\begin{pmatrix}
 \bI & \bzero \\ -F & w \\
\bzero^\top & 1
\end{pmatrix}  \\ &   +
\frac{1}{2}\begin{pmatrix}
q \\ \beta
\end{pmatrix}
\begin{pmatrix}
- \theta^\top & 1+\theta^\top \bar y\end{pmatrix}+\frac{1}{2}\begin{pmatrix}
-\theta \\  1+\theta^\top \bar y \end{pmatrix}\begin{pmatrix}
q^\top & \beta
\end{pmatrix}\Bigg ]
\begin{pmatrix}
-A_{1:n} & b_{1:n}\\
\bzero^\top & 1
\end{pmatrix}.
\end{aligned}
$$
In view of~\eqref{eq:w}, we have
$$
\begin{pmatrix}
\bI & \bzero\\
-F & w  \\
\bzero^\top & 1
\end{pmatrix} \begin{pmatrix}
-A_{1:n} & b_{1:n}\\
\bzero^\top & 1
\end{pmatrix}=\begin{pmatrix}
-A_{1:n} & b_{1:n}\\
-A_{n+1:m} & b_{n+1:m} \\
\bzero^\top &1
\end{pmatrix} =\begin{pmatrix}
-A &\enspace  b \\
\bzero^\top & \enspace 1
\end{pmatrix},
$$
and from~\eqref{eq:c}, we have
$$
\begin{pmatrix}
- \theta^\top & 1+\theta^\top \bar y
\end{pmatrix}
\begin{pmatrix}
-A_{1:n} & b_{1:n}\\
\bzero^\top & 1
\end{pmatrix}=\begin{pmatrix}
 \theta^\top A_{1:n} & 1- \theta^\top b_{1:n}+\theta^\top \bar y
\end{pmatrix}=\begin{pmatrix}
\displaystyle -c^\top  & 1+ c^\top  \bar x
\end{pmatrix}.
$$
By~
\Cref{l:q=lambda}, we have
$$
Q \bar x+d =-A^\top \bar \lambda= -A_{1:n}^\top q,
$$
and 
$$
(A \bar x-b)^\top \bar \lambda=(A_{1:n} \bar x-b_{1:n})^\top q=0.
$$
Therefore,
$$
\begin{pmatrix}
q^\top & \beta
\end{pmatrix}
\begin{pmatrix}
-A_{1:n} & b_{1:n}\\
\bzero^\top & 1
\end{pmatrix}=\begin{pmatrix}
-q^\top A_{1:n} & \enspace q^\top b_{1:n}+\beta
\end{pmatrix}=\begin{pmatrix} \left( Q\bar x+d\right)^\top &\enspace  -\bar x^\top Q \bar x- d^\top\bar x+\beta \end{pmatrix}.
$$
Thus, for $c=-A^\top_{1:n}\theta$ it holds that
\begin{align*}
\begin{pmatrix}
Q & d\\
d^\top & -\nu_{_R}
\end{pmatrix}=&
\begin{pmatrix}
-A_{1:n} & b_{1:n}\\
\bzero^\top & 1
\end{pmatrix}^\top \hat H
\begin{pmatrix}
-A_{1:n} & b_{1:n}\\
\bzero^\top & 1
\end{pmatrix}+
\begin{pmatrix}
-A & b\\
\bzero^\top & 1
\end{pmatrix}^\top T \begin{pmatrix}
-A & b\\
\bzero^\top & 1
\end{pmatrix}  \\ &\qquad +\frac{1}{2}\begin{pmatrix}
Q\bar x + d\\
-\bar x^\top Q\bar x - d^\top \bar x +\beta
\end{pmatrix}\begin{pmatrix}
\displaystyle -c^\top  & 1+ c^\top  \bar x
\end{pmatrix}\\& \qquad +\frac{1}{2}\begin{pmatrix}
\displaystyle -c  \\ 1+ c^\top  \bar x
\end{pmatrix}\begin{pmatrix}
(Q\bar x + d)^\top & -\bar x^\top Q\bar x - d^\top \bar x +\beta
\end{pmatrix}.
\end{align*} 
We proved the feasibility of~\eqref{eq:SDPsearchofc} if $\min_{i\in [k]} q_i>0$ or $\beta>0$.   

Suppose that $\Phi(\bar x)=\upsilon>\nu_{_R}$ and  $0<\beta<\upsilon-\nu_{_R}$. 
To prove the strict feasibility of~\eqref{eq:SDPsearchofc}, it suffices to show that $\hat H\succ 0$.
Since $\hat H\succeq 0$, it remains to  argue that 
\begin{equation}\label{eq:yHy0}
\begin{pmatrix}
y^\top & y_0
\end{pmatrix} \hat H \begin{pmatrix}
y \\ y_0
\end{pmatrix} =0
\end{equation}
if and only if $y=\bzero$ and $y_0=0$. Given $S_{1:k,1:k}\succ 0$, $S_{k+1:n,k+1:n}=\bzero$ and $\upsilon-\nu_{_R}-\beta>0$,~\eqref{eq:yHy0} implies that $y_{1:k}=\bzero$ and $y_0=0$. Consequently, $$\begin{pmatrix}
y^\top & y_0
\end{pmatrix} \hat H \begin{pmatrix}
y \\ y_0
\end{pmatrix}=
\begin{pmatrix}
\bzero \\
y_{k+1:n}\\ 0
\end{pmatrix}^\top H \begin{pmatrix}
\bzero \\
y_{k+1:n}\\ 0
\end{pmatrix} \overset{\eqref{eq:defH}}{=}y_{k+1:n}^\top R_{22} y_{k+1:n}.
$$
Since $R_{22}\succ 0$, it follows that~\eqref{eq:yHy0} holds only if $y=\bzero$ and $y_0=0$.

\end{proof}

\section{Proof of~\Cref{thm:gmcfinite}}

The following two lemmas are basic observations from the steps of~\Cref{alg:generalized downhill}.
\begin{lemma}\label{l:nonincreasing} Let  $\{x_k: k=1,2,\ldots\}$ be the iterates of~\Cref{alg:generalized downhill}. Let any $k\geq 0$.
\begin{enumerate}
\item If  $Q|_{\cH_{I_{x_k}}}\succ 0$, then $$
\Phi(x_{k+1})\leq \Phi^*_{I_{x_k}} \leq \Phi(x_k).
$$
\item If $Q|_{\cH_{I_{x_k}}}\nsucc 0$, then
$$
\Phi(x_{k+1})\leq \Phi(x_k)\enspace \mathrm{and~~}
I_{x_{k+1}}\supsetneqq I_{x_k}.
$$
\end{enumerate}
\end{lemma}

    \begin{lemma}\label{l:sefs}
    Let $x_k$  satisfy $Q|_{\cH_{I_{x_k}}}\succ 0$.
 If the while loop does not break  at  iteration $k$, then 
 $\Phi(x_{k+1})< \Phi_{I_{x_k}}^*$.
    \end{lemma}

\begin{proof}[proof of~\Cref{thm:gmcfinite}]
Suppose the while loop never terminates. If $Q|_{\cH_{I_{x_k}}} \nsucc 0$, then by item 2 of \Cref{l:nonincreasing}, the index set $I_{x_{k+1}}$ strictly contains $I_{x_k}$. Thus, there must be infinitely many iterations $k$ for which  $Q|_{\cH_{I_{x_k}}} \succ 0$.
Consequently, there must exist two iterations $j > k$ such that
\begin{equation}\label{eq:isdfe}
\begin{aligned}
  Q|_{\cH_{I_{x_k}}} \succ 0,\enspace
   Q|_{\cH_{I_{x_j}}} \succ 0,\enspace
  I_{x_j} = I_{x_k}.
\end{aligned}
\end{equation}
By item 1 of \Cref{l:nonincreasing} and the non-increasing property of the sequence $\{\Phi(x_k)\}_{k \geq 0}$, we have
\begin{align}\label{a:powe}
\Phi^*_{I_{x_j}} \leq \Phi(x_j) \leq \Phi(x_{k+1}) \leq \Phi^*_{I_{x_k}}.
\end{align}

Since $I_{x_j} = I_{x_k}$, it follows that $\Phi^*_{I_{x_j}} = \Phi^*_{I_{x_k}}$, so all the inequalities in~\eqref{a:powe} are in fact equalities. Applying \Cref{l:sefs} then leads to a contradiction.

\end{proof}

\section{Auxiliary Algorithms}\label{sec-appendix:auxialgo}
\begin{algorithm}[H]
    \caption{\texttt{DNN\_Lower\_Bound}}
    \begin{algorithmic}[1]
    \Require $Q \in \cS^n$, $d \in \R^n$, $A \in \R^{m\times n}$, $b \in \R^m$.
    \State Solve~\eqref{eq:SDP-lb}-\eqref{eq:SDP-lb-d} to obtain  approximate optimal solutions ($\lambda^*$, $S^*\succeq 0$, $T^*\geq 0$) and $U^*= \begin{pmatrix}
Y^* & y^* \\
(y^*)^\top & 1
\end{pmatrix}$ with $y^*\in\cF$.
    \State $\Delta\leftarrow  \begin{pmatrix}
            Q & d \\
            d^\top & -\lambda^*
        \end{pmatrix}
        - S^* 
        - \begin{pmatrix}
            -A & b \\
            \bzero^\top & 1
        \end{pmatrix}^\top
        T^*
        \begin{pmatrix}
            -A & b \\
            \bzero^\top & 1
        \end{pmatrix}$.
        \State $\delta\leftarrow \min(0,\lambda_{\min}(\Delta))$.
        \State Compute $r>0$ such that $\cF\subset \{x: \|x\|\leq r\}$.
    \Ensure Lower bound $\lambda^*+\delta(1+r^2)$, a vector $y^*\in \cF$.
    \end{algorithmic}
    \label{alg:dnnbound}
\end{algorithm}

\begin{algorithm}[H]
    \caption{\texttt{DNN\_Cut}}
    \begin{algorithmic}[1]
    \Require $Q \in \cS^n$, $d \in \R^n$, $A \in \R^{m\times n}$, $b \in \R^m$, $\bar x\in \cF$, $\bar z\in \cF$, $\nu_R\in \R$, $\nu\in \R$.
    \State Solve~\eqref{eq:SDPsearchofc}--\eqref{eq:SDPsearchofcdual} to obtain  approximate optimal solutions ($S^*\succeq 0$, $T^*\geq 0$, $c^*\in \R^n$).
    \State Compute the error matrix $\Delta$ such that~\eqref{eq:SDPeq} holds.
        \State $\delta\leftarrow \min(0,\lambda_{\min}(\Delta))$.
        \State Compute $r>0$ such that $\cF\subset \{x: \|x\|\leq r\}$.
        \State $w\leftarrow  \nu_R+\delta(1+r^2)$
    \If { $w < \nu $}
    \State $A_c \leftarrow \begin{pmatrix}
    A\\
    (c^*)^\top
    \end{pmatrix}$, $b_c \leftarrow \begin{pmatrix}
    b\\
    (c^*)^\top \bar x+1
    \end{pmatrix}$.
    \State  $w \leftarrow \texttt{DNN\_Lower\_Bound($Q,d, A_c,b_c$).}$ \Comment{See~\Cref{alg:dnnbound}.}
    \EndIf
    \Ensure Cut $c^*\in \R^n$, and the associated cut lower bound $w$. 
    \end{algorithmic}
    \label{alg:dnncut}
\end{algorithm}

\bibliographystyle{spmpsci}      
\bibliography{references.bib}   

\begin{thebibliography}{10}
\providecommand{\url}[1]{{#1}}
\providecommand{\urlprefix}{URL }
\expandafter\ifx\csname urlstyle\endcsname\relax
  \providecommand{\doi}[1]{DOI~\discretionary{}{}{}#1}\else
  \providecommand{\doi}{DOI~\discretionary{}{}{}\begingroup
  \urlstyle{rm}\Url}\fi

\bibitem{BonamiOktay18}
Bonami, P., G{\"u}nl{\"u}k, O., Linderoth, J.: Globally solving nonconvex
  quadratic programming problems with box constraints via integer programming
  methods.
\newblock Mathematical Programming Computation \textbf{10}(3), 333--382 (2018)

\bibitem{Bonami2019}
Bonami, P., Lodi, A., Schweiger, J., Tramontani, A.: Solving quadratic
  programming by cutting planes.
\newblock SIAM Journal on Optimization \textbf{29}(2), 1076--1105 (2019)

\bibitem{Burer2012}
Burer, S.: Copositive programming.
\newblock In: Handbook on semidefinite, conic and polynomial optimization, pp.
  201--218. Springer US, Boston, MA (2012)

\bibitem{Burer2009a}
Burer, S., Letchford, A.N.: On nonconvex quadratic programming with box
  constraints.
\newblock SIAM Journal on Optimization \textbf{20}(2), 1073--1089 (2009).
\newblock \doi{10.1137/080729529}.
\newblock \urlprefix\url{https://doi.org/10.1137/080729529}

\bibitem{burer2008finite}
Burer, S., Vandenbussche, D.: A finite branch-and-bound algorithm for nonconvex
  quadratic programming via semidefinite relaxations.
\newblock Mathematical Programming \textbf{113}(2), 259--282 (2008)

\bibitem{chen2012globally}
Chen, J., Burer, S.: Globally solving nonconvex quadratic programming problems
  via completely positive programming.
\newblock Mathematical Programming Computation \textbf{4}(1), 33--52 (2012)

\bibitem{Coffrin2015}
Coffrin, C., Hijazi, H.L., Van~Hentenryck, P.: Strengthening convex relaxations
  with bound tightening for power network optimization.
\newblock In: G.~Pesant (ed.) Principles and Practice of Constraint
  Programming, pp. 39--57. Springer International Publishing (2015)

\bibitem{Cottle1992}
Cottle, R.W., Pang, J.S., Stone, R.E.: The linear complementarity problem.
\newblock Society for Industrial and Applied Mathematics (1992)

\bibitem{EdirisingheJeong2017}
Edirisinghe, C., Jeong, J.: Tight bounds on indefinite separable
  singly-constrained quadratic programs in linear-time.
\newblock Mathematical Programming \textbf{164}, 193--227 (2017)

\bibitem{Gleixner17}
Gleixner, A.M., Berthold, T., M{\"u}ller, B., Weltge, S.: Three enhancements
  for optimization-based bound tightening.
\newblock Journal of Global Optimization \textbf{67}(4), 731--757 (2017)

\bibitem{horst2013handbook}
Horst, R., Pardalos, P.M.: Handbook of global optimization, vol.~2.
\newblock Springer Science \& Business Media (2013)

\bibitem{konno1976cutting}
Konno, H.: A cutting plane algorithm for solving bilinear programs.
\newblock Mathematical Programming \textbf{11}(1), 14--27 (1976)

\bibitem{konno1976maximization}
Konno, H.: Maximization of a convex quadratic function under linear
  constraints.
\newblock Mathematical programming \textbf{11}(1), 117--127 (1976)

\bibitem{Lee2005}
Lee, G.M., Tam, N.N., Yen, N.D.: Necessary and sufficient optimality conditions
  for quadratic programs, pp. 45--63.
\newblock Springer US (2005)

\bibitem{LiDengLuWuDaiWang2023}
Li, S., Deng, Z., Lu, C., Wu, J., Dai, J., Wang, Q.: An efficient global
  algorithm for indefinite separable quadratic knapsack problems with box
  constraints.
\newblock Computational Optimization and Applications \textbf{86}, 241--273
  (2023)

\bibitem{LiuzziLocatelliPiccialli2022}
Liuzzi, G., Locatelli, M., Piccialli, V.: A computational study on {QP}
  problems with general linear constraints.
\newblock Optimization Letters \textbf{16}, 1633--1647 (2022)

\bibitem{LiuzziLocatelliPiccialliRass2021}
Liuzzi, G., Locatelli, M., Piccialli, V., Rass, S.: Computing mixed strategies
  equilibria in presence of switching costs by the solution of nonconvex {QP}
  problems.
\newblock Computational Optimization and Applications \textbf{79}, 561--599
  (2021)

\bibitem{LocatelliPiccialliSudoso2025}
Locatelli, M., Piccialli, V., Sudoso, A.: Fix and bound: an efficient approach
  for solving large-scale quadratic programming problems with box constraints.
\newblock Mathematical Programming Computation \textbf{17}, 231--263 (2025)

\bibitem{LuisContesse1980}
Luis, C.: Une caract\'{e}risation compl\`{e}te des minima locaux en
  programmation quadratique.
\newblock Numerische Mathematik \textbf{34}, 315--332 (1980)

\bibitem{Mangasarian1980}
Mangasarian, O.: Locally unique solutions of quadratic programs, linear and
  nonlinear complementarity problems.
\newblock Mathematical programming \textbf{19}(1), 200--212 (1980)

\bibitem{mccormick1976computability}
McCormick, G.P.: Computability of global solutions to factorable nonconvex
  programs: Part {I}—convex underestimating problems.
\newblock Mathematical programming \textbf{10}(1), 147--175 (1976)

\bibitem{MotzkinStraus1965}
Motzkin, T., Straus, E.: Maxima for graphs and a new proof of a theorem of
  {Tur\'{a}n}.
\newblock Canadian Journal of Mathematics \textbf{17}, 533--540 (1965)

\bibitem{nohra2021}
Nohra, C.J., Raghunathan, A.U., Sahinidis, N.: Spectral relaxations and
  branching strategies for global optimization of mixed-integer quadratic
  programs.
\newblock SIAM Journal on Optimization \textbf{31}(1), 142--171 (2021)

\bibitem{Nowak1999}
Nowak, I.: A new semidefinite programming bound for indefinite quadratic forms
  over a simplex.
\newblock Journal of Global Optimization \textbf{14}, 357--364 (1999)

\bibitem{PardalosGlickRosen1987}
Pardalos, P., Glick, J., Rosen, J.: Global minimization of indefinite quadratic
  problems.
\newblock Computing \textbf{39}, 281--291 (1987)

\bibitem{pardalos1991quadratic}
Pardalos, P.M., Vavasis, S.A.: Quadratic programming with one negative
  eigenvalue is {NP-hard}.
\newblock Journal of Global optimization \textbf{1}, 15--22 (1991)

\bibitem{Puranik17}
Puranik, Y., Sahinidis, N.V.: Bounds tightening based on optimality conditions
  for nonconvex box-constrained optimization.
\newblock Journal of Global Optimization \textbf{67}(1), 59--77 (2017)

\bibitem{QuZengLou2025}
Qu, Z., Zeng, T., Lou, Y.: Globally solving concave quadratic programs via
  doubly nonnegative relaxation.
\newblock Mathematical Programming Computation  (2025)

\bibitem{SheraliFraticelli2002}
Sherali, H., Fraticelli, B.: Enhancing {RLT} relaxations via a new class of
  semidefinite cuts.
\newblock Journal of Global Optimization \textbf{22}, 233--261 (2002)

\bibitem{SheraliAdams1999}
Sherali, H.D., Adams, W.P.: A reformulation-linearization technique for solving
  discrete and continuous nonconvex problems, vol.~31.
\newblock Kluwer Academic Publishers (1999)

\bibitem{Shor1987}
Shor, N.: Quadratic optimization problems.
\newblock Soviet Journal of Computer and Systems Sciences \textbf{25}, 1--11
  (1987)

\bibitem{Sundar2023}
Sundar, K., Nagarajan, H., Misra, S., Lu, M., Coffrin, C., Bent, R.:
  Optimization-based bound tightening using a strengthened qc-relaxation of the
  optimal power flow problem.
\newblock In: 2023 62nd IEEE Conference on Decision and Control (CDC), pp.
  4598--4605 (2023).
\newblock \doi{10.1109/CDC49753.2023.10384116}

\bibitem{Tao1997}
Tao, P.D., An, L.T.H.: Convex analysis approach to {D. C.} programming: Theory,
  algorithms and applications.
\newblock Acta Mathematica Vietnamica \textbf{22}(1), 289--355 (1997)

\bibitem{TawarmalaniSahinidis2004}
Tawarmalani, M., Sahinidis, N.: Global optimization of mixed-integer nonlinear
  programs: A theoretical and computational study.
\newblock Mathematical Programming \textbf{Ser. A 99}, 563--591 (2004)

\bibitem{tuy1964concave}
Tuy, H.: Concave programming under linear constraints.
\newblock Soviet Math. \textbf{5}, 1437--1440 (1964)

\bibitem{VandenbusscheNemhauser2005a}
Vandenbussche, D., Nemhauser, G.: A branch-and-cut algorithm for nonconvex
  quadratic programs with box constraints.
\newblock Mathematical Programming \textbf{102}, 559--575 (2005)

\bibitem{VandenbusscheNemhauser2005}
Vandenbussche, D., Nemhauser, G.: A polyhedral study of nonconvex quadratic
  programs with box constraints.
\newblock Mathematical Programming \textbf{102}, 531--557 (2005)

\bibitem{Vavasis1992a}
Vavasis, S.: Approximation algorithms for indefinite quadratic programming.
\newblock Mathematical Programming \textbf{57}, 279--311 (1992)

\bibitem{Vavasis1992}
Vavasis, S.: Local minima for indefinite quadratic knapsack problems.
\newblock Mathematical Programming \textbf{54}, 127--153 (1992)

\bibitem{Xia2020}
Xia, W., Vera, J.C., Zuluaga, L.F.: Globally solving nonconvex quadratic
  programs via linear integer programming techniques.
\newblock INFORMS Journal on Computing \textbf{32}(1), 40--56 (2020)

\bibitem{ZhangHanPang2024}
Zhang, X., Han, S., Pang, J.S.: Improving the solution of indefinite quadratic
  programs and linear programs with complementarity constraints by a
  progressive {MIP} method.
\newblock arxiv preprint arXiv:2409.09964  (2024)

\end{thebibliography}

\end{document}